\documentclass[10pt]{amsart}
\input{preamble.tex}
\usetikzlibrary{patterns,snakes}

\newcommand{\kirby}{\omega} 
\newcommand{\Kirby}{\Omega} 
\newcommand{\JW}{P} 
\newcommand{\symobj}{\JW} 
\newcommand{\bn}{\mathrm{BN}} 
\newcommand{\BN}{\mathcal{BN}} 
\newcommand{\ABN}{\Ann\BN} 
\newcommand{\PBN}{\mathcal{PBN}} 
\newcommand{\Ann}{\mathcal{A}}
\newcommand{\TL}{\mathrm{TL}} 
\newcommand{\dTL}{\mathrm{d}\TL} 
\newcommand{\KdTL}{\mathrm{Kd}\TL} 
\newcommand{\tpc}{\mathrm{M}\dTL} 
\newcommand{\ctpc}{\overline{\mathrm{M}\dTL}} 
\newcommand{\Mat}{\mathrm{Mat}} 
\DeclareMathOperator*{\colim}{colim}
\newcommand{\Web}{\mathrm{Web}} 
\newcommand{\dWeb}{\mathrm{d}\Web} 
\newcommand{\Pol}{\mathbf{Pol}} 
\newcommand{\xs}{z} 
\newcommand{\sx}{\tilde{x}} 
\newcommand{\ux}{x} 
\newcommand{\gen}{c} 

\newcommand{\incl}{\iota}
\newcommand{\Incl}{\mathrm{Incl}}

\newcommand{\Vect}{\mathrm{Vect}} 
\newcommand{\VectZ}{\mathrm{Vect}_{0}^{\Gamma}} 
 
\newcommand{\CSZ}{\CS_{0}} 
\newcommand{\cdTL}{\overline{\dTL}}
\newcommand{\cABN}{\overline{\ABN}}
\newcommand{\cPBN}{\overline{\PBN}}

\definecolor{kirb}{RGB}{254,112,198}
\tikzset{kirbystyle/.style={ultra thick,kirb},
  }
\newcommand{\kirbcolword}{pink }

\usepackage{etoolbox}
\tracingpatches
\makeatletter
\patchcmd{\@setref}{\bfseries ??}{\bfseries\color{red} OWE A COFFEE/BEER}{}{}
\makeatother

\begin{document}

\author{Matthew Hogancamp}
\address{Department of Mathematics, Northeastern University, 360 Huntington Ave, Boston,
MA 02115, USA}
\email{m.hogancamp@northeastern.edu}

\author{David~E.~V.~Rose}
\address{Department of Mathematics, University of North Carolina, 
Phillips Hall, CB \#3250, UNC-CH, 
Chapel Hill, NC 27599-3250, USA
\href{https://davidev.web.unc.edu/}{davidev.web.unc.edu}}
\email{davidrose@unc.edu}

\author{Paul Wedrich}
\address{Fachbereich Mathematik, Universit\"at Hamburg, 
Bundesstra{\ss}e 55, 
20146 Hamburg, Germany
\href{http://paul.wedrich.at}{paul.wedrich.at}}
\email{paul.wedrich@uni-hamburg.de}

\title{A Kirby color for Khovanov homology}

\begin{abstract} 
We construct a Kirby color in the setting of Khovanov homology:
an ind-object of the annular Bar-Natan category that
is equipped with a natural handle slide isomorphism.
Using functoriality and cabling properties of Khovanov homology, 
we define a Kirby-colored Khovanov homology that is invariant under
the handle slide Kirby move,
up to isomorphism.
Via the Manolescu--Neithalath 2-handle formula, 
Kirby-colored Khovanov homology agrees with the $\glnn{2}$ skein lasagna module, 
hence is an invariant of $4$-dimensional $2$-handlebodies.
\end{abstract}

\maketitle

\setcounter{tocdepth}{1}

\tableofcontents

\section{Introduction}
In the context of quantum link invariants (such as the Jones polynomial) a
\emph{Kirby color}\footnote{This name seems to have come into use sometime
before 2001 \cite{MR2009998}, although the concept is older. Another common name
in the tensor category literature is ``(virtual) regular object'', see e.g.
\cite[p.270]{MR3242743}} is a certain linear combination of cabling patterns
yielding a framed link invariant that is invariant under the second
(handle slide) Kirby move \cite{MR0467753}. The resulting ``Kirby-colored''
quantum invariant of a link $\mathcal{L} \subset S^3$ can then be regarded as an
invariant of the 4-manifold obtained by attaching $2$-handles to the $4$-ball
$B^4$ along the link $\mathcal{L} \subset \partial B^4$ \cite{MR1185809}.
Perhaps the most famous examples of this process are the
Witten--Reshetikhin--Turaev invariants \cite{MR990772,MR1091619}, though we note
these are special in that they admit a further renormalization that is
invariant under the first (blow-up) Kirby move\footnote{Being invariant under
both Kirby moves, the Witten--Reshetikhin--Turaev invariants are thus
3-manifold invariants, depending only on the boundary of the aforementioned
4-dimensional 2-handlebody.}.

Khovanov homology \cite{Kho} is a bigraded homology theory for links $\cal{L}
\subset S^3$, which categorifies the Jones polynomial. One of the
longest-standing problems regarding Khovanov homology is whether (and if so,
how) it extends to an invariant of $3$- or $4$-manifolds. The goal of this paper
is to propose a solution to this problem that proceeds by developing a
\emph{Kirby color for Khovanov homology}.

\subsection{The Kirby color}
\label{ss:intro kirby}
To describe our work with slightly more precision, recall that in the
(pre-categorified) context of framed link invariants associated with a linear
ribbon category, the set of cabling patterns forms an algebra. For the Jones
polynomial, this is the Kauffman bracket skein algebra of the thickened annulus.
Analogously, in the categorified context the cabling patterns form a monoidal
category: for Khovanov homology this role is played\footnote{See
Remark~\ref{rem:dgtrace} for a more sophisticated choice, which will not be
necessary for the purpose of this paper.} by the Bar-Natan category of the
thickened annulus, which hereafter will be denoted $\ABN$. For this
introduction, it suffices to note that objects of $\ABN$ are embedded
$1$-manifolds in the annulus and morphisms are formal linear combinations of
(dotted) cobordisms embedded in the thickened annulus, modulo local relations (from
\cite[\S  11.2]{BN2}). Moreover, $\ABN$ is $\Z$-graded and linear, with monoidal
structure given by inserting one annulus into the interior of another. See \S
\ref{ss:dTL} for the precise definitions. 

We let $\gen$ denote the object of $\ABN$ which is a single essential circle, 
so $\gen^n$ denotes $n$ concentric essential circles. 
The symmetric group $\mathfrak{S}_n$ acts on $\gen^n$, 
thus the symmetric power $\Sym^n(\gen)$ exists in
the Karoubi completion $\Kar(\ABN)$ 
as the image of the symmetrizing idempotent. 

\begin{definition}
Let $\kirby:=\kirby_0\oplus \kirby_1$, where the summands are the colimits
\begin{align*}
\kirby_0&:=\colim\Big(\Sym^0(\gen)\rightarrow q^{-2}\Sym^2(\gen) \rightarrow  q^{-4}\Sym^4(\gen) \rightarrow\cdots  \Big)\\
\kirby_1&:=\colim\Big(q^{-1}\Sym^{1}(\gen)\rightarrow q^{-3}\Sym^3(\gen) \rightarrow  q^{-5}\Sym^5(\gen) \rightarrow\cdots  \Big),
\end{align*}
regarded as objects of an appropriate completion $\overline{\ABN}$ (see Convention~\ref{conv}). 
Here, the maps 
are given by dotted annulus cobordisms, 
and the grading shifts $q^{-n}$ ensure that these maps are degree zero.
\end{definition}

The object $\kirby \in \overline{\ABN}$ is the titular Kirby color for Khovanov homology. 
In contrast to the pre-categorified situation, it is an \emph{object} of a monoidal
category rather than an \emph{element} of an algebra, 
and it manifestly does not require working at a root of unity. 
In particular, the Kirby color $\kirby$ for
Khovanov homology does not categorify any Kirby element for 
the $\slnn{2}$ Witten--Reshetikhin--Turaev invariants!   
\medskip

Nevertheless, the Kirby color $\kirby$ \emph{behaves} like a categorified Kirby
element~\cite{MR2251160} in the sense that it leads to link invariants that are
invariant under handle slide. 
In the categorical setting, handle slide invariance is an additional structure, rather than a property.
This structure is best phrased in the annular setting,
namely using the relative Bar-Natan category of the annulus $\PBN$, 
wherein objects are annular curves with boundary (see \S \ref{ss:punctured BN}). 
This category is a module category for $\ABN$ and contains two special objects
\[
L := 
\begin{tikzpicture}[anchorbase,scale=1]
\node[gray] at (0,0) {$\bullet$};
\draw[gray] (0,0) circle (.5);
\draw[very thick] (0,.5) to [out=270,in=90] (-.25,0) to [out=270,in=90] (0,-.5);
\end{tikzpicture}
\; , \quad
R :=
\begin{tikzpicture}[anchorbase,scale=1]
\node[gray] at (0,0) {$\bullet$};
\draw[gray] (0,0) circle (.5);
\draw[very thick] (0,.5) to [out=270,in=90] (.25,0) to [out=270,in=90] (0,-.5);
\end{tikzpicture} \, .
\]

\begin{theoremA}[Handle slide] 
	\hypertarget{thm:A}
The Kirby color $\kirby$ is equipped with the following handle slide structure:
	\begin{itemize}
		\item (Lemma~\ref{lem:handleslide}) There exists a distinguished
		isomorphism $\kirby\bullet L \cong \kirby\bullet R$ in an appropriate
		completion $\overline{\PBN}$, that we call the \emph{elementary handle
		slide}. Graphically, this may be depicted as in the first isomorphism of \eqref{eq:slide}. 	
		\begin{equation}
			\label{eq:slide}
		\begin{tikzpicture}[anchorbase,scale=.625]
			\node[gray] at (0,0) {$\bullet$};
			\draw[gray] (0,0) circle (2);
			\draw[very thick] (0,2) to [out=270,in=90] (-1.5,0) to [out=270,in=90] (0,-2);
			\draw[very thick, double] (0,0) circle (.75);
			\filldraw[white] (.375,.25) rectangle (1.125,-.25); 
			\draw[very thick] (.375,.25) rectangle (1.125,-.25);
			\node at (.75,0) {$\kirby$};
			\end{tikzpicture}
			\cong
			\begin{tikzpicture}[anchorbase,rotate=180,scale=.625]
			\node[gray] at (0,0) {$\bullet$};
			\draw[gray] (0,0) circle (2);
			\draw[very thick] (0,2) to [out=270,in=90] (-1.5,0) to [out=270,in=90] (0,-2);
			\draw[very thick, double] (0,0) circle (.75);
			\filldraw[white] (-.375,.25) rectangle (-1.125,-.25); 
			\draw[very thick] (-.375,.25) rectangle (-1.125,-.25);
			\node at (-.75,0) {$\kirby$};
			\end{tikzpicture}
			\quad, \qquad
					\begin{tikzpicture}[anchorbase,scale=.625]
					\node[gray] at (0,0) {$\bullet$};
					\draw[gray] (0,0) circle (2);
					\draw[very thick] (.25,2) to [out=270,in=90] (-1,0) to [out=270,in=90] (.25,-2);
					\draw[very thick] (-.25,2) to [out=270,in=90] (-1.5,0) to [out=270,in=90] (-.25,-2);
					\draw[very thick, double] (0,0) circle (.5);
					\filldraw[white] (.25,.25) rectangle (.75,-.25); 
					\draw[very thick] (.25,.25) rectangle (.75,-.25);
					\filldraw[white] (-1.75,.325) rectangle (-.75,-.325); 
					\draw[very thick] (-1.75,.325) rectangle (-.75,-.325); 
					\node at (.5,0) {$\kirby$};
					\node at (-1.25,0) {$D$};
					\end{tikzpicture} 
				\cong
				\begin{tikzpicture}[anchorbase,rotate=180,scale=.625]
					\node[gray] at (0,0) {$\bullet$};
					\draw[gray] (0,0) circle (2);
					\draw[very thick] (.25,2) to [out=270,in=90] (-1,0) to [out=270,in=90] (.25,-2);
					\draw[very thick] (-.25,2) to [out=270,in=90] (-1.5,0) to [out=270,in=90] (-.25,-2);
					\draw[very thick, double] (0,0) circle (.5);
					\filldraw[white] (.25,.25) rectangle (.75,-.25); 
					\draw[very thick] (.25,.25) rectangle (.75,-.25);
					\filldraw[white] (-1.75,.325) rectangle (-.75,-.325); 
					\draw[very thick] (-1.75,.325) rectangle (-.75,-.325); 
					\node at (.5,0) {$\kirby$};
					\node at (-1.25,0) {$D$};
					\end{tikzpicture}
		\end{equation}

		\item (Lemma~\ref{lem:naturalhandleslide}) Compositions of elementary
		handle slides assemble into a collection of handle slide isomorphisms,
		which are natural with respect to cobordisms involving the
		``sliding strands'' (illustrated as $D$ in the second isomorphism of \eqref{eq:slide}).
	\end{itemize}
\end{theoremA}

 In Theorem~\ref{thm:handleslide}, we further show that the Kirby color $\kirby$,
together with its handle slide isomorphisms, constitutes an object of the
Drinfeld center of the completed relative Bar-Natan category $\cPBN$ of the
annulus, considered as a bimodule for the monoidal Bar-Natan bicategory $\BN$
associated to the rectangle. 
The proof of Theorem \hyperlink{thm:A}{A} uses a generators and relations 
presentation of a certain subcategory of $\PBN$ 
that we provide in Theorem~\ref{thm:MdTL}.

\subsection{Manifold invariants and TQFT context}
\label{ss:intro invariants}
In order to port our results on handle slide invariance from the annular setting
to links in $S^3$, 
we need a well-defined notion of cabling in Khovanov homology. 
The following is a straightforward consequence of the functoriality of Khovanov homology:

\begin{theoremB}[Cabling in Khovanov homology]
	\hypertarget{introthm:cabling}
Let $\mathcal{L}=\mathcal{K}_1\cup \cdots \cup \mathcal{K}_r$ be an
$r$-component framed oriented link in $S^3$.
There is a functor
\[
\Kh_{\mathcal{L}} \colon \ABN^{\times r} \rightarrow \Vect^{\Z\times \Z}
\]
sending $(\gen^{k_1},\cdots, \gen^{k_r})$ 
to the Khovanov homology of the cable $\mathcal{L}^{\underline{k}} := \mathcal{K}_1^{k_1} \cup \cdots \cup \mathcal{K}_r^{k_r}$.  
\end{theoremB}

Here, $\mathcal{K}_i^{k_i}$ denotes the $k_i$-fold parallel cable of the component $\mathcal{K}_i$
and $\Vect^{\Z\times \Z}$ denotes the category of bigraded vector spaces.
The functor in Theorem \hyperlink{introthm:cabling}{B} requires certain choices that
fix the (original) well-known sign ambiguity in the functoriality of Khovanov homology;
this is discussed further in \S \ref{ss:cabling} where we re-state and prove this result 
as Theorem \ref{thm:cable functor}.
Since $\Vect^{\Z\times \Z}$ is closed under all of the relevant operations used to define $\overline{\ABN}$
(grading shifts, direct sums and summands, filtered colimits), 
cabling extends to a functor
\[
\Kh_\mathcal{L} \colon(\overline{\ABN})^{\times r} \rightarrow \Vect^{\Z\times \Z} \, .
\]

\begin{definition}[Kirby-colored Khovanov homology]
	\label{def:KKh}
	Let $\mathcal{L}$ be a framed oriented link in $S^3$ with a decomposition
	into sublinks $\mathcal{L}= \mathcal{L}_1\cup \mathcal{L}_2$. Set
\[
\Kh(\mathcal{L}_1\cup \mathcal{L}_2^\kirby) :=	
\Kh_{\mathcal{L}_1\cup
\mathcal{L}_2}(\gen,\ldots,\gen, \kirby,\ldots,\kirby)
\]
in which all the components of $\mathcal{L}_1$ carry the label $\gen$, and all
the components of $\mathcal{L}_2$ carry the label $\kirby$.
\end{definition}

\begin{remark}
In \cite[\S  7.5]{GLW} Grigsby--Licata--Wehrli consider a different 
colimit of Khovanov homologies of cables of a knot $\mathcal{K}$, 
which appears to be unrelated to our Kirby colored Khovanov homology.
\end{remark}

Theorems \hyperlink{thm:A}{A} and \hyperlink{introthm:cabling}{B}
suggest that $\Kh(\mathcal{L}_1\cup \mathcal{L}_2^\kirby)$ 
may be a diffeomorphism invariant of the pair $(B^4(\mathcal{L}_2), \mathcal{L}_1)$, where $B^4(\mathcal{L}_2)$ is
the $4$-dimensional 2-handlebody obtained by attaching $2$-handles to $B^4$
along the components of $\mathcal{L}_2$, 
and $\mathcal{L}_1$ is regarded as a link in the boundary $3$-manifold
$S^3(\mathcal{L}_2):=\partial B^4(\mathcal{L}_2)$.
To establish this directly, one would further need to incorporate $1$- and $3$-handles 
in a way that they satisfy an appropriate cancellation property with $\kirby$.

The resulting $4$-manifold invariant would a priori be valued in \emph{isomorphism classes} of
bigraded vector spaces, rather than valued in bigraded vector spaces. 
An upgrade to the latter would in particular require not only isomorphisms associated to handle
slide Kirby moves, but a \emph{coherent family} thereof. 
For this, one would need a classification of movie moves for Kirby moves,
in analogy with the Carter--Saito movie moves for isotopic link cobordisms \cite{CSbook}.
We are not aware of such a classification.

In the absence of both of the above, we instead establish
$\Kh(\mathcal{L}_1\cup \mathcal{L}_2^\kirby)$ as a bigraded vector space-valued 
invariant of $(B^4(\mathcal{L}_2), \mathcal{L}_1)$
by comparison with the $\mathfrak{gl}_N$ skein lasagna module 
$4$-manifold invariants from \cite{2019arXiv190712194M}, 
whose invariance is manifest.

\begin{theoremC}
	\hypertarget{thm:intro lasagna} 
Let $\mathcal{L}=\mathcal{L}_1\cup
\mathcal{L}_2$ be a framed oriented link in $S^3$.   Decorate all the components
of $\mathcal{L}_1$ with $\gen$, and all the components of $\mathcal{L}_2$ with
$\kirby$.  The following bigraded vector spaces are isomorphic:

\begin{enumerate}
\item the Kirby-colored Khovanov homology $\Kh(\mathcal{L}_1\cup \mathcal{L}_2^\kirby)$,
\item the \emph{cabled} $N=2$ Khovanov--Rozansky homology of $\mathcal{L}_1\cup \mathcal{L}_2^\kirby$
as defined by Manolescu--Neithalath \cite{2020arXiv200908520M} for
split unions $\mathcal{L}_1\sqcup \mathcal{L}_2^\kirby$ and extended to the general case in
\cite{2206.04616}, and
\item the $N=2$ skein lasagna module (degree zero blob homology) of $(B^4(\mathcal{L}_2),
\mathcal{L}_1)$ from
\cite{2019arXiv190712194M}.
\end{enumerate}
As a consequence, the Kirby-colored Khovanov homology $\Kh(\mathcal{L}_1\cup
\mathcal{L}_2^\kirby)$ is an invariant of the pair $(B^4(\mathcal{L}_2),
\mathcal{L}_1)$, valued in bigraded vector spaces.
\end{theoremC}

The isomorphisms $(2) \cong (3)$ have been established in
\cite{2020arXiv200908520M, 2206.04616}.
In \S \ref{sec:KirbyKh}, 
we verify $(1) \cong (2)$ by showing that cabling with the Kirby color $\kirby$
implements the Manolescu--Neithalath $2$-handle formula 
from \cite{2020arXiv200908520M} that defines $(2)$.

\begin{remark}
The $\slnn{2}$-version of the skein lasagna module of $(B^4(\mathcal{L}_2), \mathcal{L}_1)$ is graded
by the relative second homology group $H_2(B^4(\mathcal{L}_2), \mathcal{L}_1; \Z/2)$.  
The isomorphism from Theorem \hyperlink{thm:intro lasagna}{C} identifies this grading with the direct sum 
decomposition of $\Kh_{\mathcal{L}_1\cup \mathcal{L}_2}(\gen,\ldots,\gen,\kirby,\ldots,\kirby)$ inherited from 
$\kirby = \kirby_0\oplus \kirby_1$.
\end{remark}

\subsection{Diagrammatics}
\label{ss:intro diagrams}
A theorem of Russell \cite{Russell} implies that the monoidal category $\ABN$ admits a diagrammatic
presentation in terms of the \emph{dotted Temperley-Lieb category} $\dTL$ 
(see our \S \ref{ss:dTL} for a definition). 
The latter is a non-semisimple $\Z$-graded monoidal category that contains the 
familiar Temperley-Lieb category as its degree zero subcategory.
A further goal in the present paper is to extend the graphical calculus for $\dTL$
to give a presentation of the
monoidal category obtained from $\ABN$ by adjoining the Kirby objects
$\kirby_0$ and $\kirby_1$. 
This is accomplished in \S \ref{sec:diag}.   

One interesting feature of this extended calculus is the necessity of infinite
sums of diagrams.  The following gives the flavor of the sort of relations one
encounters:
\[
\begin{tikzpicture}[anchorbase,xscale=.75]
\draw[kirbystyle]  (0,0)node[below=-2pt]{\scriptsize $[0]$} to (0,1)node[above=-2pt]{\scriptsize $[0]$};
\draw[kirbystyle]  (1,0)node[below=-2pt]{\scriptsize $[0]$} to (1,1)node[above=-2pt]{\scriptsize $[0]$};
\end{tikzpicture}
= \
\sum_{n\geq 0}
\frac{(-1)^{\genfrac(){0pt}{2}{n}{2}}}{n!}
\begin{tikzpicture}[anchorbase,scale=1]
	\draw[kirbystyle]  (0,0)node[below=-2pt]{\scriptsize $[0]$}  to (0,1.25)node[above=-2pt]{\scriptsize $[0]$};
	\draw[kirbystyle]  (1,0)node[below=-2pt]{\scriptsize $[0]$} to (0,.5);
	\draw[kirbystyle] (0,1) to[out=-90,in=180] (.5,.75) to[out=0,in=-90] (1,1)
		to (1,1.25)node[above=-2pt]{\scriptsize $[0]$};
	\draw[fill=white] (.5,.75) circle (.15);
	\node at (.5,1) {\scriptsize $n$};
	\node at (.5,.25) {\small $\bullet$};
	\node at (.65,.4) {\scriptsize $n$};
\end{tikzpicture} \, .
\]
Here, the $[0]$-label indicates the Kirby object $\kirby_0$, 
and this equation gives a decomposition of the identity morphism 
of $\kirby_0 \otimes \kirby_0$ into mutually orthogonal idempotents. 
This relation reflects a certain \emph{quasi-idempotence} property of the Kirby object, 
namely that $\kirby_0\otimes \kirby_0 \cong \bigoplus_{n\geq 0} q^{-2n} \kirby_{0}$ 
(with similar statements for $\kirby_i\otimes \kirby_j$). 
See Corollary \ref{cor:kirby squared} for details.

\subsection{The Kirby color as representing planar evaluation}
\label{ss:intro representing}

For the remainder of this introduction,  
we phrase all of our results in terms of the diagrammatic category $\dTL$, 
instead of $\ABN$ (to which it is equivalent).  
If $\mathcal{U}$ is the 0-framed unknot,  the associated the cabling functor
$\Kh_{\mathcal{U}} \colon {\dTL}\rightarrow \Vect^{\Z\times \Z}$ satisfies
\[
\Kh_{\mathcal{U}}(\gen^{n})\cong q^{-n}\K[x_1,\ldots,x_n]/(x_i^2=0) \, .
\]
We refer to $\Kh_{\mathcal{U}}$ as the \emph{polynomial representation} of $\dTL$ (or its completion $\cdTL$) 
and we denote it by $\Pol(-) = \Kh_{\mathcal{U}}(-)$. 
Since the cohomological grading of $\Kh_\mathcal{U}(X)$ is
trivial for all $X\in \dTL$, 
we typically regard $\Pol(-)$ as taking values in the category $\Vect^{\Z}$ of singly-graded vector spaces. 
The category-theoretic role of the Kirby color $\kirby$ is that it represents the functor
\[
\Pol^\ast \colon \cdTL^{\op} \rightarrow \Vect^{\Z} \, , \quad X\mapsto \Pol(X)^\ast
\]
where on the right-hand side $\Pol(X)^\ast$ denotes the graded dual of the graded vector space $\Pol(X)$. 
Specifically, in \S \ref{ss:poly redux} we restate and prove the following.

\begin{theoremD}
There is an isomorphism $\Hom_{\cdTL}(-,\kirby) \cong \Pol^\ast(-)$ of functors $\cdTL^{\op} \rightarrow \Vect^{\Z}$. 
Additionally, we have isomorphisms
\[
\kirby\otimes X \cong \Pol(X)\otimes \kirby
\]
natural in $X\in \cdTL$.
\end{theoremD}

\begin{remark} We would like to emphasize that the surprising feature in
	representing the functor $\Pol^*$ is not that a representing object exists
	in \emph{some appropriate completion} of $\dTL$. Rather, it is that the
	representing object $\kirby$ can be described very explicitly as the,
	arguably, simplest non-trivial directed system (i.e.~ind-object) over
	$\dTL$. Furthermore, the explicit description is essential for computations
	of Kirby-colored Khovanov homology (Definition~\ref{def:KKh}) and for the
	diagrammatics described in \S \ref{ss:intro diagrams}.
	\end{remark}

The functor $\Pol \colon \dTL\rightarrow \Vect^{\Z}$ also has an algebraic description via 
$\slnn{2}$ representation theory. 
As mentioned above, 
the usual (undotted) Temperley-Lieb category $\TL$ 
(at circle-value $2$)
can be regarded as the subcategory of degree zero morphisms in $\dTL$.   
As is well-known, there is a fully faithful monoidal functor 
$\TL\rightarrow \Rep(\slnn{2})$ that sends $\gen \mapsto V$, 
the defining $2$-dimensional representation of $\slnn{2}$.  
If we forget the action of the Chevalley generators $E,F\in \slnn{2}$ but
remember the weight grading, we obtain a functor $\TL\rightarrow \Vect^{\Z}$
that coincides with $\Pol|_{\TL}$. Now $\Pol \colon \dTL\rightarrow
\Vect^{\Z}$ may be thought of as the extension of this functor to $\dTL$, defined
by sending the ``dot'' endomorphism of $\gen$ to the action of $E\in \slnn{2}$
on $V$.

\begin{remark}
Since the Karoubi (idempotent) completion of $\TL$ \emph{is} semisimple, it is
not hard to see that the restricted functor $\Pol^\ast|_{\overline{\TL}} \colon
\overline{\TL}^{\op} \rightarrow \Vect^\Z$ is representable by the object
\[
\bigoplus_{n\geq 0} \Pol \big( \Sym^n(c) \big)^\ast \otimes \Sym^n(c)
\]
where here $\Sym^n(c)$ is the object in $\Kar(\TL)$ 
corresponding to the simple finite-dimensional $\slnn{2}$-module $\Sym^n(V)$.
Since $\Pol \big( \Sym^n(c) \big)$ has graded dimension equal to
the quantum integer $[n+1]$, this formula is reminiscent of the familiar formula
for the Kirby element in $\slnn{2}$ Witten--Reshetikhin--Turaev theory.

However, it is the category $\dTL$ (non-semisimple even after Karoubi completing) 
that naturally arises in Khovanov homology, 
and the representing object for $\Pol^\ast\colon \cdTL \to \Vect^\Z$ need not be (and indeed is not) 
isomorphic to the representing object for its restriction to $\overline{\TL}$. 
This elucidates why our Kirby color does not categorify the familiar 
Kirby element from $\slnn{2}$ Witten--Reshetikhin--Turaev theory.  
\end{remark}

\subsection*{Conventions}
All results in this paper hold over any field $\K$ of characteristic zero. We
let $\k$ denote an arbitrary field. Knots and links are always framed and oriented.

\subsection*{Acknowledgements}
We thank 
Christian Blanchet,
Elijah Bodish, 
Ben Elias, 
Eugene Gorsky, 
Jiuzu Hong, 
David Reutter, 
Kevin Walker,
and Hans Wenzl for useful discussions and correspondence.
Special thanks to Ciprian Manolescu and Ikshu Neithalath,
whose 2-handle formula motivated the construction of the Kirby color.

\subsection*{Funding}

D.R. was partially supported by NSF CAREER grant DMS-2144463 and Simons
Collaboration Grant 523992. P.W. acknowledges support by the Deutsche
Forschungsgemeinschaft (DFG, German Research Foundation) under Germany's
Excellence Strategy - EXC 2121 ``Quantum Universe'' - 390833306.

\section{Categorical Background}

In this section we recall category-theoretic constructions
that will be used throughout. 
In particular, we review graded, linear categories 
and discuss (co)limits and various completions in this setting.

\subsection{Graded linear categories}

We begin with some basic notions, mostly for the purpose of establishing notation.
Let $\Gamma$ be an abelian group and $\k$ a field. 
A \emph{$\Gamma$-graded vector space} is a
$\Gamma$-indexed collection of $\k$-vector spaces $(V_i)_{i\in \Gamma}$. 
Given a pair of $\Gamma$-graded vector spaces $V,W$, 
we let $\HOM(V,W)$ denote the $\Gamma$-graded vector space
which is given in degree $i\in \Gamma$ by
\[
\HOM^i(V,W) := \prod_{j\in \Gamma}\Hom_\k(V_j,W_{i+j}) \, .
\]
We let $\VectZ$ denote the category with objects $\Gamma$-graded vector spaces
and with morphisms
\[
\Hom_{\VectZ}(V,W) := \HOM^0(V,W) \, .
\]
The category $\VectZ$ is symmetric monoidal, with tensor product given by
\[
\big( (V_i)_{i\in \Gamma} \otimes (W_j)_{j\in \Gamma} \big)_k := \bigoplus_{i+j=k} V_i \otimes W_j \, .
\]
Further, the tensor-hom adjunction
\[
\Hom_{\VectZ}(U \otimes V , W) \cong 
\Hom_{\VectZ}(U , \HOM(V , W)) \, .
\]
holds in $\VectZ$. 
In other words, 
$\VectZ$ is \emph{closed} with \emph{internal hom} given by $\HOM$. 
Since $\VectZ$ is further (co)complete (i.e.~has all (co)limits, see below), 
it is natural to consider categories enriched thereover.

\begin{definition}
\label{def:CSZ}
A \emph{$\Gamma$-graded $\k$-linear category} is a category $\CS$ that is enriched over $\VectZ$. 
If $\CS$ is a $\Gamma$-graded $\k$-linear category, 
then we let $\CSZ$ (its \emph{degree zero subcategory}) denote the non-full subcategory 
with the same objects as $\CS$ and with morphisms the degree zero morphisms in $\CS$.
\end{definition}

\begin{remark}
	\label{rem:isodegzero}
An \emph{isomorphism} in a $\Gamma$-graded $\k$-linear category $\CS$ will 
always mean an element of $\Hom_{\CS_0}(X,Y)$
with a two-sided inverse (necessarily also degree zero).
\end{remark}

The prototypical example of a $\Gamma$-graded $\k$-linear category is the category 
$\Vect^\Gamma$ with objects $\Gamma$-graded vector spaces
and with morphisms
\[
\Hom_{\Vect^\Gamma}(V,W) := \HOM(V,W) \, .
\]
The degree zero subcategory of $\Vect^\Gamma$ is simply $\VectZ$.

We will need to adjoin the images of idempotents to $\Gamma$-graded $\k$-linear categories. 
This is accomplished with the following operation.

\begin{definition}[Karoubi completion]\label{def:karoubi}
Let $\CS$ be a $\Gamma$-graded $\k$-linear category.
The \emph{(graded) Karoubi envelope} $\Kar(\CS)$ is the category 
whose objects are pairs $(X,e)$ in which $X\in \CS$ and $e$ is a 
degree zero idempotent endomorphism of $X$,
i.e.~$e\in \End_\CS^0(X)$ and $e^2=e$. 
Morphisms in $\Kar(\CS)$ are defined by
\[
\Hom_{\Kar(\CS)}\Big((X,e),(X',e')\Big) := \Big\{f\in \Hom_\CS(X,X') \: |\: f = e'\circ f \circ e\Big\}.
\]
\end{definition}

In other words, 
the graded Karoubi envelope of $\CS$ has the same objects as the usual Karoubi envelope of $\CS_0$.
Often, we will use abbreviated notation for objects in $\Kar(\CS)$, 
denoting the object $(X,e)$ simply by the idempotent $e$ itself.

\subsection{Shifts and biproducts}

We now discuss (in turn) shifts and biproducts in the graded linear setting.
Our eventual aim is Definition \ref{def:GAcompletion}, which gives the 
completion of a $\Gamma$-graded $\k$-linear categories with respect to these notions.

Given $j\in \Gamma$, 
let $q^j$ denote the endofunctor of $\VectZ$ (or $\Vect^\Gamma$) that shifts 
grading up by $j$.
In other words, if $V=(V_i)_{i\in \Gamma}$ is a graded vector space, 
then $q^j V$ is the graded vector space with $(q^jV)_i = V_{i-j}$. 
The following extends this notion to graded linear categories.

\begin{definition}
	\label{def:shifts}
Let $X,Y\in \CS$ be objects of a $\Gamma$-graded $\k$-linear category and let $k \in \Gamma$.
We denote $Y\cong q^k X$ 
if $Y$ is equipped with an invertible morphism $\sigma\in \Hom_{\CS}^k(X,Y)$.
Such objects are called \emph{shifts of X}.
\end{definition}

Note that the morphism $\sigma$ in Definition \ref{def:shifts} is (typically) \emph{not} an isomorphism in $\CS$
(see Remark \ref{rem:isodegzero}).
Shifts of $X$ are characterized by the following universal property, whose proof is a straightforward exercise.
\begin{lemma}
The following are equivalent.
\begin{enumerate}
\item $Y\cong q^kX$.
\item $\Hom_\CS(Y,-) \cong q^{-k}\Hom_\CS(X,-)$ as functors $\CS\rightarrow \Vect^\Gamma$.
\item $\Hom_\CS(-,Y) \cong q^{k}\Hom_\CS(-,X)$ as functors $\CS\rightarrow \Vect^\Gamma$. \qed
\end{enumerate} 
\end{lemma}

We next discuss \emph{biproducts}. We emphasize to the reader that we will use
the standard notation $\coprod$ and $\prod$ for general category-theoretic coproducts and products
(reviewed in the graded linear setting below), 
and reserve the symbol $\bigoplus$ for biproducts, which we now recall.
Note that a $\Gamma$-graded $\k$-linear category can be viewed 
as enriched over pointed sets: 
in each $\Hom$-space there exists a zero-morphism $0_{Y,X}\in \Hom_\CS(X,Y)$ 
satisfying
\begin{itemize}
\item $f\circ 0_{Y,X} = 0_{Z,X}$ for all $f\in \Hom_{\CS}(Y,Z)$, and
\item $0_{Y,X}\circ g = 0_{Y,Z}$ for all $g\in \Hom_{\CS}(Z,X)$.
\end{itemize}

\begin{definition}
	\label{def:biproducts}
Let $\CS$ be a $\Gamma$-graded $\k$-linear category.
Given objects $Y,X_i\in \CS$ where $i$ ranges over a (possibly infinite) set $I$, 
we say that $Y$ is the \emph{biproduct} of the $X_i$, and denote this $Y\cong \bigoplus_i X_i$, 
provided $Y$ is equipped with morphisms $\sigma_i\in \Hom_{\CS_0}(X_i,Y)$ 
and $\pi_i\in \Hom_{\CS_0}(Y,X_i)$ such that
\begin{enumerate}[(i)]
\item $\pi_i\circ \sigma_j = \id_{Y_i}$ if $i=j$ and $0_{Y_i, Y_j}$ if $i\neq j$,
\item the collection $\{\sigma_i\}_{i \in I}$ gives $Y$ the structure of the coproduct $\coprod X_i$, and
\item the collection $\{\pi_i\}_{i \in I}$ gives $Y$ the structure of the product $\prod_i X_i$.
\end{enumerate}
\end{definition}

\begin{remark}
More-generally, the above definition could be used in any category enriched over pointed sets.
Our condition (i) is readily seen to be equivalent to the condition sometimes seen in the literature, 
namely that the ``canonical comparison map'' $\coprod_{i \in I} X_i \to \prod_{i \in I} X_i$ 
is an isomorphism.
\end{remark}

In the (graded) linear setting, finite biproducts are easily recognized, 
as they are characterized ``equationally.''
The following is standard, see e.g.~\cite[\S  VIII.2, Theorem 2]{MacLane}.

\begin{lemma}[Finite biproduct recognition]
\label{lemma:finite biproduct recognition} 
Let $\CS$ be a $\Gamma$-graded $\k$-linear category and let 
$\sigma_i\in \Hom_{\CS_0}(X_i,Y)$ and $\pi_i\in \Hom_{\CS_0}(Y,X_i)$ 
for $i=1,\ldots,k$.
The maps $\{\pi_i\}_{i=1}^k$ and $\{\sigma_i\}_{i=1}^k$ 
exhibit $Y$ as the biproduct of the $X_i$ if and only if:
\begin{itemize}
\item  $\pi_i\circ \sigma_j = \delta_{ij} \id_{X_i}$, and
\item  $\id_Y = \sum_{i=1}^k \sigma_i\circ \pi_i$.\qed
\end{itemize} 
\end{lemma}

\begin{remark}
	\label{rem:absolute}
An immediate consequence of Lemma \ref{lemma:finite biproduct recognition} is that 
any linear functor $\CS\rightarrow \DS$ sends finite biproducts to finite biproducts
(i.e.~biproducts are absolute colimits in categories enriched in abelian groups).
Shifts are similarly preserved by all graded linear functors.
\end{remark}

In \S\ref{ss:biprod recognition} below we given an extension of Lemma \ref{lemma:finite biproduct recognition}
that gives a similar ``equational'' characterization of certain infinite biproducts.
Note that infinite biproducts do sometimes exist ``in nature'', 
e.g.~the direct sum of countably many graded vector spaces is a biproduct 
provided the direct sum is locally finite.

It is possible to formally adjoin both grading shifts and finite biproducts 
to a given $\Gamma$-graded $\k$-linear category that may lack them. 

\begin{definition}
	\label{def:GAcompletion}
If $\CS$ is a $\Gamma$-graded $\k$-linear category, then the
\emph{$\Gamma$-additive completion of $\CS$} 
is the category $\Mat(\CS)$ wherein objects are
formal expressions 
$\bigoplus_{i \in I} q^{k_i} X_i$ where $I$ is a finite set, 
$k_i \in \Gamma$ and $X_i \in \CS$.
Morphisms are given by matrices, i.e.
\[
\Hom_{\Mat(\CS)}\left(
\bigoplus_{i \in I} q^{k_i} X_i , \bigoplus_{j \in J} q^{\ell_j} Y_j \right) 
:= \bigoplus_{(j,i) \in J\times I} q^{\ell_j - k_i}\Hom_\CS(X_i,Y_j) \, .
\]
\end{definition}

\begin{remark}
If $\CS$ already admits grading shifts and finite biproducts, 
then the canonical inclusion $\CS \hookrightarrow \Mat(\CS)$ is an equivalence.
\end{remark}

We conclude this section with a discussion of \emph{copowers}, 
which are a common generalization of both shifts and certain finite biproducts.

\begin{definition}
	\label{def:copowers}
Let $\CS$ be a $\Gamma$-graded $\k$-linear category. 
Given $X\in \CS$ and $V\in \Vect^\Gamma$, the \emph{copower}
$V\otimes X$ is the object (unique up to canonical isomorphism when it exists) characterized by 
isomorphisms
\[
\Hom_\CS(V\otimes X,Y)\cong \Hom_{\Vect^\Gamma}(V,\Hom_\CS(X,Y))
\]
natural in $Y$.
\end{definition}

\begin{example}
If $V=q^k \k$, then the copower $V\otimes X$ coincides with the shift $q^kX$.
If $V = \coprod_i q^{k_i}\k$ (not necessarily finite), 
then the copower $V\otimes X$ satisfies
\[
V\otimes X \cong \coprod_i q^{k_i}X \, .
\]
\end{example}

\subsection{Colimits and ind-completion}
	\label{ss:Ind}

We assume the reader is familiar with the standard notions of 
limits and colimits in ordinary category theory. 
Although there is a rich theory of (co)limits in the setting of enriched categories, 
our purposes require only slight elaboration on the usual definitions. 
(Essentially, we require that all diagrams and structure maps have degree zero.)
We now recall the necessary background, 
restricting to the case of colimits in $\VectZ$-enriched categories for efficiency of exposition.
 
\begin{definition}
	\label{def:colimit}
Let $\CS$ be a $\Gamma$-graded $\k$-linear category.  
A \emph{diagram} in $\CS$ is a small category $\IS$ and a functor 
$\a \colon \IS \to \CS$ that factors through the canonical inclusion $\CS_0\rightarrow \CS$. 
The \emph{colimit} of $\a$ in $\CS$ is the object $\colim_{i \in \IS}\alpha(i)$ 
(unique up to canonical isomorphism if it exists) characterized by
\begin{equation}
	\label{eq:colimit}
\Hom_{\CS}(\colim_{i \in \IS}\alpha(i), - ) \cong \lim_{i \in \IS} \Hom_{\CS}(\alpha(i),-)
\end{equation}
as functors $\CS \to \Vect^\Gamma$. 
\end{definition}

We will refer to a diagram $\a \colon \IS \to \CS$ as 
\emph{$\IS$-indexed} and will call $\IS$ the \emph{indexing category}.

\begin{remark}
	\label{rem:colimit}
We draw the reader's attention to some details in Definition \ref{def:colimit}.
For fixed $Y\in \CS$, the limit $\lim_{i \in \IS} \Hom_{\CS}(\alpha(i),Y)$ can be calculated in $\VectZ$ 
since the structure maps have degree zero by hypothesis on $\a$.  
This limit can be described as an explicit graded subspace of $\prod_{i\in \IS} \Hom_\CS(\a(i),Y)$.
Letting $Y$ vary yields the functor $\CS \to \Vect^\Gamma$ appearing in the right-hand side of \eqref{eq:colimit}. 
By Remark \ref{rem:isodegzero}, the structure maps in $\Hom_\CS(\a(i),\colim_{i \in \IS}\alpha(i))$
(obtained by plugging the identity morphism of $\colim_{i \in \IS}\alpha(i)$
into the left-hand side of \eqref{eq:colimit})
are necessarily degree zero.
\end{remark}

\begin{remark} 
Colimits as in Definition \ref{def:colimit} are called \emph{conical} colimits in the language of enriched category theory, 
to distinguish them from the more general notion of \emph{weighted} colimits. 
All explicit references to colimits in this paper are conical, so we drop the adjective. 
The notion of (conical) limit is dual.
\end{remark}

\begin{example}
Let $\IS$ be a (small) discrete category, 
i.e.~a category wherein the only morphisms are the identity morphisms.
An $\IS$-indexed diagram in $\CS$ then consists of a set of objects $\{X_i\}_{i \in \IS}$ (given by $X_i = \alpha(i)$) 
and we have $\colim_{i \in \IS} \a(i) = \coprod_{i \in \IS} X_i$ and $\lim_{i \in \IS} \a(i) = \prod_{i \in \IS} X_i$.
As noted in Remark \ref{rem:colimit}, the inclusion $\rho_i \colon X_i \to \coprod_{i \in \IS} X_i$
and projection $\pi_i \colon \prod_{i \in \IS} X_i \to  X_i$ are degree zero.
\end{example}

We will make particular use of \emph{filtered} colimits, 
and now review all we will need concerning them.

\begin{definition}
	\label{def:filtered}
A small category $\IS$ is \emph{filtered} provided the following conditions hold:
\begin{itemize}
\item given objects $i,j \in \IS$, there exists an object $k$ and morphisms 
$i \to k$ and $j \to k$, and
\item given parallel morphisms $f \colon i \to j$ and $g \colon i \to j$ in $\IS$, 
there exists a morphism $h \colon j \to k$ so that $hf = hg$.
\end{itemize}
If $\CS$ is a $\Gamma$-graded $\k$-linear category, then a 
\emph{directed system} in $\CS$ is a $\IS$-indexed diagram in $\CS$
where $\IS$ is a (small) filtered category;
a \emph{filtered colimit} in $\CS$ is the colimit of such a directed system.
\end{definition}

It is often useful to compute (filtered) colimits by restricting to 
certain subcategories of the indexing category. 
More generally, suppose that $F\colon \IS \to \JS$ is a functor 
and that $\a \colon \JS \to \CS$ is a $\JS$-indexed diagram in $\CS$. 
The structure maps $\alpha(F(i)) \to \colim_{j \in \JS} \alpha(j)$ then induce 
a canonical comparison map
\begin{equation}
	\label{eq:comp}
\colim_{i \in \IS} \alpha(F(i)) \to \colim_{j \in \JS} \alpha(j) \, .
\end{equation}

\begin{definition}
	\label{def:final}
Let $F\colon \IS \to \JS$ be a functor between (indexing) categories.
The functor $F$ is called \emph{final} if the canonical comparison map 
\eqref{eq:comp} is an isomorphism for all diagrams $\a \colon \JS \to \CS$
such that both colimits involved exist.
\end{definition}

We will also refer to a full subcategory $\IS\subset \JS$ as \emph{final}
if the inclusion functor is final. The following is a standard criterion for establishing that 
a functor is final (it is a special case of the equivalent characterization of final 
from \cite[\S  IX.3]{MacLane}).

\begin{lemma}
	\label{lem:final}
If $F \colon \IS \to \JS$ is a functor such that:
\begin{itemize}
\item for each $j\in \JS$ there is a morphism $j\rightarrow F(i)$ for some $i\in \IS$, and
\item given $i \in \IS$ and parallel morphisms $f_1,f_2\colon j \to F(i)$ in $\JS$
there exists a morphism $g \colon i \rightarrow i'$ in $\IS$ such that $F(g)\circ f_1 = F(g)\circ f_2$,
\end{itemize}
then $F$ is final. \qed
\end{lemma}

\begin{example}
The poset $(\Z,\leq)$ is a small filtered category. 
A $\Z$-indexed diagram in a $\Gamma$-graded $\k$-linear category $\CS$ 
can be understood as a sequence of objects and degree zero morphisms
\[
\cdots \xrightarrow{f_{-1}} A_0 \xrightarrow{f_1} A_1 \xrightarrow{f_2} \cdots
\]
in $\CS$. The subcategory $(\N,\leq) \subset (\Z,\leq)$ is final.
\end{example}

Recall that a category $\CS$ is called \emph{cocomplete} 
when colimits of all diagrams in $\CS$ exists therein. 
When this is not the case, it is always possible to embed $\CS$ into its \emph{presheaf category} $\CS^\wedge$
(of contravariant functors to $\Vect^\Gamma$, when $\CS$ is $\Gamma$-graded $\k$-linear) which is cocomplete; 
however, the latter is often intractable. 
If one is primarily interested in filtered colimits, 
the following provides an intermediary means to formally adjoin such colimits to
a $\Gamma$-graded $\k$-linear category.

\begin{definition}
	\label{def:Ind}
Let $\CS$ be a $\Gamma$-graded $\k$-linear category. 
The \emph{ind-completion} of $\CS$ is the category $\Ind(\CS)$ 
with objects the directed systems $\alpha \colon \IS \to \CS$. 
Given directed systems $\alpha \colon \IS \to \CS$ and  
$\beta \colon \JS \to \CS$, the morphism space is given by
\[
\Hom_{\Ind(\CS)}\left(\alpha,\beta\right) 
:= \lim_{i \in \IS} \colim_{j \in \JS} \Hom_{\CS}(\alpha(i),\beta(j))
\]
where on the right-hand side the colimit and limit are taken in $\Vect^\Gamma$
(equivalently here, in $\VectZ$).
\end{definition}

An object of $\Ind(\CS)$ will be informally referred to as an \emph{ind-object} of $\CS$.
We will denote ind-objects either by explicitly displaying the directed system, e.g.
\[
\alpha(0) \to \alpha(1) \to \cdots
\]
or simply by writing\footnote{Since the objects of $\Ind(\CS)$ are ``formal" colimits,
it is common to denote them by ``$\colim_{i \in \IS}$''$\alpha(i)$; 
however, we drop the quotation marks since it will be clear from context whether
the colimit is an object of $\Ind(\CS)$ or an object in a category admitting filtered colimits.} 
$\colim_{i \in \IS}\alpha(i)$.
In the following, we collect a number of useful facts/observations 
that will facilitate working in $\Ind(\CS)$; 
almost all can be found in \cite{KaSch}.

\begin{remark}
	\label{rmk:C^I to Ind(C)}
Let $\IS$ be a (small) category and let $\CS$ be $\Gamma$-graded $\k$-linear.
The collection of $\IS$-indexed diagrams in $\CS$ forms a $\Gamma$-graded $\k$-linear category $\CS^{\IS}$
wherein morphisms are homogeneous natural transformations. 
If $\IS$ is filtered, then there is a functor $\CS^\IS \to \Ind(\CS)$, 
since natural transformations of diagrams induce maps of colimits. 
In fact, such functors describe all morphisms in $\Ind(\CS)$. 
To wit, suppose we are given a morphism $f \in \Hom_{\Ind(\CS)}(\alpha,\beta)$ from an 
ind-object $\alpha \colon \IS \to \CS$ to $\beta \colon \JS \to \CS$.
Then, there exists a filtered category $\KS$ and final functors $F_{\IS} \colon \KS \to \IS$ and $F_{\JS} \colon \KS \to \JS$ 
and a natural transformation $\hat{f} \colon \alpha \circ F_{\IS} \to \beta \circ F_{\JS}$ so that the induced morphism
\[
\colim_{i \in \IS} \alpha(i) \cong \colim_{k \in \KS} \alpha(F_{\IS}(k)) \xrightarrow{\hat{f}}
\colim_{k \in \KS} \beta(F_{\JS}(k)) \cong \colim_{j \in \JS} \beta(j)
\]
agrees with $f$. 
(Here, we also denote  the map induced on colimits by $\hat{f}$ by $\hat{f}$ as well.)
This description e.g.~elucidates composition of morphisms in $\Ind(\CS)$.
\end{remark}

\begin{remark}
	\label{rem:compact}
Suppose that $\DS$ is a ($\Gamma$-graded $\k$-linear) category that is closed under all (small) filtered colimits. 
Recall that an object $K \in \DS$ is called \emph{compact} provided for all filtered colimits 
$\colim_{i\in \IS}\a(i)$ in $\DS$ we have that
\[
\Hom_\DS(K,\colim_{i\in \IS}\a(i)) = \colim_{i\in \IS} \Hom_\DS(K,\alpha(i)) \, .
\]
If $\CS$ is a $\Gamma$-graded $\k$-linear category, then we may regard it as a
full subcategory of $\Ind(\CS)$ via the fully faithful functor $\CS
\hookrightarrow \Ind(\CS)$ that sends an object $X$ in $\CS$ to the directed
system that is constant at $X$ and indexed by the one element poset.
Under this inclusion, $\CS$ is equivalent to the full
subcategory of compact objects in $\Ind(\CS)$. 
This implies that $\Ind(\CS)$ is \emph{compactly generated}, 
i.e.~all objects are filtered colimits of compact objects.

Further, $\Ind(\CS)$ admits all filtered colimits and the inclusion $\CS \hookrightarrow \Ind(\CS)$ is initial
among all functors from $\CS$ to categories that admit all filtered colimits. 
Explicitly, if $\DS$ admits all filtered colimits, then we can extend any functor $F \colon \CS \to \DS$ 
to $\Ind(\CS)$ by declaring that $F\big(\colim_{i \in \IS} \alpha(i) \big) = \colim_{i \in \IS} F \big( \alpha(i) \big)$.
\end{remark}

\begin{remark}
	\label{rem:Indmonoidal}
Remark \ref{rem:compact} implies that the Yoneda embedding $\CS \hookrightarrow \CS^{\wedge}$
factors through the inclusion $\CS \hookrightarrow \Ind(\CS)$.
Indeed, $\Ind(\CS)$ admits an alternative definition as the full subcategory of $\CS^{\wedge}$
of objects that are filtered colimits of representable functors.

Now suppose that $\CS$ is a \emph{monoidal} $\Gamma$-graded $\k$-linear category.
The presheaf category $\CS^{\wedge}$ is then monoidal under Day convolution \cite{Day}.
This restricts to a monoidal structure on $\Ind(\CS)$ that commutes with filtered colimits 
in each factor, i.e.
\[
\colim_{i \in \IS} \alpha(i) \otimes \colim_{j \in \IS} \beta(j) \cong \colim_{i,j \in \IS \times \JS} \alpha(i) \otimes \beta(j) \, .
\]
Note that the inclusion $\CS \hookrightarrow \Ind(\CS)$ is monoidal.
\end{remark}

\subsection{Recognizing infinite biproducts}
\label{ss:biprod recognition}

We conclude our categorical background with the following technical result, 
which provides an analogue of Lemma \ref{lemma:finite biproduct recognition} 
for certain infinite biproducts.

\begin{lemma}[Biproduct Recognition]
	\label{lemma:biproduct recognition} 
Let $\DS$ be a compactly generated $\Gamma$-graded $\k$-linear category, 
and let morphisms $\pi_i \in \Hom_{\DS_0}(Y, Y_i)$ 
and $\sigma_i \in \Hom_{\DS_0}(Y_i ,Y)$ 
be given for $i \in I$ (an arbitrary indexing set). 
The following conditions suffice for the maps $\{\pi_i\}_{i \in I}$ and $\{\sigma_i\}_{i \in I}$
to exhibit $Y$ as the biproduct of the $Y_i$:
\begin{enumerate}[(i)]
\item  $\pi_i\circ \sigma_j = \delta_{ij} \id_{Y_i}$,
\item for each compact $K \in \DS$ and $k \in \Gamma$, 
we have that $\Hom_{\DS}^k(K,Y_i)=0$ for all but finitely many $i \in I$, and
\item for each compact $K \in \DS$ and $f \in \Hom_{\DS}(K,Y)$,
we have that $f = \sum_{i \in I} \sigma_i\circ \pi_i \circ f$.
\end{enumerate}
(Note that the sum in (iii) is finite by (ii).)
\end{lemma}
\begin{proof}
We first show that the maps $\{\pi_i\}_{i \in I}$ exhibit $Y$ as a product $Y \cong \prod_{i \in I} Y_i$.  
For this, we must construct a two-sided inverse to the assignment:
\[
\Phi \colon \Hom_\DS(X,Y) \to \prod_{i \in I} \Hom_\DS(X,Y_i) \, , \quad F \mapsto (\pi_i\circ F)_{i \in I}
\]
where $X\in \CS$ is arbitrary. 

To construct the inverse $\Phi'$, we approximate $X$ by compact objects, 
i.e.~we write $X = \colim_\alpha K_\alpha$ with $K_\alpha$ compact.
Now let $(f_i)_{i \in I} \in \prod_{i \in I} \Hom_\CS(X,Y_i)$ be given. 
For each $\a$, set
\[
F'_\a:=\sum_{i \in I} \sigma_i\circ f_i\circ \iota_\a \in \Hom_{\DS}(K_\a,Y),
\]
where $\iota_\a \in \Hom_{\DS}(K_\alpha,X)$ 
are the structure maps of the colimit. 
Since $K_\a$ is compact and $f_i \circ \iota_\alpha \in \Hom_\DS(K_\alpha,Y_i)$,
this sum is finite (in each homogeneous degree), hence gives a well-defined morphism.
Furthermore, the collection of maps $F'_\a \in \Hom_\DS(K_\a,Y)$ is compatible with the morphisms
$\gamma_{\b,\a} \colon K_\a \to K_\b$ in the directed system (whose colimit is $X$) 
in the following sense: 
\[
F'_\b\circ \gamma_{\b,\a} = \sum_{i \in I} \sigma_i\circ f_i\circ \iota_\b \circ \gamma_{\b,\a} 
=  \sum_{i \in I} \sigma_i\circ f_i\circ \iota_\a = F'_\a.
\]

We hence let
\[
\Phi' \colon \prod_{i \in I} \Hom_\DS(X,Y_i) \rightarrow  \Hom_\DS(X,Y) \, , \quad 
\Phi' \big((f_i)_{i \in I} \big) := \colim_\alpha F'_\a \, .
\]
Since $X=\colim_\a K_\a$, we have that
$\Hom_\DS(X,Y) \cong \lim_\alpha \Hom_\DS(K_\alpha,Y)$.
Thus, we may check that $\Phi'(\Phi(F))=F$ by showing that this holds
upon precomposing with each $\iota_\a$.
We thus compute:
\[
\Phi'(\Phi(F))\circ \iota_\a  = \Phi' \big( (\pi_i\circ F)_{i \in I} \big) \circ \iota_\a 
= \sum_{i \in I} \sigma_i\circ (\pi_i \circ F)\circ \iota_\alpha 
= \sum_{i \in I} \sigma_i\circ \pi_i \circ (F \circ \iota_\alpha) 
\stackrel{(iii)}{=} F \circ \iota_\a
\]
as desired.
To see that $\Phi \big(\Phi' \big( (f_i)_{i\in I} \big) \big) = (f_i)_{i \in I}$, 
we must check that
\[
\pi_i \circ \Phi'((f_i)_{i \in I}) =  f_i
\]
as morphisms $X\rightarrow Y_i$, for all $i \in I$. 
Again, it suffices to prove equality holds upon precomposing with each $\iota_\a$, 
so we compute
\[
\pi_i \circ \Phi'(\{f_i\}_i)\circ \iota_\a 
= \pi_i\circ \sum_{j \in I} \sigma_j\circ f_j\circ \iota_\a 
= \sum_{j \in I} \pi_i \circ  \sigma_j\circ f_j\circ \iota_\a 
\stackrel{(i)}{=} f_i\circ \iota_\alpha.
\]
as desired.

It remains to show that the maps $\{\sigma_i\}_{i \in I}$ exhibit $Y$ as a 
coproduct $Y \cong \coprod_{i \in I} Y_i$. 
To see this, we must construct a two-sided inverse to
\[
\Psi \colon \Hom_\DS(Y,Z) \to \prod_{i \in I} \Hom_\DS(Y_i,Z) \, , \quad 
G \mapsto (G\circ \sigma_i)_{i \in I}
\]
where $Z\in \DS$ is arbitrary. 
We begin by approximating $Y$ by compact objects, 
i.e.~ we write $Y = \colim_\b C_\b$ with $C_\b \in \DS$ compact. 
Our candidate two-sided inverse to $\Psi$ is defined by:
\begin{equation}
	\label{eq:Psiprime}
\Psi' \colon \prod_{i \in I} \Hom_\DS(Y_i,Z) \to \Hom_\DS(Y,Z) \, , \quad 
\Psi' \big( (g_i)_{i \in I} \big) \circ \iota_\b = \sum_{i \in I} g_i\circ \pi_i \circ \iota_\b \, .
\end{equation}
Here, $\iota_\b \in \Hom_{\DS}(C_\b,Y)$ are the structure maps of the colimit, 
and we again use that $\Hom_\DS(Y,Z) \cong \lim_\beta \Hom_\DS(C_\beta,Z)$ 
to see that the formulae specify $\Psi' \big( (g_i)_{i \in I} \big) \in \Hom_\DS(Y,Z)$.
For each $\beta$, the sum in \eqref{eq:Psiprime}
is finite (thus well-defined) by condition (ii) 
applied to the maps $\pi_i \circ \iota_\beta \in \Hom_\DS(C_\beta,Y_i)$.

We compute that:
\[
\Psi'(\Psi(G)) \circ \iota_\b 
=  \sum_{i \in I} (G\circ \sigma_i)\circ \pi_i \circ \iota_\b 
= G \circ \sum_{i \in I} \sigma_i \circ \pi_i \circ \iota_\b
\stackrel{(iii)}{=} G \circ \iota_\b
\]
for each $\b$, thus $\Psi'(\Psi(G)) = G$.
To see that $\Psi \big( \Psi' \big( (g_i)_{i \in I} \big) \big) = (g_i)_{i \in I}$, 
we must show that $\Psi' \big( (g_i)_{i \in I} \big) \circ \sigma_i = g_i$ for all $i$. 
Since $\id_{Y_i} = \pi_i \circ \sigma_i$,
it suffices to show that, for all $i \in I$, 
$\Psi'((g_i)_{i \in I}) \circ \sigma_i\circ \pi_i = g_i\circ \pi_i$ as maps $Y \to Z$.  
Yet again, since $Y= \colim_\b C_\b$, it suffices to show that this identity
holds after precomposing with $\iota_{\b'}$ for all $\b'$. In other words, we
have reduced the problem to establishing the identity
\[
\colim_\b \Big( \sum_{j \in I} g_j \circ \pi_j \circ \iota_\b \Big) \circ \sigma_i\circ \pi_i \circ \iota_{\b'}
= g_i\circ \pi_i \circ \iota_{\b'}
\]
for all $i$ and all $\b'$.   
For this, we will use compactness of $C_{\beta'}$ in an essential way.

To wit, since $Y = \colim_\b C_\b$ is a \emph{filtered} colimit and $C_{\beta'}$ is compact, 
we have that 
\[
\Hom_{\DS}(C_{\beta'},Y) = \Hom_{\DS}(C_{\beta'},\colim_\b C_\b)
= \colim_\b \Hom_{\DS}(C_{\beta'}, C_\b) \, .
\]
Further, since the colimit is taken over a directed system, 
every element of the latter lies in the image of the map
\[
\Hom_{\DS}(C_{\beta'}, C_{\beta_0}) \xrightarrow{\iota_{\beta_0} \circ -} \Hom_{\DS}(C_{\beta'},Y)
\]
for some $\beta_0$.
It follows that we may write
\[
\sigma_i \circ \pi_i \circ \iota_{\beta'}  = \iota_{\beta_0} \circ \nu
\]
for some $\nu \in \Hom_{\DS}(C_{\beta'}, C_{\beta_0})$.
We then compute
\begin{align*}
\colim_\b \Big( \sum_{j \in I} g_j \circ \pi_j \circ \iota_\b \Big) \circ \sigma_i\circ \pi_i \circ \iota_{\b'}
&= \colim_\b \Big( \sum_{j \in I} g_j \circ \pi_j \circ \iota_\b \Big) \circ \iota_{\beta_0} \circ \nu \\
&= \sum_{j \in I} g_j \circ \pi_j \circ \iota_{\beta_0} \circ \nu \\
&= \sum_{j \in I} g_j \circ \pi_j \circ \sigma_i \circ \pi_i \circ \iota_{\beta'} \\
&= g_i \circ \pi_i \circ \iota_{\beta'}
\end{align*}
as desired.
\end{proof}

\section{Dotted Temperley--Lieb and annular Bar-Natan categories}
\label{sec:dTL}

In this section, we recall the dotted Temperley--Lieb category 
and a theorem of Russell \cite{Russell}
showing that it gives an explicit monoidal presentation 
for the Bar-Natan skein module of the annulus. 
We further discuss a variation on the latter in which curves 
(and cobordisms) are permitted 
to meet the boundary of the (thickened) annulus. 
Such categories are module categories for 
the Bar-Natan skein module of the annulus (without boundary), 
and we give a presentation 
for the $2$-boundary version and its action 
by the dotted Temperley--Lieb category.

From now on, we work with a field $\K$ of characteristic zero.

\subsection{The categories $\dTL$ and $\ABN$}
	\label{ss:dTL}

\begin{definition}
	\label{def:dTL}
The \emph{dotted Temperley--Lieb category} $\dTL$ is the 
$\Z$-graded $\K$-linear 
pivotal category 
freely generated by a single (symmetrically) self-dual object $\gen$ and an endomorphism 
$x \in \End^2_{\dTL}(\gen)$,
modulo a certain monoidal ideal generated by three relations. 
Using the standard graphical language for pivotal categories 
and denoting the endomorphism $x$ as a dot, the relations are as follows:
\begin{equation}
	\label{eq:dTLrels}
\begin{tikzpicture}[anchorbase,scale=.75]
\draw[very thick] (0,0) circle (.5);
\end{tikzpicture} = 2 \,\, , \quad
\begin{tikzpicture}[anchorbase,scale=1]
\draw[very thick] (0,0) to node{$\bullet$} node[right=-1pt,yshift=3pt]{\scriptsize$2$} (0,1);
\end{tikzpicture} 
:=
\begin{tikzpicture}[anchorbase,scale=1]
\draw[very thick] (0,0) to node[pos=.3]{$\bullet$} node[pos=.7]{$\bullet$} (0,1);
\end{tikzpicture} 
=0 \, , \quad
\begin{tikzpicture}[anchorbase,scale=1]
\draw[very thick] (0,0) to node{$\bullet$}(0,1);
\draw[very thick] (.5,0) to (.5,1);
\end{tikzpicture} 
\, + \,
\begin{tikzpicture}[anchorbase,scale=1]
\draw[very thick] (0,0) to (0,1);
\draw[very thick] (.5,0) to node{$\bullet$} (.5,1);
\end{tikzpicture} 
=
\begin{tikzpicture}[anchorbase,scale=1]
\draw[very thick] (0,0) to [out=90,in=180] (.25,.375) node{$\bullet$} to [out=0,in=90] (.5,0);
\draw[very thick] (0,1) to [out=270,in=180] (.25,.625) to [out=0,in=270] (.5,1);
\end{tikzpicture} 
+
\begin{tikzpicture}[anchorbase,scale=1]
\draw[very thick] (0,0) to [out=90,in=180] (.25,.375) to [out=0,in=90] (.5,0);
\draw[very thick] (0,1) to [out=270,in=180] (.25,.625) node{$\bullet$} to [out=0,in=270] (.5,1);
\end{tikzpicture} \, .
\end{equation}
\end{definition}

These relations, together with the relation that a dotted circle equals zero
(which is implied by \eqref{eq:dTLrels} since $\mathrm{char(\K)}\neq 2$), 
are called the \emph{SBN relations} in \cite{MR2370224}.

\begin{remark}
Unpacking Definition \ref{def:dTL}, we see that objects in $\dTL$ are the tensor
powers $\gen^n:=\gen^{\otimes n}$ of the generating object $\gen$ for $n\in
\N$. The morphism space $\Hom_{\dTL}(\gen^m,\gen^n)$ is the $\Z$-graded $\K$-vector space spanned
by dotted $(m,n)$-planar tangles (with $m$ points at the bottom and $n$ points
at the top), modulo the relations in \eqref{eq:dTLrels} and planar isotopy
(which is implicit in the word ``pivotal'' above). The degree of such a dotted
tangle equals $2(\# \text{ of dots})$. Note that the relations in
\eqref{eq:dTLrels} imply that
\begin{equation}
	\label{eq:twodotswitch}
\begin{tikzpicture}[anchorbase,scale=1]
\draw[very thick] (0,0) to node{$\bullet$} (0,1);
\draw[very thick] (.5,0) to node{$\bullet$} (.5,1);
\end{tikzpicture} 
=
\begin{tikzpicture}[anchorbase,scale=1]
\draw[very thick] (0,0) to [out=90,in=180] (.25,.375) node{$\bullet$} 
	to [out=0,in=90] (.5,0);
\draw[very thick] (0,1) to [out=270,in=180] (.25,.625) node{$\bullet$} 
	to [out=0,in=270] (.5,1);
\end{tikzpicture} \, .
\end{equation} 
\end{remark}

\begin{remark}
The endomorphism algebras of $\dTL$ are also called ``nil-blob
algebras''. In \cite{MR4187261}, it is shown that they arise naturally as
diagrammatically defined endomorphism subalgebras of type $\tilde{A_1}$ Soergel
bimodules and as singular weight idempotent truncations of suitable KLR
algebras.
\end{remark}

\begin{remark}
	\label{rem:dTLrepth}
	Let $V=\spann_{\Q}(v_+,v_-)$ denote the graded $\Q$-vector space with basis
	elements of degree $\deg(v_\pm)=\pm 1$. The dotted Temperley--Lieb category
	$\dTL$ over $\K=\Q$ acts on tensor powers of $V$ as follows:
	\begin{gather*}
		\begin{tikzpicture}[anchorbase,scale=1]
			\draw[very thick] (0,0) to [out=90,in=180] (.25,.375)
				to [out=0,in=90] (.5,0);
			\end{tikzpicture}\colon V\otimes V \to \Q\, , 
			\begin{cases}
				v_+\otimes v_+\mapsto 0\\
				v_+\otimes v_-\mapsto 1\\
				v_-\otimes v_+\mapsto 1\\
				v_-\otimes v_-\mapsto 0\\
			\end{cases}, \quad 
			\begin{tikzpicture}[anchorbase,scale=1]
				\draw[very thick] (0,0) to node{$\bullet$} (0,1);
				\end{tikzpicture} 
				\colon V \to V\, , \; 
				\begin{cases}
					v_+ \mapsto 0\\
					v_- \mapsto v_+
				\end{cases}
			\, ,\\		
			\begin{tikzpicture}[anchorbase,yscale=-1]
				\draw[very thick] (0,0) to [out=90,in=180] (.25,.375)
					to [out=0,in=90] (.5,0);
				\end{tikzpicture}\colon
				\Q \to V \otimes V\, , \; 1 \mapsto v_+\otimes v_- + v_- \otimes v_+\, .
	\end{gather*}
 The cup and the cap maps can be
interpreted as the $q=1$ reduction of morphisms between tensor powers of type II
vector representations of the quantum group $U_q(\slnn{2})$. The dot corresponds
to the action of the quantum group 
Chevalley generator $E$. The action is
faithful, as we will see (in different language) in \S\ref{ss:PolyRep}.
\end{remark}

\begin{remark}
Restricting to the subcategory of degree zero morphisms in $\dTL$ recovers the 
circle-value $2$ specialization of the
Temperley--Lieb category.
Recall that this category is equivalent to the full subcategory of 
$\slnn{2}$ representations tensor generated by the defining representation.
In the following, the Temperley--Lieb category will always refer to this specialization, 
which we denote by $\TL$.
\end{remark}

We next discuss the Bar-Natan category \cite{BN2} of an orientable surfaces 
$(\Sigma, \mathbf{p})$ with marked points on its boundary, 
which is the natural setting for Khovanov homology 
(of tangles in thickened surfaces).
When $\Sigma = \D$ and $|\mathbf{p}| = m+n$, 
this categorifies the $\Hom$-space $\Hom_{\TL}(\gen^m,\gen^n)$. 
Surprisingly (at least to those unfamiliar with the theory of trace decategorification), 
we'll see below that in the case when $\Sigma=\Ann$ is the annulus 
and $\mathbf{p}=\varnothing$, this category agrees with $\dTL$.

\begin{definition}
	\label{def:BN}
Let $\Sigma$ be an orientable surface with (possibly empty) boundary and let
$\mathbf{p} \subset \partial \Sigma$ be finite. 
The \emph{Bar-Natan category} is 
$\BN(\Sigma ; \mathbf{p}) := \Mat\big(\bn(\Sigma ; \mathbf{p})\big)$
where $\bn(\Sigma ; \mathbf{p})$ is the $\Z$-graded $\K$-linear category
defined as follows.
Objects in $\bn(\Sigma ; \mathbf{p})$ are
smoothly embedded $1$-manifolds $C \subset \Sigma$ with boundary
$\partial C = \mathbf{p}$ meeting $\partial \Sigma$ transversely. Given objects
$C_1,C_2$, $\Hom_{\bn}(C_1,C_2)$ is the $\Z$-graded $\K$-module spanned by
embedded orientable cobordisms $W \subset \Sigma \times [0,1]$ 
with corners (when $\mathbf{p} \neq \varnothing$) from
$C_1$ to $C_2$, modulo the following local relations:
\begin{equation}
	\label{eq:BNrels}
\begin{tikzpicture} [fill opacity=0.2,anchorbase, scale=.375]
	\path[fill=red, opacity=.2] (1,0) arc[start angle=0, end angle=180,x radius=1,y radius=.5] 
		to (-1,4) arc[start angle=180, end angle=0,x radius=1,y radius=.5] to (1,0);
	\path[fill=red, opacity=.2] (1,0) arc[start angle=360, end angle=180,x radius=1,y radius=.5] 
		to (-1,4) arc[start angle=180, end angle=360,x radius=1,y radius=.5] to (1,0);
	\draw [very thick] (0,4) ellipse (1 and 0.5);
	\draw [very thick] (0,0) ellipse (1 and 0.5);
	\draw[very thick] (1,4) -- (1,0);
	\draw[very thick] (-1,4) -- (-1,0);
\end{tikzpicture}
\, = \,
\begin{tikzpicture} [fill opacity=0.2,anchorbase, scale=.375,rotate=180]
	\path[fill=red,opacity=.2] (1,4) arc[start angle=0, end angle=180,x radius=1,y radius=.5] 
		to [out=270,in=180] (0,2.5) to [out=0,in=270] (1,4);
	\path[fill=red,opacity=.2] (1,4) arc[start angle=360, end angle=180,x radius=1,y radius=.5] 
		to [out=270,in=180] (0,2.5) to [out=0,in=270] (1,4);
	\draw[very thick] (0,4) ellipse (1 and 0.5);
	\draw[very thick] (-1,4) to [out=270,in=180] (0,2.5) to [out=0,in=270] (1,4);
	\path[fill=red,opacity=.2] (1,0) arc[start angle=0, end angle=180,x radius=1,y radius=.5] 
		to [out=90,in=180] (0,1.5) to [out=0,in=90] (1,0);
	\path[fill=red,opacity=.2] (1,0) arc[start angle=360, end angle=180,x radius=1,y radius=.5] 
		to [out=90,in=180] (0,1.5) to [out=0,in=90] (1,0);
	\draw[very thick] (0,0) ellipse (1 and 0.5);
	\draw[very thick] (-1,0) to [out=90,in=180] (0,1.5) to [out=0,in=90] (1,0);
	\node[opacity=1] at (0,1) {\footnotesize$\bullet$};
\end{tikzpicture}
+
\begin{tikzpicture} [fill opacity=0.2,anchorbase, scale=.375]
	\path[fill=red,opacity=.2] (1,4) arc[start angle=0, end angle=180,x radius=1,y radius=.5] 
		to [out=270,in=180] (0,2.5) to [out=0,in=270] (1,4);
	\path[fill=red,opacity=.2] (1,4) arc[start angle=360, end angle=180,x radius=1,y radius=.5] 
		to [out=270,in=180] (0,2.5) to [out=0,in=270] (1,4);
	\draw[very thick] (0,4) ellipse (1 and 0.5);
	\draw[very thick] (-1,4) to [out=270,in=180] (0,2.5) to [out=0,in=270] (1,4);
	\path[fill=red,opacity=.2] (1,0) arc[start angle=0, end angle=180,x radius=1,y radius=.5] 
		to [out=90,in=180] (0,1.5) to [out=0,in=90] (1,0);
	\path[fill=red,opacity=.2] (1,0) arc[start angle=360, end angle=180,x radius=1,y radius=.5] 
		to [out=90,in=180] (0,1.5) to [out=0,in=90] (1,0);
	\draw[very thick] (0,0) ellipse (1 and 0.5);
	\draw[very thick] (-1,0) to [out=90,in=180] (0,1.5) to [out=0,in=90] (1,0);
	\node[opacity=1] at (0,1) {\footnotesize$\bullet$};
\end{tikzpicture}
\, , \quad
\begin{tikzpicture}[anchorbase, scale=.375]
	\path [fill=red,opacity=0.3] (0,0) circle (1);
	\draw (-1,0) .. controls (-1,-.4) and (1,-.4) .. (1,0);
	\draw[dashed] (-1,0) .. controls (-1,.4) and (1,.4) .. (1,0);
	\draw[very thick] (0,0) circle (1);
\end{tikzpicture}
= 0
\, , \quad
\begin{tikzpicture}[anchorbase, scale=.375]
	\path [fill=red,opacity=0.3] (0,0) circle (1);
	\draw (-1,0) .. controls (-1,-.4) and (1,-.4) .. (1,0);
	\draw[dashed] (-1,0) .. controls (-1,.4) and (1,.4) .. (1,0);
	\draw[very thick] (0,0) circle (1);
	\node at (0,0.6) {\footnotesize$\bullet$};
\end{tikzpicture}
= 1
\, , \quad
\begin{tikzpicture}[fill opacity=.3, scale=.5, anchorbase]
	\filldraw [very thick,fill=red] (-1,-1) rectangle (1,1);
	\node [opacity=1] at (0,-.25) {$\bullet$};
	\node [opacity=1] at (0,.25) {$\bullet$};
	\end{tikzpicture}
= 0 \, .
\end{equation}
 The degree of a cobordism with corners $W \colon C_1 \to C_2$ is given
by $\deg(W) = \frac{1}{2}|\mathbf{p}| - \chi(W)$, and a dot on a surface is used
as shorthand for taking connect sum with a torus at that point and multiplying
by $\frac{1}{2}$. For example, 
\begin{equation}
	\label{eq:dotdef}
\begin{tikzpicture} [fill opacity=0.2,anchorbase, scale=.375]
	\path[fill=red,opacity=.2] (1,4) arc[start angle=0, end angle=180,x radius=1,y radius=.5] 
		to [out=270,in=180] (0,2.5) to [out=0,in=270] (1,4);
	\path[fill=red,opacity=.2] (1,4) arc[start angle=360, end angle=180,x radius=1,y radius=.5] 
		to [out=270,in=180] (0,2.5) to [out=0,in=270] (1,4);
	\draw[very thick] (0,4) ellipse (1 and 0.5);
	\draw[very thick] (-1,4) to [out=270,in=180] (0,2.5) to [out=0,in=270] (1,4);
	\node[opacity=1] at (0,3) {\footnotesize$\bullet$};
\end{tikzpicture}
:=
\frac{1}{2}
\begin{tikzpicture}[scale=.5,anchorbase]
\begin{scope}
    \clip (-.65,4) arc[start angle=180, end angle=360,x radius=.65,y radius=.25]
    	to (1.25,4) to (1.25,.4) to (-1.25,.4) to  (-1.25,4) to (-.65,4);
\fill[red,opacity=.3] (0,2.5) ellipse (1 and 2);
\draw[very thick] (0,2.5) ellipse (1 and 2);
\fill[white] (0,3) to [out=300,in=60] (0,2) to [out=120,in=240] (0,3);
\draw[very thick] (0,3) to [out=300,in=60] (0,2);
\draw[very thick] (0.1,1.8) to [out=125,in=235] (0.1,3.2);
\end{scope}
\fill[red,opacity=.2] (0,4) ellipse (.65 and .25);
\draw[very thick] (0,4) ellipse (.66 and .25);
\end{tikzpicture} \, .
\end{equation}
\end{definition}

\begin{remark} It is possible to define the
Bar-Natan category integrally, if we work with ``formally dotted'' cobordisms.
In this case, the first relation in \eqref{eq:BNrels} (the so-called
\emph{neck-cutting relation}) implies the characterization of a dot given in
\eqref{eq:dotdef}, after clearing denominators.
\end{remark}

\begin{remark}
The first three relations in \eqref{eq:BNrels} imply that the neck-cutting relation gives
an idempotent decomposition for the identity morphism of any null-homotopic circle 
in $\Sigma$, i.e.~we have the ``circle removal'' isomorphism
\begin{equation}
	\label{eq:circle}
\begin{tikzpicture}[anchorbase,scale=.5]
\draw[very thick] (0,0) circle (.5);
\end{tikzpicture}
\cong q^{-1} \varnothing \oplus q \varnothing
\end{equation}
in $\BN(\Sigma ; \mathbf{p})$ for such circles.
\end{remark}

\begin{definition}\label{def:ABN}
Let $\ABN:=\BN(S^1\times [0,1] ; \varnothing)$ denote the Bar-Natan category 
associated to the annulus with no points on the boundary.  
We refer to $\ABN$ as the \emph{annular Bar-Natan category}.
\end{definition}
\begin{remark}
	\label{rem:ABNmonoidal}
$\ABN$ is a monoidal category, with tensor product given by glueing one 
(thickened) annulus inside the other, so that $A\otimes B$ is ``$A$ inside $B$''.
\end{remark}

The following is essentially a re-packaging of a theorem of  Russell \cite{Russell}.

\begin{prop}
	\label{prop:dTLABN}
There is a fully faithful monoidal functor 
$\Phi \colon \dTL \hookrightarrow \ABN$ defined by 
``rotating dotted Temperley--Lieb diagrams'' around the annular core, i.e.
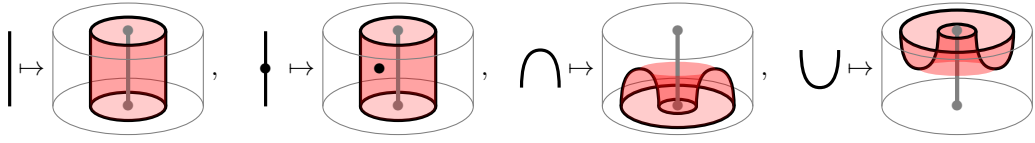
\[
\begin{tikzpicture}[anchorbase,scale=1]
\draw[very thick] (0,0) to (0,1);
\end{tikzpicture} 
\mapsto
\begin{tikzpicture}[anchorbase,scale=1]
\node[gray] at (0,0) {$\bullet$};
\draw[gray] (0,0) ellipse (1 and 0.375);
\draw[gray] (-1,0) to (-1,1);
\draw[gray] (1,0) to (1,1);
\draw[very thick] (0,0) ellipse (.5 and 0.1875);
\draw[very thick] (-.5,0) to (-.5,1);
\draw[very thick] (.5,0) to (.5,1);
\path[fill=red,opacity=.2] (.5,1) arc[start angle=0, end angle=180,x radius=.5,y radius=.1875] 
	to (-.5,0) arc[start angle=180, end angle=0,x radius=.5,y radius=.1875] to (.5,1);
\draw[gray,ultra thick] (0,0) to (0,1);
\draw[very thick] (0,1) ellipse (.5 and 0.1875);
\path[fill=red,opacity=.2] (.5,1) arc[start angle=360, end angle=180,x radius=.5,y radius=.1875] 
	to (-.5,0) arc[start angle=180, end angle=360,x radius=.5,y radius=.1875] to (.5,1);
\node[gray] at (0,1) {$\bullet$};
\draw[gray] (0,1) ellipse (1 and 0.375);
\end{tikzpicture} 
\; , \quad
\begin{tikzpicture}[anchorbase,scale=1]
\draw[very thick] (0,0) to node{$\bullet$} (0,1);
\end{tikzpicture} 
\mapsto
\begin{tikzpicture}[anchorbase,scale=1]
\node[gray] at (0,0) {$\bullet$};
\draw[gray] (0,0) ellipse (1 and 0.375);
\draw[gray] (-1,0) to (-1,1);
\draw[gray] (1,0) to (1,1);
\draw[very thick] (0,0) ellipse (.5 and 0.1875);
\draw[very thick] (-.5,0) to (-.5,1);
\draw[very thick] (.5,0) to (.5,1);
\path[fill=red,opacity=.2] (.5,1) arc[start angle=0, end angle=180,x radius=.5,y radius=.1875] 
	to (-.5,0) arc[start angle=180, end angle=0,x radius=.5,y radius=.1875] to (.5,1);
\draw[gray,ultra thick] (0,0) to (0,1);
\draw[very thick] (0,1) ellipse (.5 and 0.1875);
\path[fill=red,opacity=.2] (.5,1) arc[start angle=360, end angle=180,x radius=.5,y radius=.1875] 
	to (-.5,0) arc[start angle=180, end angle=360,x radius=.5,y radius=.1875] to (.5,1);
\node at (-.25,.5) {$\bullet$};
\node[gray] at (0,1) {$\bullet$};
\draw[gray] (0,1) ellipse (1 and 0.375);
\end{tikzpicture} 
\; , \quad
\begin{tikzpicture}[anchorbase,scale=1]
\draw[very thick] (0,0) to [out=90,in=180] (.25,.5) to [out=0,in=90] (.5,0);
\end{tikzpicture} 
\mapsto
\begin{tikzpicture}[anchorbase,scale=1]
\node[gray] at (0,0) {$\bullet$};
\draw[gray] (0,0) ellipse (1 and 0.375);
\draw[gray] (-1,0) to (-1,1);
\draw[gray] (1,0) to (1,1);
\draw[very thick] (0,0) ellipse (.25 and 0.09375);
\draw[very thick] (0,0) ellipse (.75 and 0.28125);
\draw[very thick] (-.75,0) to [out=90,in=180] (-.5,.5) to [out=0,in=90] (-.25,0);
\draw[very thick] (.25,0) to [out=90,in=180] (.5,.5) to [out=0,in=90] (.75,0);
\path[fill=red,opacity=.2] (-.75,0) to [out=90,in=180] (-.5,.5)
	arc[start angle=180, end angle=0,x radius=.5,y radius=0.09375] to [out=0,in=90] 
		(.75,0) arc[start angle=0, end angle=180,x radius=.75,y radius=0.28125];
\path[fill=red,opacity=.2] (-.5,.5) arc[start angle=180, end angle=0,x radius=.5,y radius=0.09375] 
	to [out=180,in=90] (.25,0) arc[start angle=0, end angle=180,x radius=.25,y radius=0.09375]
		to [out=90,in=0] (-.5,.5);
\draw[gray,ultra thick] (0,0) to (0,1);
\path[fill=red,opacity=.2] (-.75,0) to [out=90,in=180] (-.5,.5)
	arc[start angle=180, end angle=360,x radius=.5,y radius=0.09375] to [out=0,in=90] 
		(.75,0) arc[start angle=360, end angle=180,x radius=.75,y radius=0.28125];
\path[fill=red,opacity=.2] (-.5,.5) arc[start angle=180, end angle=360,x radius=.5,y radius=0.09375] 
	to [out=180,in=90] (.25,0) arc[start angle=360, end angle=180,x radius=.25,y radius=0.09375]
		to [out=90,in=0] (-.5,.5);
\node[gray] at (0,1) {$\bullet$};
\draw[gray] (0,1) ellipse (1 and 0.375);
\end{tikzpicture} 
\; , \quad
\begin{tikzpicture}[anchorbase,scale=1,rotate=180]
\draw[very thick] (0,0) to [out=90,in=180] (.25,.5) to [out=0,in=90] (.5,0);
\end{tikzpicture} 
\mapsto
\begin{tikzpicture}[anchorbase,scale=1]
\node[gray] at (0,0) {$\bullet$};
\draw[gray] (0,0) ellipse (1 and 0.375);
\draw[gray] (-1,0) to (-1,1);
\draw[gray] (1,0) to (1,1);
\path[fill=red,opacity=.2] (-.75,1) to [out=270,in=180] (-.5,.5)
	arc[start angle=180, end angle=0,x radius=.5,y radius=0.09375] to [out=0,in=270] 
		(.75,1) arc[start angle=0, end angle=180,x radius=.75,y radius=0.28125];
\path[fill=red,opacity=.2] (-.5,.5) arc[start angle=180, end angle=0,x radius=.5,y radius=0.09375] 
	to [out=180,in=270] (.25,1) arc[start angle=0, end angle=180,x radius=.25,y radius=0.09375]
		to [out=270,in=0] (-.5,.5);
\draw[gray,ultra thick] (0,0) to (0,1);
\draw[very thick] (0,1) ellipse (.25 and 0.09375);
\draw[very thick] (0,1) ellipse (.75 and 0.28125);
\draw[very thick] (-.75,1) to [out=270,in=180] (-.5,.5) to [out=0,in=270] (-.25,1);
\draw[very thick] (.25,1) to [out=270,in=180] (.5,.5) to [out=0,in=270] (.75,1);
\path[fill=red,opacity=.2] (-.75,1) to [out=270,in=180] (-.5,.5)
	arc[start angle=180, end angle=360,x radius=.5,y radius=0.09375] to [out=0,in=270] 
		(.75,1) arc[start angle=360, end angle=180,x radius=.75,y radius=0.28125];
\path[fill=red,opacity=.2] (-.5,.5) arc[start angle=180, end angle=360,x radius=.5,y radius=0.09375] 
	to [out=180,in=270] (.25,1) arc[start angle=360, end angle=180,x radius=.25,y radius=0.09375]
		to [out=270,in=0] (-.5,.5);
\node[gray] at (0,1) {$\bullet$};
\draw[gray] (0,1) ellipse (1 and 0.375);
\end{tikzpicture} 
\]
The induced functor $\Mat(\dTL) \to \ABN$ is an equivalence of
categories.
\end{prop}

\begin{proof}
Fullness and essential surjectivity (after passing to $\Mat(\dTL)$) 
follow as in \cite[Propositions 3.8 and 4.2]{QR2}, 
which establish the analogous result in the setting of $\glnn{2}$ foams. 
Faithfulness follows from \cite[Theorem 2.1]{Russell}.
\end{proof}

\begin{remark}
We could have worked in greater generality throughout \S \ref{ss:dTL} and in the following. 
Generalizing Definition~\ref{def:dTL}, for graded commutative ring $R$ and fixed
elements $h\in R_2$, $t\in R_4$, one can also consider the \emph{equivariant}
$R$-linear dotted Temperley--Lieb category $\dTL^R$, defined analogously as
above, but subject to the relations:
\begin{equation}
	\label{eq:eqdTLrels}
\begin{tikzpicture}[anchorbase,scale=.75]
\draw[very thick] (0,0) circle (.5);
\end{tikzpicture} = 2 \,\, , \quad
\begin{tikzpicture}[anchorbase,scale=1]
\draw[very thick] (0,0) to node{$\bullet$} node[right=-1pt,yshift=3pt]{\scriptsize$2$} (0,1);
\end{tikzpicture} 
:=
h 
\begin{tikzpicture}[anchorbase,scale=1]
\draw[very thick] (0,0) to node[pos=.5]{$\bullet$}  (0,1);
\end{tikzpicture}
+t\,
\begin{tikzpicture}[anchorbase,scale=1]
	\draw[very thick] (0,0) to (0,1);
	\end{tikzpicture} 
 \, , \quad
\begin{tikzpicture}[anchorbase,scale=1]
\draw[very thick] (0,0) to node{$\bullet$}(0,1);
\draw[very thick] (.5,0) to (.5,1);
\end{tikzpicture} 
\, + \,
\begin{tikzpicture}[anchorbase,scale=1]
\draw[very thick] (0,0) to (0,1);
\draw[very thick] (.5,0) to node{$\bullet$} (.5,1);
\end{tikzpicture} 
 -h \,\,
\begin{tikzpicture}[anchorbase,scale=1]
\draw[very thick] (0,0) to (0,1);
\draw[very thick] (.5,0) to (.5,1);
\end{tikzpicture} 
=
\begin{tikzpicture}[anchorbase,scale=1]
\draw[very thick] (0,0) to [out=90,in=180] (.25,.375) node{$\bullet$} to [out=0,in=90] (.5,0);
\draw[very thick] (0,1) to [out=270,in=180] (.25,.625) to [out=0,in=270] (.5,1);
\end{tikzpicture} 
+
\begin{tikzpicture}[anchorbase,scale=1]
\draw[very thick] (0,0) to [out=90,in=180] (.25,.375) to [out=0,in=90] (.5,0);
\draw[very thick] (0,1) to [out=270,in=180] (.25,.625) node{$\bullet$} to [out=0,in=270] (.5,1);
\end{tikzpicture} 
- h\,
\begin{tikzpicture}[anchorbase,scale=1]
\draw[very thick] (0,0) to [out=90,in=180] (.25,.375) to [out=0,in=90] (.5,0);
\draw[very thick] (0,1) to [out=270,in=180] (.25,.625) to [out=0,in=270] (.5,1);
\end{tikzpicture} 
\end{equation} 
If $2$ is invertible, the third relation implies that the dotted circle
evaluates to $h$. 
All of the above results carry over to the equivariant setting, provided we work
with analogously generalized categories $\BN^R(\Sigma ; \mathbf{p})$.
\end{remark}

\subsection{Basic structure of $\dTL$}
\label{ss:basic struct}

Since $\dTL$ is a $\Z$-graded category, its center $Z(\dTL)$ 
(i.e.~the endo-natural transformations of the identity functor) is a $\Z$-graded algebra.  

\begin{lem}
	\label{lemma:center of dTL}
There is a map of graded algebras $\K[s,z]/(s^2=1)\rightarrow Z(\dTL)$ defined by 
\[
z|_{\gen^n} := \begin{tikzpicture}[anchorbase]
\draw[very thick] (0,0) to (0,.5);
\draw[very thick] (.3,0) to (.3,.5);
\draw[very thick] (.8,0) to (.8,.5);
\node at (.55,.25) {$\mydots$};
\node at (0,.25) {$\bullet$};
\end{tikzpicture}
\ -\
 \begin{tikzpicture}[anchorbase]
\draw[very thick] (0,0) to (0,.5);
\draw[very thick] (.3,0) to (.3,.5);
\draw[very thick] (.8,0) to (.8,.5);
\node at (.55,.25) {$\mydots$};
\node at (.3,.25) {$\bullet$};
\end{tikzpicture}
\ + \cdots  +(-1)^{n-1}  
\begin{tikzpicture}[anchorbase]
\draw[very thick] (0,0) to (0,.5);
\draw[very thick] (.3,0) to (.3,.5);
\draw[very thick] (.8,0) to (.8,.5);
\node at (.55,.25) {$\mydots$};
\node at (.8,.25) {$\bullet$};
\end{tikzpicture}
\]
\end{lem}
and $s|_{\gen^n}=(-1)^n\id_{\gen^n}$.
\begin{proof}
Straightforward.
\end{proof}

In Corollary \ref{cor:center of dTL} we prove that this algebra map is an isomorphism.

\begin{remark}\label{rmk:too many z's}
We will use the abbreviation $z_n:= z|_{\gen^n}$.
We have $z_n^{n+1}=0$ because each of the $n$ strands in $\id_{\gen^n}$ can carry at most one dot.
\end{remark}

\begin{remark}\label{rmk:involution sign} The involution $\dTL\rightarrow \dTL$
given by reflecting all diagrams across a vertical axis sends $z_n\mapsto
(-1)^{n-1}z_n$, i.e.~$z\mapsto -sz$.
\end{remark}

The object $\gen^n\in \dTL$ carries an action of the symmetric group $\mathfrak{S}_n$, 
by defining 
\begin{equation}
	\label{eqn:TLbraiding}
\begin{tikzpicture}[anchorbase]
\draw[very thick] (0,0) to (.5,.5);
\draw[very thick] (.5,0) to (0,.5);
\end{tikzpicture}
:=
\begin{tikzpicture}[anchorbase]
\draw[very thick] (0,0) to (0,.5);
\draw[very thick] (.5,0) to (.5,.5);
\end{tikzpicture}
-
\begin{tikzpicture}[anchorbase]
\draw[very thick] (0,0) to [in=135,out=45](.5,0);
\draw[very thick] (0,.5) to [in=-135, out=-45](.5,.5);
\end{tikzpicture}
\end{equation}
However,
this does not quite make $\dTL$ into a symmetric monoidal category, because dots slide only up to a sign:
\[
\begin{tikzpicture}[anchorbase]
\draw[very thick] (0,0) to (.5,.5);
\draw[very thick] (.5,0) to (0,.5);
\node at (.37,.37) {$\bullet$} ;
\end{tikzpicture}
=
-
\begin{tikzpicture}[anchorbase]
\draw[very thick] (0,0) to (.5,.5);
\draw[very thick] (.5,0) to (0,.5);
\node at (.12,.12){$\bullet$};
\end{tikzpicture} \; , \quad 
\begin{tikzpicture}[anchorbase,xscale=-1]
\draw[very thick] (0,0) to (.5,.5);
\draw[very thick] (.5,0) to (0,.5);
\node at (.37,.37) {$\bullet$} ;
\end{tikzpicture}
=
-
\begin{tikzpicture}[anchorbase,xscale=-1]
\draw[very thick] (0,0) to (.5,.5);
\draw[very thick] (.5,0) to (0,.5);
\node at (.12,.12){$\bullet$};
\end{tikzpicture}
\]
It is sometimes convenient to normalize away this sign. 

\begin{conv}\label{conv:End}
Let $\ux_i\in \End_{\dTL}(\gen^n)$ denote a dot on the $i$-th strand 
and set $\sx_i := (-1)^{i-1}\ux_i$,  i.e.~
\[
\sx_i = (-1)^{i-1} \ux_i = (-1)^{i-1}  
\begin{tikzpicture}[anchorbase,xscale=.8]
\draw[double, very thick] (0,0) node[below=-2pt]{\tiny$i{-}1$}  to (0,1);
\draw[very thick] (.5,0) to node{$\bullet$} (.5,1);
\draw[double,very thick] (1,0) node[below=-2pt]{\tiny$n{-}i$} to (1,1);
\end{tikzpicture}\, 
\quad \text{where} \quad 
\begin{tikzpicture}[anchorbase,scale=1]
\draw[double, very thick] (0,0) node[below=-2pt]{\tiny$k$} to (0,1);
\end{tikzpicture}
:=
\underbrace{
\begin{tikzpicture}[baseline=.8em,scale=1]
\draw[very thick] (0,0) to (0,1);
\node at (.25,.5) {$\mydots$};
\draw[very thick] (.5,0) to (.5,1);
\end{tikzpicture}}_k \, .
\]
We will sometimes also denote
$f \in \K[\ux_1,\ldots,\ux_n]/(\ux_i^2) = \K[\sx_1,\ldots,\sx_n]/(\sx_i^2) \subset \End_{\dTL}(\gen^n)$ by
\[
\begin{tikzpicture}[anchorbase,scale=1]
\draw[double, very thick] (0,0) to node[pos=.85,left=-1pt]{\scriptsize $n$} 
	node[pos=.45]{\tiny $\CQGbox{f}$} (0,1);
\end{tikzpicture} \, .
\]
Note that $w f(\sx_1,\ldots,\sx_n) w\inv= f(\sx_{w(1)},\ldots,\sx_{w(n)})$ for all $w\in \mathfrak{S}_n$ 
and all $f \in \K[\sx_1,\ldots,\sx_n]/(\sx_i^2)$, 
and $z_n = \sx_1+\cdots+\sx_n$. 
\end{conv}

\subsection{The symmetrizing idempotent}
\label{ss:JW}

\begin{definition}\label{def:JW} Let $\JW_n\in \End_{\dTL}(\gen^n)$ denote the
symmetrizing idempotent, also known as the $n^{th}$ 
\emph{Jones--Wenzl projector} 
(at circle-value $2$),
defined by
\[
\JW_n := 
\begin{tikzpicture}[anchorbase,scale=.75]
	\draw[very thick] (-.25,-.75) to (-.25,.75);
	\node at (0,-.5){\mysdots};
	\node at (0,.5){\mysdots};
	\draw[very thick] (.25,-.75) to (.25,.75);
	\filldraw[white] (-.5,.25) rectangle (.5,-.25); 
	\draw[very thick] (-.5,.25) rectangle (.5,-.25);
	\node at (0,0) {\scriptsize$\JW_{n}$};
\end{tikzpicture}
:=\frac{1}{n!} \sum_{w\in \mathfrak{S}_n} 
\begin{tikzpicture}[anchorbase,scale=.75]
	\draw[very thick] (-.25,-.75) to (-.25,.75);
	\node at (0,-.5){\mysdots};
	\node at (0,.5){\mysdots};
	\draw[very thick] (.25,-.75) to (.25,.75);
	\filldraw[white] (-.5,.25) rectangle (.5,-.25); 
	\draw[very thick] (-.5,.25) rectangle (.5,-.25);
	\node at (0,0) {$w$};
\end{tikzpicture} \, .
\]
\end{definition}

The Jones--Wenzl projector $\JW_n$ is uniquely characterized by the
following properties:
\begin{itemize}
	\item $\JW_n^2 = \JW_n$, 
	\item $\JW_n w =\JW_n = w\JW_n$ for all  $w\in \mathfrak{S}_n$ (equivalently, $\JW_n$ ``kills turnbacks''),  and
	\item $\JW_n \neq 0$
\end{itemize}
and admits an inductive description via
\begin{equation}
	\label{eq:JWrecur}
\begin{tikzpicture}[anchorbase,scale=1]
	\draw[very thick] (-.25,-.75) to (-.25,.75);
	\node at (0,-.5){\mysdots};
	\node at (0,.5){\mysdots};
	\draw[very thick] (.25,-.75) to (.25,.75);
	\filldraw[white] (-.5,.25) rectangle (.5,-.25); 
	\draw[very thick] (-.5,.25) rectangle (.5,-.25);
	\node at (0,0) {\scriptsize$\JW_{n}$};
\end{tikzpicture}
=
\begin{tikzpicture}[anchorbase,scale=1]
	\draw[very thick] (-.25,-.75) to (-.25,.75);
	\node at (0,-.5){\mysdots};
	\node at (0,.5){\mysdots};
	\draw[very thick] (.25,-.75) to (.25,.75);
	\filldraw[white] (-.5,.25) rectangle (.5,-.25); 
	\draw[very thick] (-.5,.25) rectangle (.5,-.25);
	\node at (0,0) {\scriptsize$\JW_{n-1}$};
	\draw[very thick] (.75,-.75) to (.75,.75);
\end{tikzpicture}
-
\frac{n-1}{n}
\begin{tikzpicture}[anchorbase,scale=.75]
	\draw[very thick] (-.75,-1.5) to (-.75,1.5);
	\node at (-.5,-1.25){\mysdots};
	\node at (-.5,0){\mysdots};
	\node at (-.5,1.25){\mysdots};
	\draw[very thick] (-.25,-1.5) to (-.25,1.5);
	\draw[very thick] (.25,1.5) to (.25,.5) to [out=270,in=180] (.5,.125)
		to [out=0,in=270] (.75,.5) to (.75,1.5);
	\draw[very thick] (.25,-1.5) to (.25,-.5) to [out=90,in=180] (.5,-.125)
		to [out=0,in=90] (.75,-.5) to (.75,-1.5);
	\filldraw[white] (-1,.5) rectangle (.5,1); 
	\draw[very thick] (-1,.5) rectangle (.5,1);
	\node at (-.25,.75) {\scriptsize$\JW_{n-1}$};
	\filldraw[white] (-1,-.5) rectangle (.5,-1); 
	\draw[very thick] (-1,-.5) rectangle (.5,-1);
	\node at (-.25,-.75) {\scriptsize$\JW_{n-1}$};
\end{tikzpicture} \, .
\end{equation}
For more details, see e.g.~\cite{MR1280463}.

\begin{lem}
	\label{lem:JWrels}
The following relations hold in $\dTL$:
\begin{equation}
	\label{eq:PxP}
\JW_n \ux_i \JW_n = -\JW_n \ux_{i+1} \JW_n \quad \text{for } 1 \leq i \leq n-1
\end{equation}
\begin{equation}\label{eq:PxxxP}
\JW_n \ux_1\cdots \ux_k \JW_n = \begin{tikzpicture}[anchorbase,yscale=.75]
	\draw[very thick] (-.75,-1.5) to node {$\bullet$} (-.75,.5);
	\draw[very thick] (-.25,-1.5) to node {$\bullet$} (-.25,.5);
	\draw[double, very thick] (0,-1.5) to node[right]{\scriptsize $n-k$} (0,.5);
	\node at (-.5,-1.375){\mysdots};
	\node at (-.5,-.5){\mysdots};
	\node at (-.5,.375){\mysdots};
	\filldraw[white] (-1,.25) rectangle (.25,-.25); 
	\draw[very thick] (-1,.25) rectangle (.25,-.25);
	\node at (-.375,0) {\scriptsize$\JW_{n}$};
	\filldraw[white] (-1,-.75) rectangle (.25,-1.25); 
	\draw[very thick] (-1,-.75) rectangle (.25,-1.25);
	\node at (-.375,-1) {\scriptsize$\JW_{n}$};
\end{tikzpicture}
=  \frac{(-1)^{\binom{k}{2}}(n-k)!}{n!} z_n^k \JW_n
\end{equation}
\begin{equation}
	\label{eq:dAwd}
\begin{tikzpicture}[anchorbase,scale=.75]
	\draw[very thick,double] (-.625,-.75) to (-.625,.75);
	\node at (-.875,-.5) {\tiny$k$};
	\draw[very thick] (-.25,.75) to (-.25,-.25) to [out=270,in=180] (0,-.625)
		node{$\bullet$} to [out=0,in=270] (.25,-.25) to (.25,.75);
	\draw[very thick,double] (.625,-.75) to (.625,.75);
	\node at (1.375,-.5) {\tiny$n{-}k{-}2$};
	\filldraw[white] (-1,.25) rectangle (1,-.25); 
	\draw[very thick] (-1,.25) rectangle (1,-.25);
	\node at (0,0) {\scriptsize$\JW_n$};
\end{tikzpicture}
=
\begin{tikzpicture}[anchorbase,scale=.75]
	\draw[very thick,double] (-.625,-.75) to (-.625,.75);
	\node at (-1.125,-.5) {\tiny$k{+}1$};
	\draw[very thick] (-.25,.75) to (-.25,-.25) to [out=270,in=180] (0,-.625)
		node{$\bullet$} to [out=0,in=270] (.25,-.25) to (.25,.75);
	\draw[very thick,double] (.625,-.75) to (.625,.75);
	\node at (1.375,-.5) {\tiny$n{-}k{-}3$};
	\filldraw[white] (-1,.25) rectangle (1,-.25); 
	\draw[very thick] (-1,.25) rectangle (1,-.25);
	\node at (0,0) {\scriptsize$\JW_n$};
\end{tikzpicture}
\quad \text{for } 0 \leq k \leq n-3 \quad \text{(and its vertical reflection)}
\end{equation}
\begin{equation}
	\label{eq:tracedot}
\begin{tikzpicture}[anchorbase,scale=1]
	\node at (-.5,.5){\tiny$n{-}1$};
	\draw[very thick,double] (-.125,-.75) to (-.125,.75);
	\draw[very thick] (.25,.25) to [out=90,in=180] (.5,.5) to [out=0,in=90] (.75,.25)
		 to node{$\bullet$} (.75,-.25) to [out=270,in=0] (.5,-.5) 
		to [out=180,in=270] (.25,-.25);
	\filldraw[white] (-.5,.25) rectangle (.5,-.25); 
	\draw[very thick] (-.5,.25) rectangle (.5,-.25);
	\node at (0,0) {\scriptsize$\JW_{n}$};
\end{tikzpicture}
=
- \frac{n-1}{n}
\begin{tikzpicture}[anchorbase,scale=.75]
	\draw[very thick] (-.75,-1.5) to (-.75,.5);
	\node at (-.5,-1.375){\mysdots};
	\node at (-.5,-.5){\mysdots};
	\node at (-.5,.375){\mysdots};
	\draw[very thick] (-.25,-1.5) to (-.25,.5);
	\draw[very thick] (0,-1.5) to node{$\bullet$} (0,.5);
	\filldraw[white] (-1,.25) rectangle (.25,-.25); 
	\draw[very thick] (-1,.25) rectangle (.25,-.25);
	\node at (-.375,0) {\scriptsize$\JW_{n-1}$};
	\filldraw[white] (-1,-.75) rectangle (.25,-1.25); 
	\draw[very thick] (-1,-.75) rectangle (.25,-1.25);
	\node at (-.375,-1) {\scriptsize$\JW_{n-1}$};
\end{tikzpicture}
=
\frac{(-1)^{n-1}}{n} z_{n-1}\JW_{n-1}.
\end{equation}
\end{lem}
\begin{proof}
Throughout the proof we use the auxiliary variables $\sx_i = (-1)^{i-1}\ux_i$
from Convention~\ref{conv:End}, as well as their straightforward interaction
with permutations $w\in \mathfrak{S}_n$. In particular, we will make repeated
use of the relation
\begin{equation}\label{eq:signed x sandwiches}
\JW_n f(\sx_1,\ldots,\sx_n) \JW_n = \JW_n wf(\sx_1,\ldots,\sx_n) w\inv \JW_n = \JW_n f(\sx_{w(1)},\ldots,\sx_{w(n)}) \JW_n
\end{equation}
which holds since $\JW_n$ absorbs permutations.
Equation \eqref{eq:PxP} is now immediate.

For \eqref{eq:PxxxP}, recall that $\xs_n = \sx_1+\cdots+\sx_n$, so $\xs_n^k = k! \sum_{i_1<\cdots < i_k} \sx_{i_1}\cdots \sx_{i_k}$.  
Now, compute:
\[
z_n^k\JW_n = \JW_n z_n^k \JW_n = k!\sum_{i_1<\cdots < i_k} \JW_n \sx_{i_1}\cdots \sx_{i_k} \JW_n = k!\binom{n}{k} \JW_n \sx_{1}\cdots \sx_{k} \JW_n. 
\]
In the first equality we use that $\xs_n$ is central and that $\JW_n^2 = \JW_n$, 
and in the last equality we use \eqref{eq:signed x sandwiches}.   
Accounting for the signs gives \eqref{eq:PxxxP}, 
since $\sx_1\cdots \sx_k = (-1)^{\binom{k}{2}} \ux_1\cdots \ux_k$.

For \eqref{eq:dAwd}, it suffices to consider the $n=3$ case, 
by the $\ell=3$ case of the ``absorption'' relation:
\begin{equation}
	\label{eq:absorb}
(\id_{i-1} \otimes \JW_\ell \otimes \id_{n-i-\ell+1}) \JW_n = \JW_n = 
\JW_n (\id_{i-1} \otimes \JW_\ell \otimes \id_{n-i-\ell+1}) \, .
\end{equation}
We compute using the last relation in \eqref{eq:dTLrels} that
\[
\begin{tikzpicture}[anchorbase,scale=.75]
	\draw[very thick] (-.5,.75) to (-.5,-.25) to [out=270,in=180] (-.25,-.625)
		node{$\bullet$} to [out=0,in=270] (0,-.25) to (0,.75);
	\draw[very thick] (.5,-.75) to (.5,.75);
	\filldraw[white] (-1,.25) rectangle (1,-.25); 
	\draw[very thick] (-1,.25) rectangle (1,-.25);
	\node at (0,0) {\scriptsize$\JW_3$};
\end{tikzpicture}
=
- \,
\begin{tikzpicture}[anchorbase,scale=.75]
	\draw[very thick] (-.5,.75) to (-.5,-.25) to [out=270,in=180] (-.25,-.625)
		to [out=0,in=270] (0,-.25) to (0,.75);
	\draw[very thick] (.5,-.75) to node[pos=.175]{$\bullet$} (.5,.75);
	\filldraw[white] (-1,.25) rectangle (1,-.25); 
	\draw[very thick] (-1,.25) rectangle (1,-.25);
	\node at (0,0) {\scriptsize$\JW_3$};
\end{tikzpicture}
\, + \,
\begin{tikzpicture}[anchorbase,scale=.75,xscale=-1]
	\draw[very thick] (-.5,.75) to (-.5,-.25) to [out=270,in=180] (-.25,-.625)
		to [out=0,in=270] (0,-.25) to (0,.75);
	\draw[very thick] (.5,-.75) to node[pos=.175]{$\bullet$} (.5,.75);
	\filldraw[white] (-1,.25) rectangle (1,-.25); 
	\draw[very thick] (-1,.25) rectangle (1,-.25);
	\node at (0,0) {\scriptsize$\JW_3$};
\end{tikzpicture}
\, + \,
\begin{tikzpicture}[anchorbase,scale=.75,xscale=-1]
	\draw[very thick] (-.5,.75) to (-.5,-.25) to [out=270,in=180] (-.25,-.625)
		node{$\bullet$} to [out=0,in=270] (0,-.25) to (0,.75);
	\draw[very thick] (.5,-.75) to (.5,.75);
	\filldraw[white] (-1,.25) rectangle (1,-.25); 
	\draw[very thick] (-1,.25) rectangle (1,-.25);
	\node at (0,0) {\scriptsize$\JW_3$};
\end{tikzpicture}
=
\begin{tikzpicture}[anchorbase,scale=.75,xscale=-1]
	\draw[very thick] (-.5,.75) to (-.5,-.25) to [out=270,in=180] (-.25,-.625)
		node{$\bullet$} to [out=0,in=270] (0,-.25) to (0,.75);
	\draw[very thick] (.5,-.75) to (.5,.75);
	\filldraw[white] (-1,.25) rectangle (1,-.25); 
	\draw[very thick] (-1,.25) rectangle (1,-.25);
	\node at (0,0) {\scriptsize$\JW_3$};
\end{tikzpicture} \, .
\]
Lastly,
\[
\begin{tikzpicture}[anchorbase,scale=1]
	\node at (-.5,.5){\tiny$n{-}1$};
	\draw[very thick,double] (-.125,-.75) to (-.125,.75);
	\draw[very thick] (.25,.25) to [out=90,in=180] (.5,.5) to [out=0,in=90] (.75,.25)
		 to node{$\bullet$} (.75,-.25) to [out=270,in=0] (.5,-.5) 
		to [out=180,in=270] (.25,-.25);
	\filldraw[white] (-.5,.25) rectangle (.5,-.25); 
	\draw[very thick] (-.5,.25) rectangle (.5,-.25);
	\node at (0,0) {\scriptsize$\JW_{n}$};
\end{tikzpicture}
\stackrel{{\eqref{eq:JWrecur}}}{=}
\begin{tikzpicture}[anchorbase,scale=1]
	\draw[very thick] (-.25,-.75) to (-.25,.75);
	\node at (0,-.5){\mysdots};
	\node at (0,.5){\mysdots};
	\draw[very thick] (.25,-.75) to (.25,.75);
	\filldraw[white] (-.5,.25) rectangle (.5,-.25); 
	\draw[very thick] (-.5,.25) rectangle (.5,-.25);
	\node at (0,0) {\scriptsize$\JW_{n-1}$};
	\draw[very thick] (1,0) circle (.25);
	\node at (1.25,0) {$\bullet$};
\end{tikzpicture}
- \frac{n-1}{n}
\begin{tikzpicture}[anchorbase,scale=.75]
	\draw[very thick] (-.75,-1.75) to (-.75,.75);
	\node at (-.5,-1.5){\mysdots};
	\node at (-.5,-.5){\mysdots};
	\node at (-.5,.5){\mysdots};
	\draw[very thick] (-.25,-1.75) to (-.25,.75);
	\draw[very thick] (0,-1.75) to node{$\bullet$} (0,.75);
	\filldraw[white] (-1,.25) rectangle (.25,-.25); 
	\draw[very thick] (-1,.25) rectangle (.25,-.25);
	\node at (-.375,0) {\scriptsize$\JW_{n-1}$};
	\filldraw[white] (-1,-.75) rectangle (.25,-1.25); 
	\draw[very thick] (-1,-.75) rectangle (.25,-1.25);
	\node at (-.375,-1) {\scriptsize$\JW_{n-1}$};
\end{tikzpicture}
\]
which gives \eqref{eq:tracedot} since the relations in \eqref{eq:dTLrels}
imply that a dotted circle is zero.
\end{proof}

\subsection{The Karoubi envelope of $\dTL$}
In this section, 
we give an explicit presentation of $\Kar(\dTL)$ in terms of the objects $\JW_n$.
Recall that the Karoubi envelope of the usual Temperley-Lieb category $\TL$ is 
generated by the (images of the) Jones-Wenzl idempotents $\JW_n$.  
Further, these objects are simple in $\Kar(\TL)$, with one-dimensional endomorphism algebras, 
and no nonzero morphisms $\JW_n\rightarrow \JW_m$ for $n\neq m$.  
The situation for $\dTL$ is quite different: 
there are many interesting morphisms in $\Kar(\dTL)$ between $\JW_n$ 
for various $n$.

Motivated by Lemma \ref{lem:JWrels}, 
we now define certain morphisms in $\Kar(\dTL)$.

\begin{definition}
	\label{def:KarGens}
For each $n \geq 0$, let $U_n\colon \symobj_n\to \symobj_{n+2}$ and
$D_{n+2}\colon \symobj_{n+2}\to \symobj_{n}$ be the following morphisms in $\Kar(\dTL)$:
\[
U_n := 
\begin{tikzpicture}[anchorbase,scale=.75]
	\draw[very thick] (-.75,-1.75) to (-.75,.75);
	\node at (-.5,-1.5){\mysdots};
	\node at (-.5,-.5){\mysdots};
	\node at (-.5,.5){\mysdots};
	\draw[very thick] (-.25,-1.75) to (-.25,.75);
	\draw[very thick] (.25,.75) to (.25,-.25) to [out=270,in=180] (.5,-.625)
		node{$\bullet$} to [out=0,in=270] (.75,-.25) to (.75,.75);
	\filldraw[white] (-1,.25) rectangle (1,-.25); 
	\draw[very thick] (-1,.25) rectangle (1,-.25);
	\node at (0,0) {\scriptsize$\JW_{n+2}$};
	\filldraw[white] (-1,-.75) rectangle (0,-1.25); 
	\draw[very thick] (-1,-.75) rectangle (0,-1.25);
	\node at (-.5,-1) {\scriptsize$\JW_{n}$};
\end{tikzpicture}
\; , \quad
D_{n+2} := 
(n+2)(n+1)
\begin{tikzpicture}[anchorbase,scale=.75,yscale=-1]
	\draw[very thick] (-.75,-1.75) to (-.75,.75);
	\node at (-.5,-1.5){\mysdots};
	\node at (-.5,-.5){\mysdots};
	\node at (-.5,.5){\mysdots};
	\draw[very thick] (-.25,-1.75) to (-.25,.75);
	\draw[very thick] (.25,.75) to (.25,-.25) to [out=270,in=180] (.5,-.625)
		node{$\bullet$} to [out=0,in=270] (.75,-.25) to (.75,.75);
	\filldraw[white] (-1,.25) rectangle (1,-.25); 
	\draw[very thick] (-1,.25) rectangle (1,-.25);
	\node at (0,0) {\scriptsize$\JW_{n+2}$};
	\filldraw[white] (-1,-.75) rectangle (0,-1.25); 
	\draw[very thick] (-1,-.75) rectangle (0,-1.25);
	\node at (-.5,-1) {\scriptsize$\JW_{n}$};
\end{tikzpicture}
\, .
\]
Let $\xs_{n}\colon \symobj_n \to \symobj_{n}$ 
denote (by slight abuse of notation)
the element of $\End_{\dTL}(\JW_n)$ given by the central element 
$z$ from Lemma \ref{lemma:center of dTL}.
\end{definition}

\begin{lem}
	\label{lem:KarRels}
The morphisms in Definition \ref{def:KarGens} satisfy the following relations:
\[
\xs_{n}^{n+1}=0
\, , \quad
U_{n-2} D_{n} = - \xs_{n}^2 = D_{n}U_{n-2} 
\, , \quad 
\xs_{n} U_{n-2} = U_{n-2} \xs_{n-2}
\, , \quad 
\xs_{n} D_{n+2} = D_{n+2} \xs_{n+2}
\]
\end{lem}
\begin{proof}
The first relation is immediate (see Remark \ref{rmk:too many z's}).
Next, we have
\[
U_{n-2}D_{n} =
n(n-1)
\begin{tikzpicture}[anchorbase,scale=.75]
	\draw[very thick] (-.75,-1.5) to (-.75,1.5);
	\node at (-.5,-1.25){\mysdots};
	\node at (-.5,0){\mysdots};
	\node at (-.5,1.25){\mysdots};
	\draw[very thick] (-.25,-1.5) to (-.25,1.5);
	\draw[very thick] (.25,1.5) to (.25,.5) to [out=270,in=180] (.5,.125)
		node{$\bullet$} to [out=0,in=270] (.75,.5) to (.75,1.5);
	\draw[very thick] (.25,-1.5) to (.25,-.5) to [out=90,in=180] (.5,-.125)
		node {$\bullet$} to [out=0,in=90] (.75,-.5) to (.75,-1.5);
	\filldraw[white] (-1,.5) rectangle (1,1); 
	\draw[very thick] (-1,.5) rectangle (1,1);
	\node at (0,.75) {\scriptsize$\JW_{n}$};
	\filldraw[white] (-1,-.5) rectangle (1,-1); 
	\draw[very thick] (-1,-.5) rectangle (1,-1);
	\node at (0,-.75) {\scriptsize$\JW_{n}$};
\end{tikzpicture}
\stackrel{{\eqref{eq:twodotswitch}}}{=}
n(n-1)
\begin{tikzpicture}[anchorbase,scale=.75]
	\draw[very thick] (-.75,-1.5) to (-.75,1.5);
	\node at (-.5,-1.25){\mysdots};
	\node at (-.5,0){\mysdots};
	\node at (-.5,1.25){\mysdots};
	\draw[very thick] (-.25,-1.5) to (-.25,1.5);
	\draw[very thick] (.25,-1.5) to node{$\bullet$} (.25,1.5);
	\draw[very thick] (.75,-1.5) to node{$\bullet$} (.75,1.5);
	\filldraw[white] (-1,.5) rectangle (1,1); 
	\draw[very thick] (-1,.5) rectangle (1,1);
	\node at (0,.75) {\scriptsize$\JW_{n}$};
	\filldraw[white] (-1,-.5) rectangle (1,-1); 
	\draw[very thick] (-1,-.5) rectangle (1,-1);
	\node at (0,-.75) {\scriptsize$\JW_{n}$};
\end{tikzpicture}
\stackrel{\eqref{eq:PxxxP}}{=}-\xs_{n}^2 \, .
\]
Next, we compute
\[
D_nU_{n-2} = 
n(n-1)
\begin{tikzpicture}[anchorbase,scale=.75]
	\draw[very thick] (-.75,-1) to (-.75,1);
	\node at (-.5,-.5){\mysdots};
	\node at (-.5,.5){\mysdots};
	\draw[very thick] (-.25,-1) to (-.25,1);
	\draw[very thick] (.25,-.25) to [out=270,in=180] (.5,-.625)
		node{$\bullet$} to [out=0,in=270] (.75,-.25)
		to (.75,.25) to [out=90,in=0] (.5,.625)
		node{$\bullet$} to [out=180,in=90] (.25,.25) to (.25,-.25);
	\filldraw[white] (-1,.25) rectangle (1,-.25); 
	\draw[very thick] (-1,.25) rectangle (1,-.25);
	\node at (0,0) {\scriptsize$\JW_{n}$};
\end{tikzpicture}
\stackrel{{\eqref{eq:JWrecur}}}{=}
n(n-1)
\begin{tikzpicture}[anchorbase,scale=.75]
	\draw[very thick] (-.75,-1) to (-.75,1);
	\node at (-.5,-.5){\mysdots};
	\node at (-.5,.5){\mysdots};
	\draw[very thick] (-.25,-1) to (-.25,1);
	\draw[very thick] (.25,-.25) to [out=270,in=180] (.5,-.625)
		node{$\bullet$} to [out=0,in=270] (.75,-.25)
		to (.75,.25) to [out=90,in=0] (.5,.625)
		node{$\bullet$} to [out=180,in=90] (.25,.25) to (.25,-.25);
	\filldraw[white] (-1,.25) rectangle (.5,-.25); 
	\draw[very thick] (-1,.25) rectangle (.5,-.25);
	\node at (-.25,0) {\scriptsize$\JW_{n-1}$};
\end{tikzpicture}
-
(n-1)^2
\begin{tikzpicture}[anchorbase,scale=.75]
	\draw[very thick] (-.75,-1.5) to (-.75,1.5);
	\node at (-.5,-1.25){\mysdots};
	\node at (-.5,0){\mysdots};
	\node at (-.5,1.25){\mysdots};
	\draw[very thick] (-.25,-1.5) to (-.25,1.5);
	\draw[very thick] (.25,.5) to [out=270,in=180] (.5,.125)
		to [out=0,in=270] (.75,.5) to node{$\bullet$} (.75,1) 
		to [out=90,in=0] (.5,1.375) to [out=180,in=90] (.25,1);
	\draw[very thick] (.25,-.5) to [out=90,in=180] (.5,-.125)
		to [out=0,in=90] (.75,-.5) to node{$\bullet$} (.75,-1)
		to [out=270,in=0] (.5,-1.375) to [out=180,in=270] (.25,-1);
	\filldraw[white] (-1,.5) rectangle (.5,1); 
	\draw[very thick] (-1,.5) rectangle (.5,1);
	\node at (-.25,.75) {\scriptsize$\JW_{n-1}$};
	\filldraw[white] (-1,-.5) rectangle (.5,-1); 
	\draw[very thick] (-1,-.5) rectangle (.5,-1);
	\node at (-.25,-.75) {\scriptsize$\JW_{n-1}$};
\end{tikzpicture}
\stackrel{{\eqref{eq:tracedot}}}{=}
- \xs_{n-2}^2 \, .
\]
Lastly, the final two relations follow from centrality of $z$.
\end{proof}

Using Lemma \ref{lem:KarRels} (and Proposition \ref{prop:PolyRep} below), 
we obtain a generators-and-relations presentation for $\Mat(\Kar(\dTL))$ as 
a $\Z$-additive $\K$-linear category.

\begin{thm}\label{thm:KardTL} The category $\Mat(\Kar(\dTL))$ is equivalent to
	the category of finitely-generated graded projective modules for the path algebra of the quiver
	(with two connected components)
	\begin{gather*}
		\begin{tikzcd}[ampersand replacement=\&,column sep=3em]  
			\mathsf{P}_0
		\ar[r,"\mathsf{U}_0",black,yshift=0.1cm,bend left=10]
		\&
		\mathsf{P}_2
		\arrow[out=120,in=60,loop,"\mathsf{z}_2"]
		\ar[r,"\mathsf{U}_2",black,yshift=0.1cm,bend left=10]
		\ar[l,"D_2",yshift=-0.1cm,bend left=10]
	   \&
	   \mathsf{P}_4
	   \arrow[out=120,in=60,loop,"\mathsf{z}_4"]
	   \ar[r,"\mathsf{U}_4",black,yshift=0.1cm,bend left=10]
	   \ar[l,"\mathsf{D}_4",yshift=-0.1cm,bend left=10]
	   \&
	   \cdots
	   \ar[l,"\mathsf{D}_6",yshift=-0.1cm,bend left=10]
		\end{tikzcd}
	\quad , \qquad
		\begin{tikzcd}[ampersand replacement=\&,column sep=3em]  
			\mathsf{P}_1
			\arrow[out=120,in=60,loop,"\mathsf{z}_1"]
			\ar[r,"\mathsf{U}_1",black,yshift=0.1cm,bend left=10]
			\&
			\mathsf{P}_3
			\arrow[out=120,in=60,loop,"\mathsf{z}_3"]
			\ar[r,"\mathsf{U}_3",black,yshift=0.1cm,bend left=10]
			\ar[l,"\mathsf{D}_3",yshift=-0.1cm,bend left=10]
		   \&
		   \mathsf{P}_5
		   \arrow[out=120,in=60,loop,"\mathsf{z}_5"]
		   \ar[r,"\mathsf{U}_5",black,yshift=0.1cm,bend left=10]
		   \ar[l,"\mathsf{D}_5",yshift=-0.1cm,bend left=10]
		   \&
		   \cdots
		   \ar[l,"\mathsf{D}_7",yshift=-0.1cm,bend left=10]
			\end{tikzcd}
		\end{gather*}
	subject to the relations in Lemma \ref{lem:KarRels}. 
\end{thm}

	Similar quiver-descriptions exist for other (non-semisimple) variations of
	   Temperley--Lieb categories, e.g.~in positive characteristic
	   \cite{2019arXiv190711560T} and at roots of unity \cite{sutton2021sl2}.
	   However, such quiver descriptions do not capture the monoidal structure
	   that is present in all these cases. For background on locally unital
	   algebras, such as path algebras of infinite quivers, and their modules, see \cite[\S 2.2]{BrSt}.

\begin{proof}
Let $\cal{C}$ denote the category of finitely-generated graded projective modules for the quiver
algebra described in the statement. The indecomposable objects in $\cal{C}$ are precisely the projective modules 
corresponding to the vertices $\mathsf{P}_n$, which we will denote by the same symbol. The assignments 
\[
\mathsf{P}_n \mapsto \JW_n \, , \quad
\mathsf{z}_n \mapsto \xs_{n} \, , \quad
\mathsf{U}_n \mapsto U_n \, , \quad
\mathsf{D}_n \mapsto D_n
\]
then determine a functor $\cal{C} \to \Mat(\Kar(\dTL))$, and it remains to show
that this establishes an equivalence of categories.

The subcategory of $\dTL$ consisting of morphisms of degree zero 
is precisely 
the Temperley--Lieb category $\TL$. 
By Definition \ref{def:karoubi}, 
the indecomposable objects in $\Mat(\Kar(\dTL))$ are thus the same as in $\Mat(\Kar(\TL))$, 
i.e.~they are the (images of the) Jones--Wenzl projectors and their shifts.
Hence, the functor $\cal{C} \to \Mat(\Kar(\dTL))$ is essentially surjective.

It remains to see that it is fully faithful.
For this, we assume that $m \leq n$
and give an explicit spanning set for 
$\Hom_{\Mat(\Kar(\dTL))}(\JW_m,\JW_n)$;
the $m>n$ case is analogous (but also follows from this case via duality).
By definition, this $\Hom$-space is spanned by elements $\JW_n X \JW_m$ where 
$X$ is an $(m,n)$-planar tangle, with each component possibly carrying a single dot.
If $X$ has any un-dotted caps at its bottom or un-dotted cups at its top, 
then $\JW_n X \JW_m=0$. 
Otherwise, we can use \eqref{eq:dAwd} to express
\[
\JW_n X \JW_m =\pm 
\begin{tikzpicture}[anchorbase,scale=.75]
	\draw[very thick,double] (-.25,-1.75) to 
		node[pos=.425,left=-15pt]{\tiny$\CQGbox{\ux_1 \cdots \ux_{\ell}}$}
		node[pos=.6,left=-1pt]{\scriptsize $k$} (-.25,1.75);
	\draw[very thick] (.25,1.75) to (.25,.75) to [out=270,in=180] (.5,.375)
		node{$\bullet$} to [out=0,in=270] (.75,.75);
	\draw[very thick] (.25,-1.75) to (.25,-.75) to [out=90,in=180] (.5,-.375)
		node {$\bullet$} to [out=0,in=90] (.75,-.75);
	\node at (1.75,1.5){\mydots};
	\node at (1,.5){\mysdots};
	\node at (1,-.5){\mysdots};
	\node at (1,-1.5){\mydots};
	\draw[very thick] (1.25,.75) to [out=270,in=180] (1.5,.375)
		node{$\bullet$} to [out=0,in=270] (1.75,.75); 
	\draw[very thick] (1.25,-.75) to [out=90,in=180] (1.5,-.375)
		node {$\bullet$} to [out=0,in=90] (1.75,-.75) to (1.75,-1.75);
	\filldraw[white] (-.5,-.75) rectangle (2,-1.25); 
	\draw[very thick] (-.5,-.75) rectangle (2,-1.25);
	\node at (.75,-1) {\scriptsize$\JW_{m}$};
	\draw[very thick] (2,.75) to [out=270,in=180] (2.25,.375)
		node{$\bullet$} to [out=0,in=270] (2.5,.75); 
	\node at (2.75,.5){\mysdots};
	\draw[very thick] (3,.75) to [out=270,in=180] (3.25,.375)
		node{$\bullet$} to [out=0,in=270] (3.5,.75) to (3.5,1.75);
	\filldraw[white] (-.5,.75) rectangle (3.75,1.25); 
	\draw[very thick] (-.5,.75) rectangle (3.75,1.25);
	\node at (1.75,1) {\scriptsize$\JW_{n}$};
\end{tikzpicture}
\stackrel{{\eqref{eq:twodotswitch}}}{=}\pm
\begin{tikzpicture}[anchorbase,scale=.75]
	\draw[very thick,double] (-.25,-1.75) to 
		node[pos=.425]{\tiny$\CQGbox{\ux_1 \cdots \ux_{\ell+m-k}}$}
		node[pos=.6,left=-1pt]{\scriptsize$m$} (-.25,1.75);
	\draw[very thick] (.25,1.75) to (.25,.75) to [out=270,in=180] (.5,.375)
		node{$\bullet$} to [out=0,in=270] (.75,.75);
	\node at (1,1.5){\mydots};
	\node at (1,.5){\mysdots};
	\draw[very thick] (1.25,.75) to [out=270,in=180] (1.5,.375)
		node{$\bullet$} to [out=0,in=270] (1.75,.75) to (1.75,1.75);
	\filldraw[white] (-.5,.75) rectangle (2,1.25); 
	\draw[very thick] (-.5,.75) rectangle (2,1.25);
	\node at (.75,1) {\scriptsize$\JW_{n}$};
	\filldraw[white] (-.75,-.75) rectangle (.25,-1.25); 
	\draw[very thick] (-.75,-.75) rectangle (.25,-1.25);
	\node at (-.25,-1) {\scriptsize$\JW_{m}$};
\end{tikzpicture}
=
\pm \frac{(k-\ell)!}{m!}
U_{n-2} \cdots U_m \xs_m^{m+\ell-k} \, .
\]
Here, $k$ is the number of ``through strands'' in $X$,
$\ell$ is the number of dots on those strands, and $\pm$ indicates some signs, not necessarily equal.

Thus, $\Hom_{\Kar(\dTL)}(\JW_m,\JW_n)$ is spanned by the elements $\big\{ U_{n-2}
\cdots U_m \xs_{m}^{k} \big\}_{k=0}^m$ when $m \leq n$ and by the elements
$\big\{ \xs_{n}^{k} D_{n+2} \cdots D_m \big\}_{k=0}^n$ when $m \geq n$. Hence,
the functor $\cal{C} \to \Mat(\Kar(\dTL))$ is full. In order to see that it is
faithful, it suffices to show that the above spanning sets are the images of
spanning sets of the corresponding morphism spaces in $\cal{C}$ (an easy check
using the relations on the quiver underlying $\cal{C}$) and are linearly independent. 
The latter follows
by showing that the spanning morphisms act linearly independently on the
polynomial representation of $\dTL$ considered in \S \ref{ss:PolyRep}, and we
delay this check until that section.
\end{proof}

\begin{cor}
	\label{cor:HomPP}
Let $m-n \in 2 \Z$.
There is an isomorphism of graded vector spaces
\[
\Hom_{\Kar(\dTL)}(\JW_m,\JW_n) \cong q^{|m-n|}\K[\xs]/\xs^{1+\min(m,n)}
\]
induced by composition with the maps $D_r$ and $U_r$,
which identifies $\xs$ with the endomorphism 
$\xs_{\min(m,n)} \in \End^2_{\Kar(\dTL)}(\JW_{\min(m,n)})$.
\end{cor}

If $m-n$ is odd, then clearly $\Hom_{\Kar(\dTL)}(\JW_m,\JW_n)=0$ for parity reasons.

\begin{cor}\label{cor:center of dTL}
The map $\K[s,z]/(s^2=1)\rightarrow Z(\dTL)$ from Lemma \ref{lemma:center of dTL} is an isomorphism of graded algebras.
\end{cor}
\begin{proof}
We will construct the inverse isomorphism $Z(\dTL)\rightarrow \K[s,z]/(s^2=1)$
explicitly.   Suppose we are given $f\in Z(\dTL)$.   For each $n$, let $f_n$
denote the restriction of $f$ to $\symobj_n$.  By Corollary \ref{cor:HomPP},
post-composing with the dotted cup $U_n$ gives an isomorphism of graded vector
spaces
\[
\End_{\dTL}(\symobj_n) \cong q^{-2} \Hom_{\dTL}(\symobj_n,\symobj_{n+2}) \, ,
\]
and pre-composing with $U_n$ induces a map
\[
\End_{\dTL}(\symobj_{n+2})\rightarrow q^{-2} \Hom_{\dTL}(\symobj_n,\symobj_{n+2})\cong \End_{\dTL}(\symobj_n) \, .
\]
By centrality, $f_n$ and $f_{n+2}$ must have the same image in $\End_{\dTL}(\symobj_n)$, 
so $(f_n)_{n\in \N}$ defines an element of the limit
\[
\lim_{n\in I} \End_{\dTL}(\symobj_n) \cong \lim_{n\in I} \K[\xs_n]/(\xs_n^n)
\]
where $I$ is the poset $2\N \sqcup (2\N+1)$ (disconnected since the maps go in steps of $2$).   
This limit is isomorphic to  $\K[z]\otimes \K[z]$, which in turn is isomorphic to $\K[s,z]/(s^2=1)$ via the map sending
\[
f\otimes g \mapsto \frac{1+s}{2} f+ \frac{1-s}{2}g \, .
\]
It is straightforward to check that this is an inverse to the algebra map from Lemma \ref{lemma:center of dTL}.
\end{proof}

\begin{remark}
In the proof above, we have used that the limit of the inverse (cofiltered) system
\[
\cdots \leftarrow \K[z]/(z^n) \leftarrow \K[z]/(z^{n+2})\leftarrow \cdots
\]
(in which all maps are canonical projections) is $\K[\xs]$.  We obtain the polynomial ring $\K[\xs]$ 
and not the power series ring $\K\llbracket \xs\rrbracket$,
since the limit is computed in the category of $\Z$-graded vector spaces.
\end{remark}

\subsection{The polynomial representation of $\dTL$}
	\label{ss:PolyRep}
	In \S  \ref{ss:intro representing} and 
	Remark~\ref{rem:dTLrepth} we have seen that the fiber functor from the
	Temperley--Lieb category to vector spaces 
	(namely representations of $\slnn{2}$) 
	admits an extension to the dotted Temperley--Lieb category,
	valued in graded vector spaces, 
	in which the dot acts as the Chevalley generator $E$. 
	In the present section we will give an alternative
	description of this action, the \emph{polynomial representation}, and prove
	that it is faithful.

\begin{definition}
Let $\{x_i\}_{i=1}^\infty$ be variables of degree $2$.
The \emph{polynomial representation} of $\dTL$ is the 
$\K$-linear monoidal functor 
$\Pol \colon \dTL\rightarrow \Vect^\Z$ given on objects by 
$\Pol(\gen^n) = q^{-n}\K[x_1,\ldots,x_n]/(x_i^2)$.
This functor is given on generating morphisms by:
\begin{itemize}
\item The dot $x \in \End_{\dTL}(\gen^1)$ maps to multiplication by $x$ on $\Pol(\gen)=q^{-1}\K[x]/(x^2)$.
\item The cap morphism 
$
\begin{tikzpicture}[anchorbase,scale=.75]
\draw[very thick] (0,0) to [out=90,in=180] (.25,.5) to [out=0,in=90] (.5,0);
\end{tikzpicture} 
\colon \gen^2 \to \gen^0$ maps to the degree zero linear map 
$q^{-2}\K[x_1,x_2]/(x_i^2)\rightarrow \K$ sending $x_i\mapsto 1$, 
$1\mapsto 0$, and $x_1x_2\mapsto 0$.

\item The cup morphism 
$
\begin{tikzpicture}[anchorbase,scale=.75,yscale=-1]
\draw[very thick] (0,0) to [out=90,in=180] (.25,.5) to [out=0,in=90] (.5,0);
\end{tikzpicture} 
\colon \gen^0 \to \gen^2$ maps to the degree zero linear map 
$\K\rightarrow q^{-2}\K[x_1,x_2]/(x_i^2)$ sending $1\mapsto x_1+x_2$.
\end{itemize}
\end{definition}

It is straightforward to check 
that the relations in $\dTL$ are satisfied by the above
assignments, hence $\Pol$ is well-defined.

\begin{proposition}
	\label{prop:PolyRep}
The polynomial representation $\Pol\colon \dTL\rightarrow \Vect^\Z$ is faithful (but not full).
\end{proposition}

For the reader's convenience we include a detailed proof, although analogous
results are well-known in the context of (nil) blob algebras. 

\begin{proof} We abbreviate $\Pol(n):=\Pol(\gen^n)$ for $n\in \N$. We have a
commutative diagram
\[
\begin{tikzcd}
	\Hom_{\dTL}(\gen^m,\gen^n) \arrow{r}{\otimes \id_n} \arrow{d} &
	 \Hom_{\dTL}(\gen^{m+n},\gen^{2n}) \arrow{r}{\mathrm{cap}\circ} \arrow{d} &
	 \Hom_{\dTL}(\gen^{m+n},\gen^0) \arrow{d} \\
	\Hom_\K\big( \Pol(m),\Pol(n) \big)  \arrow{r}{\otimes\id_{\Pol(n)}}  &
	\Hom_\K\big( \Pol(m+n),\Pol(2n) \big)  \arrow{r}{\Pol(\mathrm{cap})\circ} &
	\Hom_\K \big(\Pol(m+n),\Pol(0) \big)
	\end{tikzcd}
\]
of $\Z$-graded $\K$-vector spaces, where $\mathrm{cap}\colon
\gen^{2n}\rightarrow c^0$ denotes the nested cap morphism in $\dTL$,  the
vertical maps are applications of $\Pol$,  and we are abbreviating by writing
$\Hom_{\Vect^\Z}= \Hom_\K$.   In the bottom row we are using the fact that
$\Pol$ is monoidal, so $\Pol(m+n)\cong \Pol(m)\otimes\Pol(n)$ for all $m$.

The composition along each row is an isomorphism, so injectivity of the leftmost vertical arrow 
will follow from the injectivity of the rightmost.
To show this, consider the spanning set 
$\BS =\{d_i\}_{i\in I}$ for $\Hom_{\dTL}(\gen^{m+n},\gen^0)$, 
consisting of dotted Temperley--Lieb diagrams with the property that all of their dots 
are adjacent to the unbounded region above the diagram. 
The last relation in \eqref{eq:dTLrels} indeed shows that such diagrams span, 
since any dotted Temperley--Lieb diagram not satisfying this property can be rewritten 
as a linear combination of diagrams where each dot has either the same number or fewer 
strands separating it from the unbounded region, and at least one has strictly fewer.

We will find a set $\BS^*=\{ \d^*_i \}_{i\in I} \subset \Pol(m+n)$ of monic monomials 
and a total order on $I$, such that 
\begin{equation}
	\label{eq:UT}
d_i(\d^*_j)=\begin{cases}
	1 & \text{ if } i=j\\
	0 & \text{ if } i>j
\end{cases}
\end{equation}
This means that the pairing matrix $(d_i(\d^*_j))_{i,j\in I}$ is
unitriangular with respect to the order on $I$. 
As a consequence, these elements act linearly independently on $\Pol(m+n)$, 
i.e.~are linearly independent in $\Hom_{\Vect^\Z} \big(\Pol(m+n),\Pol(0) \big)$.
(This further shows that $\BS$ is in fact a basis for $\Hom_{\dTL}(\gen^{m+n},\gen^0)$.)

Given $d_i \in \BS$, define $\d^*_i$ as an ordered
monomial $x_{i_1}\cdots x_{i_k}$ with $i_1<\cdots < i_k$ by the following rule: 
for every cap in $d_i$ record the indices $a<b$ of its boundary points. 
Now include $x_a$ (but not $x_b$) in the monomial if there is no dot on the cap and 
include neither variable if there is a dot on the cap.
By construction, we have that $d_i(\d^*_i)=1$.

The total order on $I$ giving \eqref{eq:UT} corresponds to 
the following total order on such monomials:
we declare $i>j$ if the total degree $\d^*_i$ is greater than that of $\d^*_j$. 
We compare monomials of equal degree using an appropriate lexicographic order
determined by $x_1>x_2>\cdots$.  Precisely, $x_{i_1}\cdots x_{i_k}>x_{j_1}\cdots x_{j_k}$ if there is an index $l\geq 1$ for which $i_p=j_p$ for $1\leq p\leq l-1$ (regarded as vacuously true if $l=1$), and $i_l<j_l$.
For example,  on degree three monomials we have
\[
x_1x_2x_3 > x_1x_2x_4>x_1x_3x_4>x_2x_3x_4>\cdots.
\]

Given a monomial $\d^*_j = x_{j_1}\cdots x_{j_k}\in \BS^*$ with $k$ variables,
we can (re)construct the corresponding dotted Temperley--Lieb diagram $d_j$ as
follows. In order to evaluate non-trivially on $\d^*_j$, the diagram $d_j$ needs
precisely $k$ undotted caps and $(m+n)/2-k$ dotted caps. The positions of the
undotted caps are uniquely determined by $\d^*_j$ via the following greedy
algorithm.  We start with the variable $x_{j_k}$ of largest
index, then there must be an undotted cap connecting the boundary points $j_k$
and $j_k +1$. (The sought-after second endpoint must be $>j_k$, but, if it were
$> j_k+1$, it would require a dotted cap nested within, which is impossible in
$\BS$.) Now remove $x_{j_k}$ from the monomial and iterate until the monomial is
$1$. Finally, connect all of the remaining points with dotted caps, with none
nested in one another.

Finally, suppose that $d_i\in \BS$ is another diagram such that $d_i(\d_j^*)\neq 0$ 
(and so $d_i$ has exactly $k$ undotted caps), then we claim that $i\leq j$. 
Indeed, $d_i(\d_j^*)\neq 0$ if and only if every undotted cap in $d_i$ has one (but not both) 
of the variables corresponding to its boundary points present in $\d_j^*$. 
If all the variables correspond to the left boundary point, then $i=j$.
Otherwise there exists a smallest index $0\leq l\leq k$, such that the
variable $x_{i_{l}}$ corresponds to the right endpoint of a cap in $d_i$; 
denote the variable for the left endpoint by $j_l$. 
Then, $\d^*_i = x_{j_1}\cdots x_{j_{l-1}} x_{i_{l}} \cdots $ and 
$\d^*_j = x_{j_1}\cdots x_{j_{l-1}} x_{j_{l}} \cdots $ with $i_l>j_l$, so $i<j$.
\end{proof}

\begin{remark}
Proposition \ref{prop:PolyRep} shows that $\End_{\dTL}(\gen^n)$ 
can be viewed as the subalgebra of 
\[
\End_\K \big( \K[x_1,\ldots,x_n]/(x_i^2) \big)
\]
generated by multiplication by $x_i$ and the image of $\End_{\TL}(\gen^n)$. 
In this way, one can arrive at a definition of $\dTL$ using only (undotted) $\TL$ and its
polynomial representation.
\end{remark}

\begin{proposition}\label{prop:Pol of symobj} 
The polynomial representation extends to a functor 
$\Pol \colon \Kar(\dTL)\rightarrow \Vect^\Z$ 
(denoted by the same symbol), and we have
\[
\Pol(\symobj_n) \cong  q^{-n}\K[\xs]/(\xs^{n+1})\subset q^{-n}\K[x_1,\ldots,x_n]/(x_i^2).
\]
where $\xs = x_1-x_2+\cdots+(-1)^{n-1}x_n$.
\end{proposition}
\begin{proof}
The extension to $\Kar(\dTL)$ is immediate, since $\Vect^\Z$ is idempotent complete.

It remains to compute $\Pol(\JW_n)$.
Recall from \eqref{eqn:TLbraiding} that there is an algebra map
$\K[\mathfrak{S}_n] \to \End_{\dTL}(\gen^n)$
and that the Jones-Wenzl idempotent $\JW_n$ is the image of the symmetrizing
idempotent $\frac{1}{n!} \sum_{w\in \mathfrak{S}_n} w$.
On the level of the polynomial representation, the action of $\mathfrak{S}_n$ is given by
\[
\Pol(w) = \e\circ w \circ \e,
\]
where $w$ is the standard action on $\K[x_1,\ldots,x_n]/(x_i^2)$ by permuting variables, 
and $\e$ is the algebra automorphism of $\K[x_1,\ldots,x_n]/(x_i^2)$ sending $x_i
\mapsto (-1)^{i-1} x_i=\sx_i$.  In other words, the action of $\Pol(w)$ is by permuting the 
variables $\{x_1,-x_2,\ldots,\pm x_n\}$=$\{\sx_1,\sx_2,\ldots,\sx_n\}$.

It follows that the image of $\JW_n$ is the $\mathfrak{S}_n$-invariant subspace
$\big(\K[x_1,-x_2,\ldots,\pm x_n]/(x_i^2)\big)^{\mathfrak{S}_n}$. This is isomorphic to the
subalgebra generated by the first power sum symmetric polynomial
$\xs = x_1-x_2+\cdots \pm x_n$, 
since all higher-degree power sum symmetric polynomials are zero.
The relation $\xs^{n+1}=0$ follows from Remark \ref{rmk:too many z's}.
\end{proof}

We finish this section by giving two alternative interpretations of the
polynomial representation. A third will appear in Example \ref{ex:eval is pol}.

\begin{remark}\label{rem:diagrammatic poly} The polynomial representation
admits the following diagrammatic presentation, which we will revisit and extend in
\S \ref{sec:diag}.  For each $n \in \N$, 
let $\Pol'(n)$ denote the $\K$-vector
space formally spanned by diagrams of the form
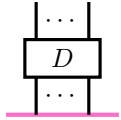
\[
\begin{tikzpicture}[anchorbase, scale=1]
	\draw[kirbystyle] (-.25,0)  to (1.25,0);
	\draw[very thick] (.15,0) to (.15,.5);
	\node at (.5,.25) {$\cdots$} ;
	\draw[very thick] (.85,0) to (.85,.5);
	\draw[very thick] (0,.5) rectangle (1,1); \node at (.5,.75) {$D$};
	\draw[very thick] (.15,1) to (.15,1.5);
	\node at (.5,1.25){$\cdots$} ;
	\draw[very thick] (.85,1) to (.85,1.5);
	\end{tikzpicture}
\]
where $D$ is a diagram in $\Hom_{\dTL}(\gen^m,\gen^n)$ for some $m$,
modulo the $\dTL$ relations among black strands, together with the local relations:
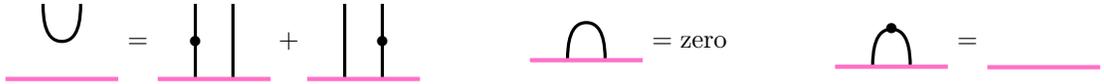
\[
\begin{tikzpicture}[anchorbase, scale=1]
	\draw[very thick] (.25,1) to [out=270,in=180] (.5,.5) to[out=0,in=270] (.75,1);
	\draw[kirbystyle] (-.25,0)  to (1.25,0);
\end{tikzpicture}
=
\begin{tikzpicture}[anchorbase, scale=1]
	\draw[very thick] (.25,0) to (.25,.5) node[black]{$\bullet$} to (.25,1);
	\draw[very thick] (.75,0) to (.75,1);
	\draw[kirbystyle] (-.25,0)  to (1.25,0);
\end{tikzpicture}
+
\begin{tikzpicture}[anchorbase, scale=1]
	\draw[very thick] (.25,0) to (.25,1);
	\draw[very thick] (.75,0)  to (.75,.5) node[black]{$\bullet$}  to  (.75,1);
	\draw[kirbystyle] (-.25,0)  to (1.25,0);
\end{tikzpicture}
\qquad\qquad
\begin{tikzpicture}[anchorbase, scale=1]
	\draw[very thick] (.25,0) to [out=90,in=180] (.5,.5) to[out=0,in=90] (.75,0);
	\draw[kirbystyle] (-.25,0)  to (1.25,0);
\end{tikzpicture}
= \text{zero}
\qquad\qquad
\begin{tikzpicture}[anchorbase, scale=1]
	\draw[very thick] (.25,0) to [out=90,in=180] (.5,.5)  node[black]{$\bullet$}  to[out=0,in=90] (.75,0);
	\draw[kirbystyle] (-.25,0)  to (1.25,0);
\end{tikzpicture}
=
\begin{tikzpicture}[baseline=.75em,scale=1]
	\draw[kirbystyle] (-.25,0)  to (1.25,0);
\end{tikzpicture}.
\]
We make $\Pol'(n)$ into a $\Z$-graded vector space by declaring that the degree of a 
diagram is $2\{\#\text{dots}\}-n$.  
The assignment $n\mapsto \Pol'(n)$ extends to a functor $\dTL\rightarrow \Vect^\Z$, 
with the action of morphisms in $\dTL$ given by vertical stacking. 
It is not hard to see that $\Pol \cong \Pol'$ as functors $\dTL \to \Vect^\Z$.
\end{remark}
		
\begin{remark}
	As shown in \cite[\S  3]{QR2} in the analogous setting of
$\glnn{2}$ foams, there is an equivalence of categories $\ABN \cong \hTr(\BN)
\cong \vTr(\BN)$ where $\hTr$ and $\vTr$ denote horizontal and (an appropriate
notion of) vertical trace of a bicategory, respectively. As discussed in
\cite[\S  9]{BHLW}, this implies that $\dTL \cong \ABN$ acts naturally on
the \emph{center of objects} of the bicategory $\BN$, i.e.~on the $\Z$-graded
vector space
\[
\bigoplus_{n \in \N} \End_{\BN}(\id_n) \, .
\]
For example, to the cap morphism 
$
\begin{tikzpicture}[anchorbase,scale=.75]
\draw[very thick] (0,0) to [out=90,in=180] (.25,.5) to [out=0,in=90] (.5,0);
\end{tikzpicture} 
\colon \gen^2 \to \gen^0$
in $\dTL$ 
we can associate the following surface with boundary:
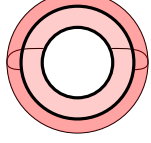
\[
\begin{tikzpicture}[anchorbase,scale=.75]
\draw (0,0) circle (1.25);
\draw (-1,-.125) to [out=180,in=270] (-1.25,0) to [out=90,in=180] (-.6,.25);
\draw (1,-.125) to [out=0,in=270] (1.25,0) to [out=90,in=0] (.6,.25);
\fill[red,opacity=.2, even odd rule] (0,0) circle[radius=1] circle[radius=1.25];
\fill[red,opacity=.2, even odd rule] (0,0) circle[radius=.625] circle[radius=1.25];
\draw[very thick] (0,0) circle (1);
\draw[very thick] (0,0) circle (.625);
\end{tikzpicture} 
\]
(a sphere with two disks removed), 
and we can ``plug in'' any endomorphism of the identity $1$-morphism $\id_2$ 
to obtain a closed (dotted) cobordism, 
i.e.~an endomorphism of the identity $1$-morphism $\id_0$.
Using the identification
\[
\End_{\BN}(\id_n) \cong \big(\K[x]/(x^2)\big)^{\otimes n} \cong \K[x_1,\ldots,x_n]/(x_i)^2
\]
one recovers the polynomial representation (up to shift).
\end{remark}

\section{Kirby color and handle slides}

In this section, we define our Kirby color and establish its handle slide invariance.

\begin{conv}\label{conv} To simplify notation in the remainder of the paper, 
we will abbreviate by writing
$\cdTL := \Ind(\Mat(\Kar(\dTL)))$ and $\cABN :=
\Ind(\Mat(\Kar(\ABN)))$\footnote{In the definition of $\cABN$ the operation
$\Mat$ can be omitted since $\ABN$ has $\Mat$ already built in and $\Kar(\ABN)$
still has finite biproducts and grading shifts.}. 
\end{conv}

\subsection{The Khovanov Kirby color}

Recall from \S  \ref{ss:Ind} that objects of $\cdTL$ are given by 
directed systems of objects in $\Mat(\Kar(\dTL))$. 
We now introduce our Kirby objects as such.

\begin{definition}
	\label{def:Kirby}
	For $k\geq 0$, the \emph{Kirby object} of \emph{winding number} $k$
	is the object $\kirby_k \in \cdTL$ given by the following directed system in $\Kar(\dTL)$:
\begin{equation}
	\label{eq:dirsys}
	 \kirby_k:=\left(
		 q^{-k}\, \symobj_{k} \xrightarrow{U_k} 
	 q^{-k-2}\, \symobj_{k+2} \xrightarrow{U_{k+2}} 
	 q^{-k-4} \symobj_{k+4}\, \xrightarrow{} \cdots \right) 
\end{equation}
where $U_{k+2i}$ is as in Definition \ref{def:KarGens}.
For $k<0$, we set $\kirby_k:=\kirby_{|k|}$.
\end{definition}

We immediately observe that the object $\kirby_k$ only depends on the parity of $k$.

\begin{lemma} 
	\label{lem:shift} 
For $k\in \Z$ there are canonical isomorphisms
$\kirby_k\cong \kirby_{k+2}$ in $\cdTL$.
\end{lemma}
\begin{proof}
Without loss of generality, we may assume that $k \geq 0$.
By Lemma \ref{lem:final}, 
the inclusion of $(k+2)+2\N \hookrightarrow k+2\N$ of filtered categories is final,
thus it induces an isomorphism $\kirby_{k+2} \xrightarrow{\cong} \kirby_k$ in $\cdTL$.
\end{proof}

The reference to winding numbers in Definition \ref{def:Kirby}
may seem obscure, given that Lemma \ref{lem:shift} shows that 
only the parity of $k$ is relevant. 
The appropriate context for this terminology is a conjectural
$\glN$-analogue of \eqref{eq:dirsys} that we discuss in 
Section \ref{ss:other+glN}.
In the present context, however, 
the following construction combines 
the distinct Kirby objects.

\begin{definition}\label{def:total kirby}
The \emph{total Kirby color} is $\kirby:=\kirby_0\oplus \kirby_1$.
\end{definition}

\begin{remark}
There are analogous definitions for Kirby colors defined as objects of a
full subcategory of the pro-completion of $\Mat(\Kar(\dTL))$. 
Our definition is chosen to be compatible with \cite{2020arXiv200908520M}.
\end{remark}

In \S  \ref{sec:handleslide} below (after reviewing further background), 
we establish the handle slide isomorphism for the total Kirby color.
Before doing so, we pause to establish an algebraic result that will be used later.
Note that $\Vect^\Z$ has grading shifts, direct sums and summands, and filtered colimits, 
so Remark \ref{rem:compact} implies that the polynomial representation from
\S \ref{ss:PolyRep} extends to a functor:
\begin{equation}
\label{eq:eval}
\Pol\colon \cdTL \rightarrow \mathrm{Vect}^\Z
\end{equation}
(denoted by the same symbol as before).
We thus evaluate this functor on the Kirby objects $\kirby_k$.

\begin{proposition}
	\label{prop:kirbyunknot} There is an isomorphism
\[
\Pol(\kirby_k) \cong \bigoplus_{i=0}^\infty q^{-2i}\K = \K\oplus q^{-2}\K\oplus q^{-4}\K\oplus \cdots
\]
of $\Z$-graded $\K$-vector spaces.
With respect to this isomorphism, the value of $\Pol$ on the canonical map 
$q^{-k}\symobj_k\rightarrow \kirby_k$ corresponds to the inclusion 
$\bigoplus_{i=0}^k q^{-2i}\K \hookrightarrow \bigoplus_{i=0}^\infty q^{-2i}\K$.
\end{proposition}

\begin{proof}
Proposition \ref{prop:Pol of symobj} gives that
\[
\Pol(\symobj_{k}) \cong q^{-k}\K[\xs]/(\xs^{k+1})
\]
which has graded dimension 
$\dim_q(\Pol(\symobj_{k})) = [k+1] = q^k + q^{k-2} + \cdots + q^{-k}$. 
It is straightforward to verify that
\[
q^{-k}\Pol(\symobj_{k}) \xrightarrow{\Pol(U_k)} q^{-k-2} \Pol(\symobj_{k+2}) 
\] 
is full rank, hence the colimit of the system 
$\big( \cdots \to q^{-k}\Pol(\symobj_{k}) \to q^{-k-2}\Pol(\symobj_{k+2}) \to \cdots \big)$ 
is the union of these graded vector spaces.
\end{proof}

\begin{remark}\label{rem:Verma} 
Proposition~\ref{prop:kirbyunknot} suggests a 
representation-theoretic interpretation of the Kirby objects: 
for $k \geq 0$ the graded dimension of $q^k\Pol(\kirby_k)$ coincides 
with the character of the (dual) Verma module for $\slnn{2}$ of highest weight $k$. 
Since there is a degree zero morphism $\JW_k \to q^k \kirby_k$ induced by the inclusion 
of the former in the directed system defining the latter, 
we see that $q^k\Pol(\kirby_k)$ is most-naturally viewed in this regard 
as the dual Verma module $\nabla(k)$, which has the irreducible $\slnn{2}$ representation 
of highest weight $k$ as a submodule.
In Corollary~\ref{cor:endkirby} below, 
we show that $\End_{\dTL}(\kirby_k) \cong \K[z]$ for a degree two element $z$.
As in Remark~\ref{rem:dTLrepth}, 
the action of this generator on $q^k\Pol(\kirby_k)$ 
again corresponds to the action of the Chevalley generator $E$ on $\nabla(k)$.
(Note that $E$ does not act nilpotently on $\nabla(k)$.)
\end{remark}

\subsection{The punctured Bar-Natan categories}
\label{ss:punctured BN}
In order to state (and prove) the handle slide invariance of $\kirby$, 
we now further discuss the (annular) Bar-Natan categories 
$\BN(S^1\times [0,1] ; \mathbf{p})$ from Definition \ref{def:BN}
in the case when $\mathbf{p} \neq \varnothing$.
To facilitate the description of additional algebraic structure on these categories, 
we will work with a punctured square $[0,1]^2 \smallsetminus \{ (\frac{1}{2},\frac{1}{2}) \}$ 
instead of the annulus.

\begin{definition}
For $n\in \N$ let $\mathbf{p}_n \subset (0,1)$ denote a chosen set of $n$ distinct points. 
Let 
\[
\BN_m^n = \BN \big( [0,1]^2 ; \mathbf{p}_m \times \{0\} \cup \mathbf{p}_n \times \{1\} \big) 
\]
and
\[
\PBN_m^n = \BN \big( [0,1]^2 \smallsetminus \{ {\textstyle (\frac{1}{2},\frac{1}{2} )} \} ; 
\mathbf{p}_m \times \{0\} \cup \mathbf{p}_n \times \{1\} \big)
\]
be the Bar-Natan categories of the square and punctured-square 
with marked points at the bottom and top.
\end{definition}

\begin{remark}
In \S \ref{ss:MdTL} below we will continue to draw curves in 
$[0,1]^2 \smallsetminus \{ (\frac{1}{2},\frac{1}{2}) \}$ as curves in $S^1\times [0,1]$.
The arcs of the outer boundary $S^1\times\{1\}$ 
corresponding to the top and bottom of the square $[0,1]^2$
will be clear from context (or entirely irrelevant).
\end{remark}

\newcommand{\tikzcube}[3]{
	\begin{scope}[shift={(#1,#2)}]
	\fill[#3, opacity=.2]  (0,0) to (-.5,.5) to (.5,.5) to (1,0) to (0,0);
	\draw[gray] (0,0) rectangle (1,1);
	\draw[gray] (-.5,.5) rectangle (.5,1.5);
	\draw[gray] (0,0) to (-.5,.5);
	\draw[gray] (0,1) to (-.5,1.5);
	\draw[gray] (1,0) to (.5,.5);
	\draw[gray] (1,1) to (.5,1.5);
	\end{scope}
}

We have the following structures on these categories, 
which we describe via glueing operations on the products 
of (punctured) squares with the interval $[0,1]$
(i.e.~the space within which Bar-Natan cobordisms are embedded):
\begin{enumerate}
	\item A composition functor 
	\[\star \colon  \BN_m^n\times \BN_k^m\rightarrow \BN^n_k \, , \quad 	
	\begin{tikzpicture}[anchorbase,scale=.7]
		\tikzcube{0}{0}{red}
		\node at (-.2,.65) {\tiny$n$};
		\node at (0.5,-.15) {\tiny$m$};
		\end{tikzpicture} 
		 \times
		 \begin{tikzpicture}[anchorbase,scale=.7]
			\tikzcube{0}{0}{blue}
			\node at (-.2,.65) {\tiny$m$};
			\node at (0.5,-.15) {\tiny$k$};
			\end{tikzpicture}  
			\mapsto
			\begin{tikzpicture}[anchorbase,scale=.7]
			\tikzcube{0}{0}{blue}
			\tikzcube{-.5}{.5}{red}
			\node at (-.7,1.15) {\tiny$n$};
			\node at (0.5,-.15) {\tiny$k$};
			\end{tikzpicture} \, .\]
	\item An external tensor product functor 
	\[\boxtimes \colon \BN_{m_1}^{n_1} \times \BN_{m_2}^{n_2} \rightarrow \BN_{m_1+m_2}^{n_1+n_2} \, , \quad 
	\begin{tikzpicture}[anchorbase,scale=.7]
		\tikzcube{0}{0}{red}
		\node at (-.2,.65) {\tiny$n_1$};
		\node at (0.5,-.15) {\tiny$m_1$};
		\end{tikzpicture} 
		 \times
		 \begin{tikzpicture}[anchorbase,scale=.7]
			\tikzcube{0}{0}{blue}
			\node at (-.2,.65) {\tiny$n_2$};
			\node at (0.5,-.15) {\tiny$m_2$};
			\end{tikzpicture}  
			\mapsto
			\begin{tikzpicture}[anchorbase,scale=.7]
			\tikzcube{0}{0}{blue}
			\tikzcube{-1}{0}{red}
			\node at (-.75,.65) {\tiny$n_1+n_2$};
			\node at (0,-.15) {\tiny$m_1+m_2$};
			\end{tikzpicture}  \, .
	\]
	\item A top and bottom action of $\BN$ on $\PBN$, which is to say a functor
	\[
	\BN^{n_2}_{n_1}\times \PBN^{n_1}_{m_1}\times \BN^{m_1}_{m_2} \rightarrow \PBN^{n_2}_{m_2} \, , \quad
	\begin{tikzpicture}[anchorbase,scale=.7]
		\tikzcube{0}{0}{red}
		\node at (-.2,.65) {\tiny$n_2$};
		\node at (0.5,-.15) {\tiny$n_1$};
		\end{tikzpicture} 
		\times
		\begin{tikzpicture}[anchorbase,scale=.7]
		 \tikzcube{0}{0}{gray}
		 \node at (-.2,.65) {\tiny$n_1$};
		 \node at (0.5,-.15) {\tiny$m_1$};
		 \node[gray] at (.25,.25) {$\bullet$};
		 \draw[gray,ultra thick] (.25,.25) to (.25,1.25);
		 \node[gray] at (.25,1.25) {$\bullet$};
		 \end{tikzpicture} 
			\times
		 \begin{tikzpicture}[anchorbase,scale=.7]
			\tikzcube{0}{0}{blue}
			\node at (-.2,.65) {\tiny$m_1$};
			\node at (0.5,-.15) {\tiny$m_2$};
			\end{tikzpicture}  
			\mapsto
			\begin{tikzpicture}[anchorbase,scale=.7]
			\tikzcube{-.5}{.5}{red}
			\tikzcube{0}{0}{gray}
			\tikzcube{.5}{-.5}{blue}
			\node[gray] at (.25,.25) {$\bullet$};
			\draw[gray,ultra thick] (.25,.25) to (.25,1.25);
			\node[gray] at (.25,1.25) {$\bullet$};
			\node at (-.7,1.15) {\tiny$n_2$};
			\node at (1,-.65) {\tiny$m_2$};
			\end{tikzpicture} 
	\]
	denoted $(X,M,Y)\mapsto X\star M\star Y$. 
	\item A left and right action of $\BN$ on $\PBN$,  which is to say a functor
	\[
	\BN^{n_1}_{m_1}\times \PBN^{n_2}_{m_2}\times \BN^{n_3}_{m_3} \rightarrow \PBN^{n_1+n_2+n_3}_{m_1+m_2+m_3} \, , \quad
	\begin{tikzpicture}[anchorbase,scale=.7]
		\tikzcube{0}{0}{red}
		\node at (-.2,.65) {\tiny$n_1$};
		\node at (0.5,-.15) {\tiny$m_1$};
		\end{tikzpicture} 
		\times
		\begin{tikzpicture}[anchorbase,scale=.7]
		 \tikzcube{0}{0}{gray}
		 \node at (-.2,.65) {\tiny$n_2$};
		 \node at (0.5,-.15) {\tiny$m_2$};
		 \node[gray] at (.25,.25) {$\bullet$};
		 \draw[gray,ultra thick] (.25,.25) to (.25,1.25);
		 \node[gray] at (.25,1.25) {$\bullet$};
		 \end{tikzpicture} 
			\times
		 \begin{tikzpicture}[anchorbase,scale=.7]
			\tikzcube{0}{0}{blue}
			\node at (-.2,.65) {\tiny$n_3$};
			\node at (0.5,-.15) {\tiny$m_3$};
			\end{tikzpicture}  
			\mapsto
			\begin{tikzpicture}[anchorbase,scale=.7]
			\node at (-.25,.65) {\tiny$n_1+n_2+n_3$};
			\tikzcube{1}{0}{blue}
			\tikzcube{0}{0}{gray}
			\tikzcube{-1}{0}{red}
			\node[gray] at (.25,.25) {$\bullet$};
			\draw[gray,ultra thick] (.25,.25) to (.25,1.25);
			\node[gray] at (.25,1.25) {$\bullet$};
			\node at (.5,-.15) {\tiny$m_1+m_2+m_3$};
			\end{tikzpicture} 
	\]
	denoted $(X,M,Y)\mapsto X\boxtimes M\boxtimes Y$. 
	\item An action of $\ABN$ on $\PBN$, which is to say a functor
	\[
		\ABN\times \PBN^{n}_{m} \rightarrow \PBN^{n}_{m} \, , \quad
		\begin{tikzpicture}[anchorbase,scale=.7]
			\node[gray] at (0,0) {$\bullet$};
			\fill[green, opacity=.2] (0,0) ellipse (.5 and 0.1875);
			\draw[gray] (0,0) ellipse (.5 and 0.1875);
			\draw[gray] (-.5,0) to (-.5,1);
			\draw[gray] (.5,0) to (.5,1);
			\draw[gray,ultra thick] (0,0) to (0,1);
			\node[gray] at (0,1) {$\bullet$};
			\draw[gray] (0,1) ellipse (.5 and 0.1875);
			\end{tikzpicture} 
			\times 
		\begin{tikzpicture}[anchorbase,scale=.7]
		 \tikzcube{0}{0}{gray}
		 \node at (-.2,.65) {\tiny$n$};
		 \node at (0.5,-.15) {\tiny$m$};
		 \node[gray] at (.25,.25) {$\bullet$};
		 \draw[gray,ultra thick] (.25,.25) to (.25,1.25);
		 \node[gray] at (.25,1.25) {$\bullet$};
		 \end{tikzpicture} 
			\mapsto
			\begin{tikzpicture}[anchorbase,scale=.7]
				\fill[gray, opacity=.2]   (-.125,-.125) to (-.875,.625) to (.625,.625) to (1.375,-.125) to (-.125,-.125);
				\draw[gray] (-.125,-.125) rectangle (1.375,0.875);
				\draw[gray] (-.875,.625) rectangle (.625,1.625);
				\draw[gray] (-.125,-.125) to (-.875,.625);
				\draw[gray] (-.125,.875) to (-.875,1.625);
				\draw[gray] (1.375,-.125) to (.625,.625);
				\draw[gray] (1.375,.875) to (.625,1.625);
				\node at (-.3,.775) {\tiny$n$};
				\node at (0.625,-.275) {\tiny$m$};
				\fill[green, opacity=.2] (.25,.25) ellipse (.5 and 0.1875);
				\node[gray] at (.25,.25) {$\bullet$};
				\draw[gray,ultra thick] (.25,.25) to (.25,1.25);
				\node[gray] at (.25,1.25) {$\bullet$};
				\end{tikzpicture} 
	\]
	denoted $(C,M)\mapsto C\bullet M$, defined by inserting $C$ in a neighborhood of the puncture.
	\end{enumerate}

The operations (1) and (2) make $\BN$ into a monoidal bicategory 
with 1-morphism categories $\BN_n^m$.  
The operations (3) and (4) make $\PBN$ into a bimodule over $\BN$ in two ways, 
corresponding to the two binary operations $\star$ and $\boxtimes$ on $\BN$. 
The operation (5) makes $\PBN$ into a left module over
$\ABN$ in a way which commutes with all the aforementioned structures; 
if we identify $\PBN_0^0 \cong \ABN$, then this recovers the monoidal structure $\otimes$ 
on $\ABN$ from Remark \ref{rem:ABNmonoidal}.
There are various relationships between these actions, 
all of which are apparent from the above descriptions, e.g.
\[
(X_1\boxtimes Y_1\boxtimes Z_1)\star (X_2\boxtimes M\boxtimes Z_2)\star(X_3\boxtimes Y_3\boxtimes Z_3)
\cong (X_1\star X_2\star X_3)\boxtimes (Y_1\star M\star Y_3) \boxtimes (Z_1\star Z_2\star Z_3) \, .
\]

\begin{remark}
All of the above structure persists upon proceeding to the completion $\cPBN:= \Ind(\Mat(\Kar(\PBN)))$.
\end{remark}

\subsection{The 2-point category}
	\label{ss:MdTL}

We now extend the Temperley-Lieb type presentation of $\ABN$ 
to a similar presentation of its module category $\PBN^1_1$.

\begin{definition}
	\label{def:MdTL}
Let $\tpc$ be the $\Z$-graded $\K$-linear category 
with objects pairs $(n,Z) \in \N \times \{L,R\}$ 
and generated as a module category over $\dTL$ 
(acting on the left)
by morphisms
\[
\begin{tikzpicture}[anchorbase,scale=1]
\draw[ultra thick,blue] (0,0) node[below=-2pt]{$Z$} to 
	node[black]{$\bullet$} (0,1) node[above=-2pt]{$Z$};
\end{tikzpicture}
\in \End^2\!\big( (0,Z) \big)
\; , \quad
\begin{tikzpicture}[anchorbase,xscale=-1]
\draw[very thick] (.5,0) to (0,.5);
\draw[ultra thick,blue] (0,0) node[below=-2pt]{$Y$} to (0,1) node[above=-2pt]{$Z$};
\end{tikzpicture}
\in \Hom^1\!\big( (1,Y) , (0,Z) \big)
\]
for $Y,Z \in \{L,R\}$ with $Y\neq Z$, modulo local relations. Setting
\[
\begin{tikzpicture}[anchorbase,xscale=-1]
\draw[very thick] (.5,1) to (0,.5);
\draw[ultra thick,blue] (0,0) node[below=-2pt]{$Y$} to (0,1) node[above=-2pt]{$Z$};
\end{tikzpicture}
:=
\begin{tikzpicture}[anchorbase,xscale=-1]
\draw[very thick] (0,.5) to (.125,.375) to [out=315,in=270] (.5,1);
\draw[ultra thick,blue] (0,0) node[below=-2pt]{$Y$} to (0,1) node[above=-2pt]{$Z$};
\end{tikzpicture}
\]
the relations are
\begin{equation}
	\label{eq:MdTLrels}
\begin{tikzpicture}[anchorbase,xscale=-1]
\draw[very thick] (.5,0) to (0,.375);
\draw[very thick] (0,.625) to (.5,1);
\draw[ultra thick,blue] (0,0) node[below=-2pt]{$Z$} to (0,1) node[above=-2pt]{$Z$};
\end{tikzpicture}
=
\begin{tikzpicture}[anchorbase,xscale=-1]
\draw[ultra thick,blue] (0,0) node[below=-2pt]{$Z$} to node[black]{$\bullet$} 
	(0,1) node[above=-2pt]{$Z$};
\draw[very thick] (.5,0) to (.5,1);
\end{tikzpicture}
+
\begin{tikzpicture}[anchorbase,xscale=-1]
\draw[ultra thick,blue] (0,0) node[below=-2pt]{$Z$} to (0,1) node[above=-2pt]{$Z$};
\draw[very thick] (.5,0) to node {$\bullet$} (.5,1);
\end{tikzpicture}
\: , \quad
\begin{tikzpicture}[anchorbase,xscale=-1]
\draw[very thick] (.5,0) to (0,.5);
\draw[ultra thick,blue] (0,0) node[below=-2pt]{$Y$} to
	node[pos=.25,black]{$\bullet$} (0,1) node[above=-2pt]{$Z$};
\end{tikzpicture}
=
\begin{tikzpicture}[anchorbase,xscale=-1]
\draw[very thick] (.5,0) to node{$\bullet$} (0,.5);
\draw[ultra thick,blue] (0,0) node[below=-2pt]{$Y$} to (0,1) node[above=-2pt]{$Z$};
\end{tikzpicture}
=
\begin{tikzpicture}[anchorbase,xscale=-1]
\draw[very thick] (.5,0) to (0,.5);
\draw[ultra thick,blue] (0,0) node[below=-2pt]{$Y$} to
	node[pos=.75,black]{$\bullet$} (0,1) node[above=-2pt]{$Z$};
\end{tikzpicture}
\; , \quad
\begin{tikzpicture}[anchorbase,xscale=-1]
\draw[ultra thick,blue] (0,0) node[below=-2pt]{$Z$} to node[black]{$\bullet$} 
	node[black,right=-1pt,yshift=3pt]{\scriptsize$2$} (0,1) node[above=-2pt]{$Z$};
\end{tikzpicture}
=0 \, .
\end{equation}
The module structure is given on objects by $\gen^m\otimes (n,Z) = (m+n,Z)$ 
and given on morphisms by placing dotted Temperley--Lieb diagrams on the left.
\end{definition}

\begin{thm}
	\label{thm:MdTL}
The assignments
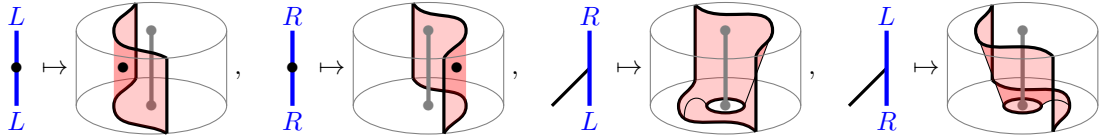
\[
L \mapsto 
\begin{tikzpicture}[anchorbase,scale=1]
\node[gray] at (0,0) {$\bullet$};
\draw[gray] (0,0) circle (.5);
\draw[very thick] (0,.5) to [out=270,in=90] (-.25,0) to [out=270,in=90] (0,-.5);
\end{tikzpicture}
\: , \quad
R \mapsto
\begin{tikzpicture}[anchorbase,scale=1]
\node[gray] at (0,0) {$\bullet$};
\draw[gray] (0,0) circle (.5);
\draw[very thick] (0,.5) to [out=270,in=90] (.25,0) to [out=270,in=90] (0,-.5);
\end{tikzpicture}
\]
and
\[
\begin{tikzpicture}[anchorbase,xscale=-1]
\draw[ultra thick,blue] (0,0) node[below=-2pt]{$L$} to 
	node[black]{$\bullet$} (0,1) node[above=-2pt]{$L$};
\end{tikzpicture}
\mapsto
\begin{tikzpicture}[anchorbase,scale=1]
\node[gray] at (0,0) {$\bullet$};
\draw[gray] (0,0) ellipse (1 and 0.375);
\draw[gray] (-1,0) to (-1,1);
\draw[gray] (1,0) to (1,1);
\draw[very thick] (-.205,.367) to (-.205,1.367);
\draw[very thick] (-.205,.367) to [out=240,in=90] (-.5,0) to
	[out=270,in=150] (.205,-.367);
\draw[gray,ultra thick] (0,0) to (0,1);
\path[fill=red,opacity=.2] (-.205,.367) to [out=240,in=90] (-.5,0)
	to (-.5,1) to [out=90,in=240] (-.205,1.367);
\path[fill=red,opacity=.2] (-.5,0) to [out=270,in=150] (.205,-.367)
	to (.205,.633) to [out=150,in=270] (-.5,1);
\draw[very thick] (-.205,1.367) to [out=240,in=90] (-.5,1) to
	[out=270,in=150] (.205,.633);
\draw[very thick] (.205,-.367) to (.205,.633);
\node at (-.375,.5) {$\bullet$};
\node[gray] at (0,1) {$\bullet$};
\draw[gray] (0,1) ellipse (1 and 0.375);
\end{tikzpicture} 
\; , \quad
\begin{tikzpicture}[anchorbase,scale=1]
\draw[ultra thick,blue] (0,0) node[below=-2pt]{$R$} to 
	node[black]{$\bullet$} (0,1) node[above=-2pt]{$R$};
\end{tikzpicture}
\mapsto
\begin{tikzpicture}[anchorbase,scale=1]
\node[gray] at (0,0) {$\bullet$};
\draw[gray] (0,0) ellipse (1 and 0.375);
\draw[gray] (-1,0) to (-1,1);
\draw[gray] (1,0) to (1,1);
\draw[very thick] (-.205,.367) to (-.205,1.367);
\draw[very thick] (-.205,.367) to [out=300,in=90] (.5,0) to
	[out=270,in=60] (.205,-.367);
\path[fill=red,opacity=.2] (-.205,.367) to [out=300,in=90] (.5,0)
	to (.5,1) to [out=90,in=300] (-.205,1.367);
\path[fill=red,opacity=.2] (.5,0) to [out=270,in=60] (.205,-.367)
	to (.205,.633) to [out=60,in=270] (.5,1);
\draw[gray,ultra thick] (0,0) to (0,1);
\draw[very thick] (-.205,1.367) to [out=300,in=90] (.5,1) to
	[out=270,in=60] (.205,.633);
\draw[very thick] (.205,-.367) to (.205,.633);
\node at (.375,.5) {$\bullet$};
\node[gray] at (0,1) {$\bullet$};
\draw[gray] (0,1) ellipse (1 and 0.375);
\end{tikzpicture}
\; , \quad
\begin{tikzpicture}[anchorbase,xscale=-1]
\draw[very thick] (.5,0) to (0,.5);
\draw[ultra thick,blue] (0,0) node[below=-2pt]{$L$} to (0,1) node[above=-2pt]{$R$};
\end{tikzpicture}
\mapsto
\begin{tikzpicture}[anchorbase,scale=1]
\node[gray] at (0,0) {$\bullet$};
\draw[gray] (0,0) ellipse (1 and 0.375);
\draw[gray] (-1,0) to (-1,1);
\draw[gray] (1,0) to (1,1);
\draw[very thick] (-.404,.343) to (-.404,1.343);
\draw[very thick] (-.404,.343) to [out=255,in=90] (-.625,0) to
	[out=270,in=135] (.404,-.343);
\draw[very thick] (0,0) ellipse (.25 and 0.09375);
\path[fill=red,opacity=.2] (-.625,0) to [out=60,in=180] (-.375,.125)
	to [out=0,in=90] (-.25,0) arc[start angle=180, end angle=0,x radius=.25,y radius=.09375]
	to (.625,1) to [out=90,in=315] (-.404,1.343) to (-.404,.343) to [out=255,in=90] (-.625,0);
\draw[gray,ultra thick] (0,0) to (0,1);
\path[fill=red,opacity=.2] (-.625,0) to [out=60,in=180] (-.375,.125)
	to [out=0,in=90] (-.25,0) arc[start angle=180, end angle=360,x radius=.25,y radius=.09375]
	to (.625,1) to [out=270,in=75] (.404,.657) to (.404,-.343) to [out=135,in=270] (-.625,0);
\draw[very thick] (-.404,1.343) to [out=315,in=90] (.625,1) to
	[out=270,in=75] (.404,.657);
\draw[very thick] (.404,-.343) to (.404,.657);
\draw (-.55,-0.17) to [out=90,in=190] (-.4075,.1) to [out=10,in=90] (-.25,-.03); 
\draw (.25,0) to (.625,1); 
\node[gray] at (0,1) {$\bullet$};
\draw[gray] (0,1) ellipse (1 and 0.375);
\end{tikzpicture} 
\; , \quad
\begin{tikzpicture}[anchorbase,xscale=-1]
\draw[very thick] (.5,0) to (0,.5);
\draw[ultra thick,blue] (0,0) node[below=-2pt]{$R$} to (0,1) node[above=-2pt]{$L$};
\end{tikzpicture}
\mapsto
\begin{tikzpicture}[anchorbase,scale=1]
\node[gray] at (0,0) {$\bullet$};
\draw[gray] (0,0) ellipse (1 and 0.375);
\draw[gray] (-1,0) to (-1,1);
\draw[gray] (1,0) to (1,1);
\draw[very thick] (-.404,.343) to (-.404,1.343);
\draw[very thick] (-.404,.343) to [out=315,in=90] (.625,0) to
	[out=270,in=75] (.404,-.343);
\draw[very thick] (0,0) ellipse (.25 and 0.09375);
\path[fill=red,opacity=.2] (-.404,.343) to (-.404,1.343) to [out=255,in=90] (-.625,1) to
	(-.25,0) arc[start angle=180, end angle=0, x radius=.25, y radius=0.09375] to 
	[out=45,in=135] (.625,0) to [out=90,in=315] (-.404,.343);
\draw[gray,ultra thick] (0,0) to (0,1);
\path[fill=red,opacity=.2] (.404,.657) to (.404,-.343) to [out=75,in=270] 
	(.625,0) to [out=135,in=45] (.25,0) 
	arc[start angle=360, end angle=180, x radius=.25, y radius=0.09375]
	to (-.625,1) to [out=270,in=135] (.404,.657);
\draw[very thick] (-.404,1.343) to [out=255,in=90] (-.625,1) to
	[out=270,in=135] (.404,.657);
\draw[very thick] (.404,-.343) to (.404,.657);
\draw (.25,0) to [out=45,in=135] (.625,0); 
\draw (-.25,0) to (-.625,1); 
\node[gray] at (0,1) {$\bullet$};
\draw[gray] (0,1) ellipse (1 and 0.375);
\end{tikzpicture} 
\]
determine a fully faithful functor $\Phi \colon \tpc
\hookrightarrow \PBN^1_1$ of module categories, i.e.~such that the diagram
\begin{equation}
	\label{eq:modcat}
\begin{tikzcd}
	\dTL\otimes \tpc \ar[rr] \ar[d,hook] & & \tpc \ar[d,hook] \\
	\ABN\otimes \PBN^1_1  \ar[rr] & & \PBN^1_1
\end{tikzcd}
\end{equation}
commutes.
The induced functor $\Mat(\tpc) \to \PBN^1_1$ is an equivalence of categories.
\end{thm}

In the dimensionally-reduced calculus of $\tpc$, the left side of diagrams
corresponds to the inside of the annulus (close to the puncture) and the right side to the outside boundary. 

\begin{proof}
It is straightforward to check that the images of the relations in \eqref{eq:MdTLrels} 
are satisfied in $\PBN^1_1$, e.g.~the image of the first relation holds by neck cutting. 
Hence, they determine a functor that satisfies \eqref{eq:modcat} by definition.

In order to show that this functor is fully faithful, 
we will establish isomorphisms between $\Hom$-spaces in $\tpc$ and $\PBN^1_1$ 
and certain $\Hom$-spaces in $\dTL$ and $\ABN$.
First, note that the relations in \eqref{eq:MdTLrels} imply that spanning sets for 
$\Hom_{\tpc} \big( (m,Z) , (n,Z) \big)$ and 
$\Hom_{\tpc}\big( (m,Y) , (n,Z) \big)$ for $Y \neq Z$ are given by
\[
\left\{
\begin{tikzpicture}[anchorbase,xscale=-.75,yscale=.75]
\draw[ultra thick,blue] (-.75,-.75) node[below=-2pt]{$Z$} to node[black]{$\bullet$} 
	node[black,right=-1pt]{\scriptsize $i$} (-.75,.75) node[above=-2pt]{$Z$};
\draw[very thick] (-.25,-.75) to (-.25,.75);
\node at (0,-.5){\mysdots};
\node at (0,.5){\mysdots};
\draw[very thick] (.25,-.75) to (.25,.75);
\filldraw[white] (-.5,.375) rectangle (.5,-.375); 
\draw[very thick] (-.5,.375) rectangle (.5,-.375);
\node at (0,0) {$D$};
\end{tikzpicture}
\Bigm|
D \in \Hom_{\dTL}(\gen^m,\gen^n)
\, , \; i=0,1
\right\}
\quad
\text{and}
\quad
\left\{
\begin{tikzpicture}[anchorbase,xscale=-.75,yscale=.75]
\draw[ultra thick,blue] (-1,-1) node[below=-2pt]{$Y$} to (-1,1) node[above=-2pt]{$Z$};
\draw[very thick] (-.25,-1) to (-.25,1);
\node at (.125,-.875){\mysdots};
\node at (.125,.5){\mysdots};
\draw[very thick] (.5,-1) to (.5,1);
\draw[very thick] (-.5,.125) to (-1,.5);
\filldraw[white] (-.75,.125) rectangle (.75,-.625); 
\draw[very thick] (-.75,.125) rectangle (.75,-.625);
\node at (0,-.25) {$D$};
\end{tikzpicture}
\Bigm|
D \in \Hom_{\dTL}(\gen^m,\gen^{n+1})
\right\}
\, .
\]
Thus,
\begin{equation}
	\label{eq:ZZdim}
\dim \Big( \Hom_{\tpc} \big( (m,Z) , (n,Z) \big) \Big) \leq 
2 \dim\Big( \Hom_{\dTL}(\gen^m,\gen^n) \Big)
\end{equation}
and
\begin{equation}
	\label{eq:YZdim}
\dim\Big( \Hom_{\tpc} \big( (m,Y) , (n,Z) \big) \Big) \leq 
\dim\Big( \Hom_{\dTL}(\gen^m,\gen^{n+1}) \Big) \, .
\end{equation}

We now restrict to the case when $Z=L$, since the $Z=R$ case is analogous.
Consider the composition of homogeneous linear maps
\begin{equation}
	\label{eq:ZZiso}
\begin{aligned}
\Hom_{\tpc} \big( (m,L) , (n,L) \big) \to&
\Hom_{\PBN^1_1}\Big(
\begin{tikzpicture}[anchorbase,scale=1]
\node[gray] at (0,0) {$\bullet$};
\draw[gray] (0,0) circle (.5);
\draw[very thick] (0,.5) to [out=270,in=90] (-.375,0) to [out=270,in=90] (0,-.5);
\draw[double, very thick] (0,0) circle (.1875);
\node at (.25,.25) {\tiny$m$};
\end{tikzpicture} 
\, , \, 
\begin{tikzpicture}[anchorbase,scale=1]
\node[gray] at (0,0) {$\bullet$};
\draw[gray] (0,0) circle (.5);
\draw[very thick] (0,.5) to [out=270,in=90] (-.375,0) to [out=270,in=90] (0,-.5);
\draw[double, very thick] (0,0) circle (.1875);
\node at (.25,.25) {\tiny$n$};
\end{tikzpicture} 
\Big)
\cong
q \Hom_{\ABN}\Big(
\begin{tikzpicture}[anchorbase,scale=1.25]
\node[gray] at (0,0) {$\bullet$};
\draw[gray] (0,0) circle (.5);
\draw[double, very thick] (0,0) circle (.1875);
\node at (.25,.25) {\tiny$m$};
\end{tikzpicture} 
\, , \, 
\begin{tikzpicture}[anchorbase,scale=1.25]
\node[gray] at (0,0) {$\bullet$};
\draw[gray] (0,0) circle (.5);
\draw[very thick] (-.375,0) circle (.0625); 
\draw[double, very thick] (0,0) circle (.1875);
\node at (.25,.25) {\tiny$n$};
\end{tikzpicture} 
\Big) \\
&\cong 
q \Hom_{\ABN}\Big(
\begin{tikzpicture}[anchorbase,scale=1.25]
\node[gray] at (0,0) {$\bullet$};
\draw[gray] (0,0) circle (.5);
\draw[double, very thick] (0,0) circle (.1875);
\node at (.25,.25) {\tiny$m$};
\end{tikzpicture} 
\, , \, 
q^{-1} \,
\begin{tikzpicture}[anchorbase,scale=1.25]
\node[gray] at (0,0) {$\bullet$};
\draw[gray] (0,0) circle (.5);
\draw[double, very thick] (0,0) circle (.1875);
\node at (.25,.25) {\tiny$n$};
\end{tikzpicture} 
\oplus
q \,
\begin{tikzpicture}[anchorbase,scale=1.25]
\node[gray] at (0,0) {$\bullet$};
\draw[gray] (0,0) circle (.5);
\draw[double, very thick] (0,0) circle (.1875);
\node at (.25,.25) {\tiny$n$};
\end{tikzpicture}
\Big) \\
&\cong 
\Hom_{\dTL}(\gen^m,\gen^n) \oplus q^2 \Hom_{\dTL}(\gen^m,\gen^n) \, .
\end{aligned}
\end{equation}
Here, the first isomorphism exists because both $\Hom$-spaces describe the
``Bar-Natan skein module'' spanned by 
(dotted) surfaces in the solid torus with boundary consisting of
$m+n$ longitudes and one null-homotopic circle. (Informally, the isomorphism can
be implemented by ``sliding'' the union of the strand at the bottom (i.e.~source
side) of the cobordism that meets the boundary and the vertical segments to the
top (i.e.~target side) of the cobordism.) Since the composition of these maps
sends
\[
\begin{tikzpicture}[anchorbase,xscale=-.75,yscale=.75]
\draw[ultra thick,blue] (-.75,-.75) node[below=-2pt]{$L$} to  
	(-.75,.75) node[above=-2pt]{$L$};
\draw[very thick] (-.25,-.75) to (-.25,.75);
\node at (0,-.5){\mysdots};
\node at (0,.5){\mysdots};
\draw[very thick] (.25,-.75) to (.25,.75);
\filldraw[white] (-.5,.375) rectangle (.5,-.375); 
\draw[very thick] (-.5,.375) rectangle (.5,-.375);
\node at (0,0) {$D$};
\end{tikzpicture}
\mapsto
\left(0,
\begin{tikzpicture}[anchorbase,scale=.75]
\draw[very thick] (-.25,-.75) to (-.25,.75);
\node at (0,-.5){\mysdots};
\node at (0,.5){\mysdots};
\draw[very thick] (.25,-.75) to (.25,.75);
\filldraw[white] (-.5,.375) rectangle (.5,-.375); 
\draw[very thick] (-.5,.375) rectangle (.5,-.375);
\node at (0,0) {$D$};
\end{tikzpicture} \,
\right)
\; , \quad
\begin{tikzpicture}[anchorbase,xscale=-.75,yscale=.75]
\draw[ultra thick,blue] (-.75,-.75) node[below=-2pt]{$L$} to node[black]{$\bullet$} 
	(-.75,.75) node[above=-2pt]{$L$};
\draw[very thick] (-.25,-.75) to (-.25,.75);
\node at (0,-.5){\mysdots};
\node at (0,.5){\mysdots};
\draw[very thick] (.25,-.75) to (.25,.75);
\filldraw[white] (-.5,.375) rectangle (.5,-.375); 
\draw[very thick] (-.5,.375) rectangle (.5,-.375);
\node at (0,0) {$D$};
\end{tikzpicture}
\mapsto
\left(
\begin{tikzpicture}[anchorbase,scale=.75]
\draw[very thick] (-.25,-.75) to (-.25,.75);
\node at (0,-.5){\mysdots};
\node at (0,.5){\mysdots};
\draw[very thick] (.25,-.75) to (.25,.75);
\filldraw[white] (-.5,.375) rectangle (.5,-.375); 
\draw[very thick] (-.5,.375) rectangle (.5,-.375);
\node at (0,0) {$D$};
\end{tikzpicture} \,
,0
\right)
\]
the first map in \eqref{eq:ZZiso} is surjective, 
thus an isomorphism by \eqref{eq:ZZdim}.

Similarly, the composition
\begin{equation}
	\label{eq:YZiso}
\begin{aligned}
\Hom_{\tpc} \big( (m,R) , (n,L) \big) \to
\Hom_{\PBN^1_1}\Big(
\begin{tikzpicture}[anchorbase,scale=1,rotate=180]
\node[gray] at (0,0) {$\bullet$};
\draw[gray] (0,0) circle (.5);
\draw[very thick] (0,.5) to [out=270,in=90] (-.375,0) to [out=270,in=90] (0,-.5);
\draw[double, very thick] (0,0) circle (.1875);
\node at (.25,-.25) {\tiny$m$};
\end{tikzpicture} 
\, , \, 
\begin{tikzpicture}[anchorbase,scale=1]
\node[gray] at (0,0) {$\bullet$};
\draw[gray] (0,0) circle (.5);
\draw[very thick] (0,.5) to [out=270,in=90] (-.375,0) to [out=270,in=90] (0,-.5);
\draw[double, very thick] (0,0) circle (.1875);
\node at (.25,.25) {\tiny$n$};
\end{tikzpicture} 
\Big)
&\cong
q \Hom_{\ABN}\Big(
\begin{tikzpicture}[anchorbase,scale=1]
\node[gray] at (0,0) {$\bullet$};
\draw[gray] (0,0) circle (.5);
\draw[double, very thick] (0,0) circle (.1875);
\node at (.25,.25) {\tiny$m$};
\end{tikzpicture} 
\, , \, 
\begin{tikzpicture}[anchorbase,scale=1]
\node[gray] at (0,0) {$\bullet$};
\draw[gray] (0,0) circle (.5);
\draw[double, very thick] (0,0) circle (.1875);
\node at (0,.375) {\tiny$n{+}1$};
\end{tikzpicture} 
\Big) \\
&\cong 
q\Hom_{\dTL}(\gen^m,\gen^{n+1})
\end{aligned}
\end{equation}
(wherein the first isomorphism can be seen as before 
by ``sliding'' the bottom and vertical segments in a cobordism to the top)
sends 
\[
\begin{tikzpicture}[anchorbase,xscale=-.75,yscale=.75]
\draw[ultra thick,blue] (-1,-1) node[below=-2pt]{$R$} to (-1,1) node[above=-2pt]{$L$};
\draw[very thick] (-.25,-1) to (-.25,1);
\node at (.125,-.875){\mysdots};
\node at (.125,.5){\mysdots};
\draw[very thick] (.5,-1) to (.5,1);
\draw[very thick] (-.5,.125) to (-1,.5);
\filldraw[white] (-.75,.125) rectangle (.75,-.625); 
\draw[very thick] (-.75,.125) rectangle (.75,-.625);
\node at (0,-.25) {$D$};
\end{tikzpicture}
\mapsto 
\begin{tikzpicture}[anchorbase,scale=.75]
\draw[very thick] (-.25,-.75) to (-.25,.75);
\node at (0,-.5){\mysdots};
\node at (0,.5){\mysdots};
\draw[very thick] (.25,-.75) to (.25,.75);
\filldraw[white] (-.5,.375) rectangle (.5,-.375); 
\draw[very thick] (-.5,.375) rectangle (.5,-.375);
\node at (0,0) {$D$};
\end{tikzpicture} \, .
\]
Therefore, the first map in \eqref{eq:YZdim} is surjective, 
hence an isomorphism by \eqref{eq:YZdim}. 
Thus, $\Phi \colon \tpc \hookrightarrow \PBN^1_1$ is fully faithful.

Finally, the induced functor $\Phi \colon \Mat(\tpc) \hookrightarrow \PBN^1_1$
is essentially surjective by the circle removal isomorphism \eqref{eq:circle}.
\end{proof}

\subsection{Invariance under handle slide}
\label{sec:handleslide}
We are now ready to establish an isomorphism 
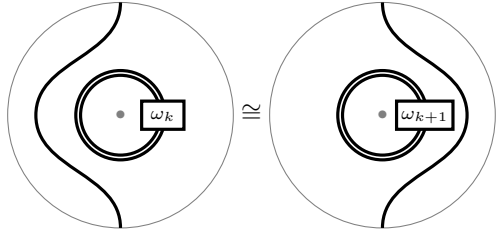
\[
\begin{tikzpicture}[anchorbase,scale=.75,smallnodes]
	\node[gray] at (0,0) {$\bullet$};
	\draw[gray] (0,0) circle (2);
	\draw[very thick] (0,2) to [out=270,in=90] (-1.5,0) to [out=270,in=90] (0,-2);
	\draw[very thick, double] (0,0) circle (.75);
	\filldraw[white] (.375,.25) rectangle (1.125,-.25); 
	\draw[very thick] (.375,.25) rectangle (1.125,-.25);
	\node at (.75,0) {\scriptsize $\kirby_k$};
	\end{tikzpicture}
	\cong
	\begin{tikzpicture}[anchorbase,rotate=180,scale=.75,smallnodes]
	\node[gray] at (0,0) {$\bullet$};
	\draw[gray] (0,0) circle (2);
	\draw[very thick] (0,2) to [out=270,in=90] (-1.5,0) to [out=270,in=90] (0,-2);
	\draw[very thick, double] (0,0) circle (.75);
	\filldraw[white] (-.25,.25) rectangle (-1.25,-.25); 
	\draw[very thick] (-.25,.25) rectangle (-1.25,-.25);
	\node at (-.75,0) {$\kirby_{k+1}$};
	\end{tikzpicture}
\]
of directed systems in $\cPBN^1_1$, which we describe 
in the equivalent category 
\[
\ctpc := \Ind(\Mat(\Kar(\tpc))) \, .
\]
For notational convenience, we will use the involution $\tau\colon \tpc\rightarrow \tpc$ 
sending $L \mapsto R$ and $\tau(X\bullet Z)=X\bullet \tau(Z)$ for all $X\in \dTL$. 

\begin{lem}[Elementary Handle Slide] 
	\label{lem:handleslide} 
Let $Z\in \{L,R\}$ and $k \in \Z$. There is an isomorphism
\[
\kirby_k \bullet Z \cong  \kirby_{k+1}\bullet \tau(Z) \,.
\]
\end{lem}
\begin{proof}
We may assume that $k\geq0$.
Consider the following diagram:
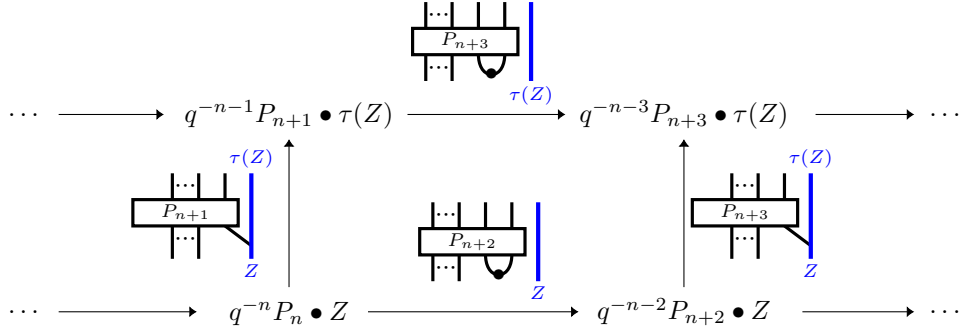
\[
	\begin{tikzpicture}[anchorbase, scale=.875]
		\node at (6,3) {$q^{-n-1} \symobj_{n+1}\bullet \tau(Z)$};
		\node at (12,3) {$q^{-n-3} \symobj_{n+3}\bullet \tau(Z)$};
		\node at (6,0) {$q^{-n}\symobj_n \bullet Z$};
		\node at (12,0) {$q^{-n-2}\symobj_{n+2}\bullet Z$};
		\node at (9,3.875) {
			\begin{tikzpicture}[anchorbase,xscale=-.7,yscale=.7]
				\draw[ultra thick,blue] (-1.25,-.75) node[below=-2pt]{\scs$\tau(Z)$} to (-1.25,.75);
				\draw[very thick] (.75,-.75) to (.75,.75);
				\node at (.5,-.5){\mysdots};
				\node at (.5,.5){\mysdots};
				\draw[very thick] (.25,-.75) to (.25,.75);
				\draw[very thick] (-.75,.75) to (-.75,-.25) to [out=270,in=180] (-.5,-.625)
					node{$\bullet$} to [out=0,in=270] (-.25,-.25) to (-.25,.75);
				\filldraw[white] (-1,.25) rectangle (1,-.25); 
				\draw[very thick] (-1,.25) rectangle (1,-.25);
				\node at (0,0) {\scriptsize$\JW_{n+3}$};
			\end{tikzpicture}
		};
		\node at (9,.875) {
			\begin{tikzpicture}[anchorbase,xscale=-.7,yscale=.7]
				\draw[ultra thick,blue] (-1.25,-.75) node[below=-2pt]{\scs$Z$} to (-1.25,.75);
				\draw[very thick] (.75,-.75) to (.75,.75);
				\node at (.5,-.5){\mysdots};
				\node at (.5,.5){\mysdots};
				\draw[very thick] (.25,-.75) to (.25,.75);
				\draw[very thick] (-.75,.75) to (-.75,-.25) to [out=270,in=180] (-.5,-.625)
					node{$\bullet$} to [out=0,in=270] (-.25,-.25) to (-.25,.75);
				\filldraw[white] (-1,.25) rectangle (1,-.25); 
				\draw[very thick] (-1,.25) rectangle (1,-.25);
				\node at (0,0) {\scriptsize$\JW_{n+2}$};
			\end{tikzpicture}
		};
		\node at (4.75,1.5) {
			\begin{tikzpicture}[anchorbase,xscale=-.7,yscale=.7]
			\draw[very thick] (-1.25,-.625) to (-.75,-.25);
			\draw[ultra thick,blue] (-1.25,-.875) node[below=-2pt]{\scs$Z$}  to (-1.25,0) 
				to (-1.25,.75) node[above=-2pt]{\scs$\tau(Z)$} ;
			\draw[very thick] (.25,-.875) to (.25,.75);
			\node at (0,-.5){\mysdots};
			\node at (0,.5){\mysdots};
			\draw[very thick] (-.25,-.875) to (-.25,.75);
			\draw[very thick] (-.75,0) to (-.75,.75);
			\filldraw[white] (-1,.25) rectangle (1,-.25); 
			\draw[very thick] (-1,.25) rectangle (1,-.25);
			\node at (0,0) {\scriptsize$\JW_{n+1}$};
			\end{tikzpicture}
		};
		\node at (13.25,1.5) {
			\begin{tikzpicture}[anchorbase,xscale=-.7,yscale=.7]
			\draw[very thick] (-1.25,-.625) to (-.75,-.25);
			\draw[ultra thick,blue] (-1.25,-.875) node[below=-2pt]{\scs$Z$}  to (-1.25,0) 
				to (-1.25,.75) node[above=-2pt]{\scs$\tau(Z)$} ;
			\draw[very thick] (.25,-.875) to (.25,.75);
			\node at (0,-.5){\mysdots};
			\node at (0,.5){\mysdots};
			\draw[very thick] (-.25,-.875) to (-.25,.75);
			\draw[very thick] (-.75,0) to (-.75,.75);
			\filldraw[white] (-1,.25) rectangle (1,-.25); 
			\draw[very thick] (-1,.25) rectangle (1,-.25);
			\node at (0,0) {\scriptsize$\JW_{n+3}$};
			\end{tikzpicture}
		};
		\draw[->] (6,.375) to (6,2.625);
		\draw[->] (12,.375) to (12,2.625);
		\draw[->] (2.5,3) to (4.1,3);
		\draw[->] (2.5,0) to (4.6,0);
		\draw[->] (7.8,3) to (10.2,3);
		\draw[->] (7.2,0) to (10.4,0);
		\draw[->] (14,3) to (15.5,3);
		\draw[->] (13.8,0) to (15.5,0);
		\node at (2,0) {$\dots$};
		\node at (2,3) {$\dots$};
		\node at (16,0) {$\dots$};
		\node at (16,3) {$\dots$};
		\end{tikzpicture}	
\]
which commutes by \eqref{eq:dAwd}.
This constitutes a natural transformation of directed systems 
and hence determines a morphism $\kirby_k\bullet Z \rightarrow \kirby_{k+1}\bullet \tau(Z)$
in $\ctpc$; see Remark~\ref{rmk:C^I to Ind(C)}.
The composition of two such maps is computed component-wise as follows:
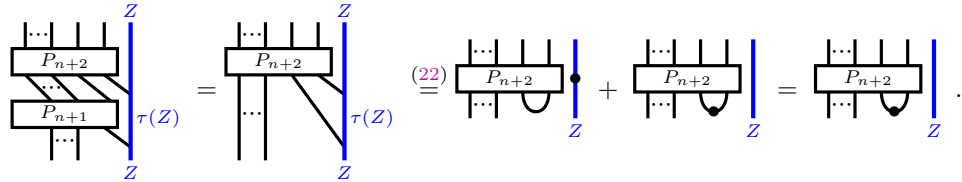
\[
	\begin{tikzpicture}[anchorbase,xscale=-.7,yscale=.7]
		\draw[very thick] (-1.25,-.625) to (-.75,-.25);
		\draw[very thick] (-1.25,.375) to (-.75,.75) to (-.75,1.75);
		\draw[ultra thick,blue] (-1.25,-.875) node[below=-2pt]{\scs$Z$} to (-1.25,0.125) 
			node[right=-2pt,yshift=-4pt]{\scs$\tau(Z)$} to (-1.25,1.75) node[above=-2pt]{\scs$Z$};
		\draw[very thick] (.25,-.875) to (.25,0.25) to (.75,.75) to (.75,1.75);
		\node at (0,-.5){\mysdots};
		\node at (0.25,.5){\mysdots};
		\node at (0.5,1.5){\mysdots};
		\draw[very thick] (-.25,-.875) to (-.25,0.25) to (.25,.75) to (.25,1.75);
		\draw[very thick]  (-.875,0.25) to (-.25,.75) to (-.25,1.75);
		\filldraw[white] (-1,.25) rectangle (1,-.25); 
		\draw[very thick] (-1,.25) rectangle (1,-.25);
		\node at (0,0) {\scriptsize$\JW_{n+1}$};
		\filldraw[white] (-1,1.25) rectangle (1,.75); 
		\draw[very thick] (-1,1.25) rectangle (1,.75);
		\node at (0,1) {\scriptsize$\JW_{n+2}$};
		\end{tikzpicture}	
		=	
		\begin{tikzpicture}[anchorbase,xscale=-.7,yscale=.7]
		\draw[very thick] (-1.25,-.625) to (-.25,.75);
		\draw[very thick] (-1.25,.375) to (-.75,.75) to (-.75,1.75);
		\draw[ultra thick,blue] (-1.25,-.875) node[below=-2pt]{\scs$Z$}  to (-1.25,0.125) 
			node[right=-2pt,yshift=-4pt]{\scs$\tau(Z)$} to (-1.25,1.75) node[above=-2pt]{\scs$Z$};
		\draw[very thick] (.75,-.875) to (.75,1.75);
		\node at (0.5,0){\mysdots};
		\node at (0.5,1.5){\mysdots};
		\draw[very thick] (.25,-.875)  to (.25,1.75);
		\draw[very thick]  (-.25,.75) to (-.25,1.75);
		\filldraw[white] (-1,1.25) rectangle (1,.75); 
		\draw[very thick] (-1,1.25) rectangle (1,.75);
		\node at (0,1) {\scriptsize$\JW_{n+2}$};
		\end{tikzpicture}
		\stackrel{\eqref{eq:MdTLrels}}{=}
		\begin{tikzpicture}[anchorbase,xscale=-.7,yscale=.7]
			\draw[ultra thick,blue] (-1.25,-.75) node[below=-2pt]{\scs$Z$} to node[black]{$\bullet$} (-1.25,.75);
			\draw[very thick] (.75,-.75) to (.75,.75);
			\node at (.5,-.5){\mysdots};
			\node at (.5,.5){\mysdots};
			\draw[very thick] (.25,-.75) to (.25,.75);
			\draw[very thick] (-.75,.75) to (-.75,-.25) to [out=270,in=180] (-.5,-.625)
				 to [out=0,in=270] (-.25,-.25) to (-.25,.75);
			\filldraw[white] (-1,.25) rectangle (1,-.25); 
			\draw[very thick] (-1,.25) rectangle (1,-.25);
			\node at (0,0) {\scriptsize$\JW_{n+2}$};
		\end{tikzpicture}
		+ \
		\begin{tikzpicture}[anchorbase,xscale=-.7,yscale=.7]
			\draw[ultra thick,blue] (-1.25,-.75) node[below=-2pt]{\scs$Z$} to (-1.25,.75);
			\draw[very thick] (.75,-.75) to (.75,.75);
			\node at (.5,-.5){\mysdots};
			\node at (.5,.5){\mysdots};
			\draw[very thick] (.25,-.75) to (.25,.75);
			\draw[very thick] (-.75,.75) to (-.75,-.25) to [out=270,in=180] (-.5,-.625)
				node{$\bullet$} to [out=0,in=270] (-.25,-.25) to (-.25,.75);
			\filldraw[white] (-1,.25) rectangle (1,-.25); 
			\draw[very thick] (-1,.25) rectangle (1,-.25);
			\node at (0,0) {\scriptsize$\JW_{n+2}$};
		\end{tikzpicture}
		= \
		\begin{tikzpicture}[anchorbase,xscale=-.7,yscale=.7]
			\draw[ultra thick,blue] (-1.25,-.75) node[below=-2pt]{\scs$Z$} to (-1.25,.75);
			\draw[very thick] (.75,-.75) to (.75,.75);
			\node at (.5,-.5){\mysdots};
			\node at (.5,.5){\mysdots};
			\draw[very thick] (.25,-.75) to (.25,.75);
			\draw[very thick] (-.75,.75) to (-.75,-.25) to [out=270,in=180] (-.5,-.625)
				node{$\bullet$} to [out=0,in=270] (-.25,-.25) to (-.25,.75);
			\filldraw[white] (-1,.25) rectangle (1,-.25); 
			\draw[very thick] (-1,.25) rectangle (1,-.25);
			\node at (0,0) {\scriptsize$\JW_{n+2}$};
		\end{tikzpicture} \, .
\]
Thus, the components of the composition 
$\kirby_k\bullet Z \to \kirby_{k+1}\bullet \tau(Z) \to \kirby_{k+2}\bullet Z$
are the transition maps in the directed system $\kirby_k\bullet Z$.
This is precisely the inverse to the isomorphism 
$\big(\kirby_{k+2} \xrightarrow{\cong} \kirby_k \big) \bullet \id_Z$
induced from Lemma~\ref{lem:shift}.
\end{proof}

For the following, we consider the (completed)
$4$-point category $\cPBN^2_2$ and its objects:
\[
LL :=
\begin{tikzpicture}[anchorbase,scale=1]
\node[gray] at (0,0) {$\bullet$};
\draw[gray] (0,0) circle (.5);
\draw[very thick] (0.15,.46) to [out=270,in=90] (-.15,0) to [out=270,in=90] (0.15,-.46);
\draw[very thick] (-.15,.46) to [out=270,in=90] (-.3,0) to [out=270,in=90] (-.15,-.46);
\end{tikzpicture}
\; , \quad
CC :=
\begin{tikzpicture}[anchorbase,scale=1]
\node[gray] at (0,0) {$\bullet$};
\draw[gray] (0,0) circle (.5);
\draw[very thick] (-0.15,.46) to [out=270,in=180] (0,.2) to [out=0,in=270] (0.15,.46);
\draw[very thick] (-0.15,-.46) to [out=90,in=180] (0,-.2) to [out=0,in=90] (0.15,-.46);
\end{tikzpicture}
\; , \quad
RR :=
\begin{tikzpicture}[anchorbase,scale=1]
\node[gray] at (0,0) {$\bullet$};
\draw[gray] (0,0) circle (.5);
\draw[very thick] (0.15,.46) to [out=270,in=90] (.3,0) to [out=270,in=90] (0.15,-.46);
\draw[very thick] (-.15,.46) to [out=270,in=90] (.15,0) to [out=270,in=90] (-.15,-.46);
\end{tikzpicture}
\; , \quad
LR :=
\begin{tikzpicture}[anchorbase,scale=1]
\node[gray] at (0,0) {$\bullet$};
\draw[gray] (0,0) circle (.5);
\draw[very thick] (0.15,.46) to [out=270,in=90] (.15,0) to [out=270,in=90] (0.15,-.46);
\draw[very thick] (-.15,.46) to [out=270,in=90] (-.15,0) to [out=270,in=90] (-.15,-.46);
\end{tikzpicture}
\; .
\]

\begin{lem}
	\label{lem:naturalhandleslide} 
There are commutative diagrams
\[
\begin{tikzcd}
	\kirby_k\bullet LL \arrow{r}{\cong}\arrow{d}{s} &  \kirby_{k+1}\bullet LR \arrow{r}{\cong}  & \kirby_{k+2} \bullet RR \arrow{d}{s} \\
	\kirby_k \bullet CC \arrow{rr}{\cong}  & & \kirby_{k+2} \bullet CC
	\end{tikzcd}
	\ , \qquad
\begin{tikzcd}
	\kirby_k \bullet L \arrow{r}{\cong}\arrow{d}{\id_\kirby \bullet x_L} & \kirby_{k+1}\bullet R \arrow{d}{\id_\kirby \bullet x_R} \\
	\kirby_k \bullet L \arrow{r}{\cong}  & \kirby_{k+1} \bullet R 
	\end{tikzcd}
\]
wherein all horizontal maps are 
elementary handle slide isomorphisms, 
$s$ denotes saddle cobordisms, and $x_L$ and $x_R$ 
are the dot endomorphisms of $L$ and $R$ illustrated in Theorem~\ref{thm:MdTL}.
 \end{lem}
 
In other words, the handle slide isomorphisms from Lemma~\ref{lem:handleslide} are
natural with respect to saddles and dot morphisms in $\PBN$. 

\begin{proof} 
Commutativity of the second diagram is clear from the second relation in 
\eqref{eq:MdTLrels}. 
Commutativity of the first is best seen when expressed in terms of
cobordisms as in Theorem~\ref{thm:MdTL},
where it corresponds to changing the order in which distant saddle cobordisms are applied.
\end{proof}

Recall that, for an $\mathcal{M}$-bimodule category $\mathcal{B}$ over a
monoidal category $\mathcal{M}$, the Drinfeld center of $\mathcal{B}$ is the
category, whose objects are pairs $(B,\tau_B)$ 
consisting of an object $B \in \mathcal{B}$
and a natural isomorphism 
$\tau_B \colon B \boxtimes - \Longrightarrow - \boxtimes B$ 
satisfying the usual coherence conditions of a half-braiding. The
following theorem uses the analogous (via categorification) concept of the
Drinfeld center of a bimodule over a monoidal \emph{bi}category. As we will not
use this formulation in the rest of the paper, we omit a detailed description
and study of such Drinfeld centers.

\begin{thm}[Handle Slide] \label{thm:handleslide}
The Kirby color $\omega$, viewed as an object of $\cPBN$, 
together with half-braidings assembled from the isomorphisms in Lemma~\ref{lem:handleslide}, 
defines an object of the Drinfeld center of the $\BN$-bimodule $\cPBN$.
\end{thm}

Here we view $\cPBN$ as a $\BN$-bimodule under the operation $\boxtimes$
described in item (4) of \S \ref{ss:punctured BN}.

\begin{proof}[Proof sketch] The objects of $\BN$ are generated 
(as $1$-morphisms in a monoidal bicategory, and under direct sums and shifts) 
by cup, cap, and identity tangles. 
The half-braidings can be assembled accordingly, and
their building blocks are only interesting for identity tangles, 
where we use the elementary handle slides from Lemma~\ref{lem:handleslide} 
(and their inverses). 
Lemma~\ref{lem:naturalhandleslide} implies that these candidate
half-braidings satisfy the requisite coherence conditions. 
\end{proof}

\section{Diagrammatic presentation}
\label{sec:diag}
In this section, we study the monoidal category obtained from $\dTL$ by adjoining
the objects $\kirby_0$ and $\kirby_1$.
Precisely, in Theorem \ref{thm:diagpres} below we give an explicit diagrammatic 
description of the full monoidal subcategory of $\cdTL$ 
generated by $\dTL$ and the Kirby objects $\kirby_0, \kirby_1$. 

\subsection{Kirby colored diagrammatics: first approximation}
\label{ss:ff1}
To begin, we define a strict monoidal category $\KdTL'$ via generators and relations and a monoidal functor $\KdTL'\rightarrow \cdTL$. 
This will give a useful framework for studying the Kirby objects in $\cdTL$.

\begin{definition}
	\label{def:KdTL prime}
Let $\KdTL'$ be the $\Z$-graded $\K$-linear monoidal category given by 
adjoining to $\dTL$ two new generating objects $\kirby_{i}'$ for $i\in \Z/2\Z$ 
(abbreviated in the diagrams below simply as $[i]$)
and morphisms
\begin{gather*}
\incl_n:= 
\begin{tikzpicture}[anchorbase,xscale=-1]
		\draw[very thick] (.35,0)  to [out=90,in=-45] (0,.5);
		\draw[very thick] (-.35,0)  to [out=90,in=-135] (0,.5);
		\draw[kirbystyle] (0,0.5) to (0,1) node[above=-2pt]{\scs$[n]$};
		\node at (0,.15) {$\mydots$};
		\draw[thick,decoration={brace},decorate] (-.4,-.1) -- node [below] {\scriptsize $n$} (.4,-.1);
\end{tikzpicture}
		\in \Hom^{-n}\!\big( \gen^n, \kirby'_n \big)
	\; , \quad
		\begin{tikzpicture}[anchorbase,scale=1]
			\draw[kirbystyle] (0,0) node[below=-2pt]{\scs$[i]$} to [out=90,in=225] (0.5,0.5) 
			to (0.5,1) node[above=-2pt]{\scs$[i{+}j]$};
			\draw[kirbystyle] (1,0) node[below=-2pt]{\scs$[j]$} to [out=90,in=315] (0.5,0.5);
		\end{tikzpicture}
	\in \Hom^0\!\big( \kirby'_i \otimes \kirby'_j , \kirby'_{i+j}\big)
	\; , \quad
		\begin{tikzpicture}[anchorbase,scale=1]
			\draw[kirbystyle] (0,0) node[below=-2pt]{\scs$[i]$} to 
			node[black]{$\bullet$} (0,1) node[above=-2pt]{\scs$[i]$};
		\end{tikzpicture}
	\in \End^2\!\big( \kirby'_i \big)
\end{gather*}
for $n\geq 0$ and $i,j \in \Z/2\Z$, modulo the defining relations
\begin{subequations}
\begin{gather}
	\label{eq:red algebra rels}
			\begin{tikzpicture}[anchorbase,scale=1]
				\draw[kirbystyle] (0,0)node{$\bullet$}  to [out=90,in=225] (0.5,0.5) 
				to (0.5,1) node[above=-2pt]{\scriptsize $[i]$};
				\draw[kirbystyle] (1,0) node[below=-2pt]{\scriptsize $[i]$} to [out=90,in=315] (0.5,0.5);
			\end{tikzpicture}
		=
			\begin{tikzpicture}[anchorbase,xscale=-1]
				\draw[kirbystyle] (0,0)node{$\bullet$}  to [out=90,in=225] (0.5,0.5) 
				to (0.5,1) node[above=-2pt]{\scriptsize $[i]$};
				\draw[kirbystyle] (1,0) node[below=-2pt]{\scriptsize $[i]$} to [out=90,in=315] (0.5,0.5);
			\end{tikzpicture}
		=
			\begin{tikzpicture}[anchorbase,xscale=-1]
				\draw[kirbystyle] (0,0) node[below=-2pt]{\scriptsize $[i]$} to (0,1) node[above=-2pt]{\scriptsize $[i]$};
			\end{tikzpicture}
		\; , \qquad
			\begin{tikzpicture}[anchorbase,xscale=.5]
				\draw[kirbystyle] (0,0) node[below=-2pt]{\scriptsize $[i]$} to [out=90,in=225] (0.5,0.3) 
				to [out=90,in=225] (1,.7) to (1,1) node[above=-2pt]{\scriptsize $[i{+}j{+}k]$};
				\draw[kirbystyle] (1,0) node[below=-2pt]{\scriptsize $[j]$} to [out=90,in=315] (0.5,0.3);
				\draw[kirbystyle] (2,0) node[below=-2pt]{\scriptsize $[k]$} to [out=90,in=315] (1,0.7);
			\end{tikzpicture}
		=
			\begin{tikzpicture}[anchorbase,xscale=-.5]
				\draw[kirbystyle] (0,0) node[below=-2pt]{\scriptsize $[k]$} to [out=90,in=225] (0.5,0.3) 
				to [out=90,in=225] (1,.7) to (1,1) node[above=-2pt]{\scriptsize $[i{+}j{+}k]$};
				\draw[kirbystyle] (1,0) node[below=-2pt]{\scriptsize $[j]$} to [out=90,in=315] (0.5,0.3);
				\draw[kirbystyle] (2,0) node[below=-2pt]{\scriptsize $[i]$} to [out=90,in=315] (1,0.7);
			\end{tikzpicture}
\\
	\label{eq:dotslideKdTL}
		\begin{tikzpicture}[anchorbase,scale=1]
			\draw[kirbystyle] (0,0) node[below=-2pt]{\scriptsize $[i]$} to [out=90,in=225] (0.5,0.5) to (.5,.7) node[black]{$\bullet$} 
			to (0.5,1) node[above=-2pt]{\scriptsize $[i{+}j]$};
			\draw[kirbystyle] (1,0) node[below=-2pt]{\scriptsize $[j]$} to [out=90,in=315] (0.5,0.5);
		\end{tikzpicture}
		\!\!=\!\!
		\begin{tikzpicture}[anchorbase,scale=1]
			\draw[kirbystyle] (0,0) node[below=-2pt]{\scriptsize $[i]$} to [out=90,in=225] (0.5,0.5)  
			to (0.5,1) node[above=-2pt]{\scriptsize $[i{+}j]$};
			\draw[kirbystyle] (1,0) node[below=-2pt]{\scriptsize $[j]$} to [out=90,in=315] (0.5,0.5);
			\node at (.25,.3) {$\bullet$};
		\end{tikzpicture}
		\!\!\!+(-1)^i\!\!\!
		\begin{tikzpicture}[anchorbase,scale=1]
			\draw[kirbystyle] (0,0) node[below=-2pt]{\scriptsize $[i]$} to [out=90,in=225] (0.5,0.5)  
			to (0.5,1) node[above=-2pt]{\scriptsize $[i{+}j]$};
			\draw[kirbystyle] (1,0) node[below=-2pt]{\scriptsize $[j]$} to [out=90,in=315] (0.5,0.5);
			\node at (.75,.3) {$\bullet$};
		\end{tikzpicture}
	\; , \qquad		
\begin{tikzpicture}[anchorbase,xscale=-1]
	\draw[very thick] (.35,0)  to [out=90,in=-45] (0,.5);
	\draw[very thick] (-.35,0)  to [out=90,in=-135] (0,.5);
	\draw[kirbystyle] (0,0.5) to (0,1) node[above=-2pt]{\scriptsize $[n]$};
	\node at (0,.15) {$\mydots$};
	\node at(0,.75) {$\bullet$};
	\draw[thick,decoration={brace},decorate] (-.4,-.1) -- node [below] {\scriptsize $n$} (.4,-.1);
\end{tikzpicture}
=
\begin{tikzpicture}[anchorbase,xscale=-1]
	\draw[very thick] (.35,-.5) to (.35,0)  to [out=90,in=-45] (0,.5);
	\draw[very thick] (-.35,-.5) to (-.35,0)  to [out=90,in=-135] (0,.5);
	\draw[kirbystyle] (0,0.5) to (0,1) node[above=-2pt]{\scriptsize $[n]$};
	\node at (0,.15) {$\mydots$};
	\filldraw[white] (-.5,-.3) rectangle (.5,0); 
	\draw[very thick] (-.5,-.3) rectangle (.5,0); 
	\node at (0,-.15) {\scriptsize$\xs_{n}$};
\end{tikzpicture}
\\
\label{eq:red relation c}
\begin{tikzpicture}[anchorbase,scale=1]
	\draw[very thick] (-1.05,-.5)  to [out=90,in=-155] (0,.5);
	\draw[very thick] (-.4,-.5)  to [out=90,in=-135] (0,.5);
	\draw[very thick] (.4,-.5)  to [out=90,in=-45] (0,.5);
	\draw[very thick] (1.05,-.5)  to [out=90,in=-25] (0,.5);
	\draw[very thick] (0,.5) to [out=-115,in=90] (-.2,-.1)   to[out=-90,in=180] (0,-.3) to[out=0,in=-90](.2,-.1) to [out=90,in=-65] (0,.5);
	\draw[kirbystyle] (0,0.5) to (0,1) node[above=-2pt]{\scriptsize $[i{+}j{+}2]$};
	\node at (-.7,-.25) {$\mydots$};
	\node at (.7,-.25) {$\mydots$};
	\draw[thick,decoration={brace,mirror},decorate] (-1.1,-.6) -- node [below] {\scriptsize $i$} (-.35,-.6);
	\draw[thick,decoration={brace},decorate] (1.1,-.6) -- node [below] {\scriptsize $j$} (.35,-.6);
\end{tikzpicture}
= 0
\; , \qquad	
\begin{tikzpicture}[anchorbase,scale=1]
	\draw[very thick] (-1.05,-.5)  to [out=90,in=-155] (0,.5);
	\draw[very thick] (-.4,-.5)  to [out=90,in=-135] (0,.5);
	\draw[very thick] (.4,-.5)  to [out=90,in=-45] (0,.5);
	\draw[very thick] (1.05,-.5)  to [out=90,in=-25] (0,.5);
	\draw[very thick] (0,.5) to [out=-115,in=90] (-.2,-.1)   to[out=-90,in=180] (0,-.3)node{$\bullet$} to[out=0,in=-90](.2,-.1) to [out=90,in=-65] (0,.5);
	\draw[kirbystyle] (0,0.5) to (0,1) node[above=-2pt]{\scriptsize $[i{+}j{+}2]$};
	\node at (-.7,-.25) {$\mydots$};
	\node at (.7,-.25) {$\mydots$};
	\draw[thick,decoration={brace,mirror},decorate] (-1.1,-.6) -- node [below] {\scriptsize $i$} (-.35,-.6);
	\draw[thick,decoration={brace},decorate] (1.1,-.6) -- node [below] {\scriptsize $j$} (.35,-.6);
\end{tikzpicture}
=
\begin{tikzpicture}[anchorbase,xscale=-1]
		\draw[very thick] (.35,0)  to [out=90,in=-45] (0,.5);
		\draw[very thick] (-.35,0)  to [out=90,in=-135] (0,.5);
		\draw[kirbystyle] (0,0.5) to (0,1) node[above=-2pt]{\scriptsize $[i{+}j]$};
		\node at (0,.15) {$\mydots$};
	\end{tikzpicture}
	\; , \qquad	
\begin{tikzpicture}[anchorbase,scale=1]
	\draw[kirbystyle] (0,0) to [out=90,in=225] (0.5,0.5) 
		to (0.5,1) node[above=-2pt]{\scriptsize $[i{+}j]$};
	\draw[kirbystyle] (1,0) to [out=90,in=315] (0.5,0.5);
	\draw[very thick] (.35,-.5)  to [out=90,in=-45] (0,0);
	\draw[very thick] (-.35,-.5)  to [out=90,in=-135] (0,0);
	\node at (0,-.35) {$\mydots$};
	\draw[very thick] (1.35,-.5)  to [out=90,in=-45] (1,0);
	\draw[very thick] (.65,-.5)  to [out=90,in=-135] (1,0);
	\node at (1,-.35) {$\mydots$};	
	\draw[thick,decoration={brace,mirror},decorate] (-.4,-.6) -- node [below] {\scriptsize $i$} (.4,-.6);
	\draw[thick,decoration={brace},decorate] (1.4,-.6) -- node [below] {\scriptsize $j$} (.6,-.6);
\end{tikzpicture}
=
\begin{tikzpicture}[anchorbase,xscale=-1]
		\draw[very thick] (.35,0)  to [out=90,in=-45] (0,.5);
		\draw[very thick] (-.35,0)  to [out=90,in=-135] (0,.5);
		\draw[kirbystyle] (0,0.5) to (0,1) node[above=-2pt]{\scriptsize $[i{+}j]$};
		\node at (0,.15) {$\mydots$};
	\end{tikzpicture}
\end{gather}
\end{subequations}
for all integers $i,j,n \geq 0$.
Here (for $\incl_0$) and below (for $\incl_1$), we use the notation
\[
\incl_0 :=
\begin{tikzpicture}[anchorbase,xscale=1]
\draw[kirbystyle] (0,0)node{$\bullet$} to (0,.5) node[above=-2pt]	{\scriptsize $[0]$};
\end{tikzpicture}
\, , \quad
\incl_1 :=
\begin{tikzpicture}[anchorbase,xscale=1]
\draw[very thick] (0,-.5) to (0,0);
\draw[kirbystyle] (0,0) node{$\bullet$} to (0,.5) node[above=-2pt]	{\scriptsize $[0]$};
\end{tikzpicture} \, .
\]
\end{definition}

Since all of the new generating morphisms have codomain $\kirby'_i$ 
and all new relations involve such morphisms, 
it follows that there is a fully faithful inclusion $\dTL\rightarrow \KdTL'$.  
An object of $\Kar(\KdTL')$ will be called \emph{finite} if it is in the image of $\Kar(\dTL)$ 
and will be called \emph{infinite} otherwise. 
In other words, an infinite object has at least one tensor factor of the form $\kirby'_i$.

\begin{remark}
Relations \eqref{eq:red algebra rels} say that $\kirby'_0\oplus \kirby'_1$ 
has the structure of a $\Z/2\Z$-graded algebra object in $\Mat(\Kar(\KdTL'))$.
Further, note that the first relation in \eqref{eq:red relation c} pairs with \eqref{eq:JWrecur}
to show that 
\[
\begin{tikzpicture}[anchorbase,xscale=-1]
	\draw[very thick] (.35,0)  to [out=90,in=-45] (0,.5);
	\draw[very thick] (-.35,0)  to [out=90,in=-135] (0,.5);
	\draw[kirbystyle] (0,0.5) to (0,1) node[above=-2pt]{\scriptsize $[n]$};
	\node at (0,.15) {$\mydots$};
	\draw[thick,decoration={brace},decorate] (-.4,-.1) -- node [below] {\scriptsize $n$} (.4,-.1);
\end{tikzpicture}
=
\begin{tikzpicture}[anchorbase,xscale=-1]
	\draw[very thick] (.35,-.5) to (.35,0)  to [out=90,in=-45] (0,.5);
	\draw[very thick] (-.35,-.5) to (-.35,0)  to [out=90,in=-135] (0,.5);
	\draw[kirbystyle] (0,0.5) to (0,1) node[above=-2pt]{\scriptsize $[n]$};
	\node at (0,.2) {$\mydots$};
	\filldraw[white] (-.5,-.35) rectangle (.5,.05); 
	\draw[very thick] (-.5,-.35) rectangle (.5,.05); 
	\node at (0,-.15) {\scriptsize$\JW_{n}$};
\end{tikzpicture} \, .
\]
\end{remark}

\begin{remark}\label{rem:involution} 
The category $\KdTL'$ admits a $\otimes$-contravariant involution
defined on diagrams by reflecting in a vertical line and multiplying every dot by the sign 
$(-1)^{i-1}$ where $i$ is the parity of the strand: 
$1$ for black strands and $i$ for a \kirbcolword strand with label $[i]$. 
\end{remark}

The following makes the connection between the formally defined category
$\KdTL'$ and $\cdTL$.

\begin{thm}
	\label{thm:from diag kirby to kirby}
	There exists a unique monoidal functor $\phi \colon \KdTL' \to \cdTL$ such that:
	\begin{itemize}
	\item $\phi$ extends the embedding $\dTL \to \cdTL$
	\item $\phi(\kirby'_{i}) = \kirby_{i}$.
	\item $\phi$ is given on the generating morphisms of $\KdTL'$ as follows:
	\begin{gather*}
		\begin{tikzpicture}[anchorbase,xscale=-1]
		\draw[very thick] (.35,0)  to [out=90,in=-45] (0,.5);
		\draw[very thick] (-.35,0)  to [out=90,in=-135] (0,.5);
		\draw[kirbystyle] (0,0.5) to (0,1) node[above=-2pt]{\scriptsize $[m]$};
		\node at (0,.15) {$\mydots$};
		\draw[thick,decoration={brace},decorate] (-.4,-.1) -- node [below] {\scriptsize $m$} (.4,-.1);
\end{tikzpicture}
\xmapsto{\phi}
	\left[
\begin{tikzpicture}[anchorbase,scale=1]
	\draw[very thick] (-.25,-.75) to (-.25,.75);
	\node at (0,-.5){\mysdots};
	\node at (0,.5){\mysdots};
	\draw[very thick] (.25,-.75) to (.25,.75);
	\filldraw[white] (-.5,.25) rectangle (.5,-.25); 
	\draw[very thick] (-.5,.25) rectangle (.5,-.25);
	\node at (0,0) {\scriptsize$\JW_{m}$};
\end{tikzpicture}
		\right]
	\, , \quad
	\begin{tikzpicture}[anchorbase,scale=1]
		\draw[kirbystyle] (0,0) node[below=-2pt]{\scriptsize $[i]$} to [out=90,in=225] (0.5,0.5) 
		to (0.5,1) node[above=-2pt]{\scriptsize $[i{+}j]$};
		\draw[kirbystyle] (1,0) node[below=-2pt]{\scriptsize $[j]$} to [out=90,in=315] (0.5,0.5);
	\end{tikzpicture} \xmapsto{\phi}
	\left\{ \left[
	\begin{tikzpicture}[anchorbase,scale=1]
		\draw[very thick] (-.75,-1.5) to (-.75,.5);
		\draw[very thick] (.75,-1.5) to (.75,.5);
		\node at (-.5,-1.375){\mysdots};
		\node at (-.5,-.5){\mysdots};
		\node at (-.5,.375){\mysdots};
		\node at (.5,-1.375){\mysdots};
		\node at (.5,-.5){\mysdots};
		\node at (.5,.375){\mysdots};
		\draw[very thick] (-.25,-1.5) to (-.25,.5);
		\draw[very thick] (.25,-1.5) to (.25,.5);
		\filldraw[white] (-1,.25) rectangle (1,-.25); 
		\draw[very thick] (-1,.25) rectangle (1,-.25);
		\node at (0,0) {\scriptsize$\JW_{i+j+4n}$};
		\filldraw[white] (-1,-.75) rectangle (-.1,-1.25); 
		\draw[very thick] (-1,-.75) rectangle (-.1,-1.25);
		\node at (-.55,-1) {\scriptsize$\JW_{i+2n}$};
		\filldraw[white] (1,-.75) rectangle (0.1,-1.25); 
		\draw[very thick] (1,-.75) rectangle (0.1,-1.25);
		\node at (.55,-1) {\scriptsize$\JW_{j+2n}$};
	\end{tikzpicture}
	\right]\right\}_{n\in \N}
	\, , \quad
	\begin{tikzpicture}[anchorbase,scale=1]
		\draw[kirbystyle] (0,0) node[below=-2pt]{\scriptsize $[i]$} to 
		node[black]{$\bullet$} (0,1) node[above=-2pt]{\scriptsize $[i]$};
	\end{tikzpicture}
	\xmapsto{\phi} 
	\{[\xs_{i+2n}]\}_{n\in \N} \, .
	\end{gather*}
	\end{itemize}
	\end{thm}
\begin{proof}
 Uniqueness is clear, since any such $\phi$ extends the embedding of
$\dTL$ and the images of the additional generators of $\KdTL'$ are explicitly
specified. It thus suffices to check that these assignments $\phi$ actually
define a functor $\KdTL' \to \cdTL$, i.e.~that $\phi$ respects all defining
relations from Definition~\ref{def:KdTL prime}. This follows easily, once we unpack the notation for
the images of the generating morphisms.

For the first, note that 
\[
\Hom_{\cdTL}(\gen^m , \kirby_m) = \Hom_{\cdTL}\big(\gen^m , \colim_{k \in \N} q^{-m-2k} \JW_{m+2k} \big)
= \colim_{k \in \N} \Hom_{\dTL}(\gen^m , q^{-m-2k} \JW_{m+2k})
\]
so $[ \JW_m ]$ simply denotes the equivalence class in the colimit 
of the map $\JW_m \in \Hom_{\dTL}(\gen^m , q^{-m} \JW_m)$.
For the second, it follows from Remark \ref{rem:Indmonoidal} and Lemma \ref{lem:final} that
\[
\kirby_i \otimes \kirby_j \cong \colim_{(k,\ell) \in \N^2} q^{-i-2k} \JW_{i+2k} \otimes q^{-j-2\ell} \JW_{j+2\ell}
\cong \colim_{n \in \N} q^{-i-j-4n} \JW_{i+2n} \otimes \JW_{j+2n} 
\]
so
\[
\Hom_{\cdTL}(\kirby_i \otimes \kirby_j , \kirby_{i+j}) 
\cong \lim_{n \in \N} \colim_{k \in \N} \Hom_{\dTL}(q^{-i-j-4n} \JW_{i+2n} \otimes \JW_{j+2n} , q^{-i-j-4k} \JW_{i+j+4k} ) \, .
\]
(Here, we additionally use that $m+4\N \hookrightarrow m+2\N$ is final.)
Elements in this limit are given by a stable family (indexed by $n \in \N$) 
of morphisms in 
\[
\colim_{k \in \N} \Hom_{\dTL}(q^{-i-j-4n} \JW_{i+2n} \otimes \JW_{j+2n} , q^{-i-j-4k} \JW_{i+j+4k} )\, ,
\]
and morphisms in the latter can be exhibited as the equivalence class of a single morphism in 
\[
\Hom_{\dTL}(q^{-i-j-4n} \JW_{i+2n} \otimes \JW_{j+2n} , q^{-i-j-4k} \JW_{i+j+4n} ) \, .
\] 
The notation above denotes such a stable family of equivalence classes.
The image of the third generator is depicted analogously.
	\end{proof}

The diagrammatic category $\KdTL'$ is insufficient to describe $\cdTL$.
Indeed, it is a consequence of Corollary \ref{cor:kirby squared} below that 
the functor $\phi \colon \KdTL' \to \cdTL$ is not full. 
For example, that corollary shows that $\Hom_{\cdTL}(\kirby_0, \kirby_0 \otimes \kirby_0)$
is uncountably infinite-dimensional, 
while it is straightforward to see that $\Hom_{\KdTL'}(\kirby_0, \kirby_0 \otimes \kirby_0)$
is countably infinite dimensional.
This issue will be remedied in \S \ref{ss:completion}, where we extend $\KdTL'$ to allow
certain infinite sums of diagrams.

\subsection{More diagrammatic relations}
\label{ss:more rels}
Introduce the following shorthand:
\[
\begin{tikzpicture}[anchorbase,scale=1]
	\draw[very thick] (.5,.5) to (0,.5);
	\draw[kirbystyle] (0,0) to (0,1) ;
\end{tikzpicture}
 = 
\begin{tikzpicture}[anchorbase,scale=1]
	\draw[very thick] (.25,.5) to [out=-45,in=180] (.5,.35) to [out=0,in=180] (1,.65);
	\draw[kirbystyle] (0,0) node[below=-2pt]{\scriptsize $[i]$} to (0,1) node[above=-2pt]{\scriptsize $[i{+}1]$};
	\draw[kirbystyle] (0,.75) to (.25,.5) node{$\bullet$};
\end{tikzpicture}
\, , \quad
\begin{tikzpicture}[anchorbase,xscale=-1]
	\draw[very thick] (.5,.5) to (0,.5);
	\draw[kirbystyle] (0,0) to (0,1) ;
\end{tikzpicture}
 = 
\begin{tikzpicture}[anchorbase,xscale=-1]
	\draw[very thick] (.25,.5) to [out=-45,in=180] (.5,.35) to [out=0,in=180] (1,.65);
	\draw[kirbystyle] (0,0) node[below=-2pt]{\scriptsize $[i]$} to (0,1) node[above=-2pt]{\scriptsize $[i{+}1]$};
	\draw[kirbystyle] (0,.75) to (.25,.5) node{$\bullet$};
\end{tikzpicture} \, .
\]

\begin{lemma}\label{lemma:secondaryrels} 
The following relations, as well as the images of these under the involution from Remark \ref{rem:involution}, 
hold in $\KdTL'$:
\begin{subequations}
\begin{gather}
\label{eq:red-black associativity}
\begin{tikzpicture}[anchorbase,scale=1]
	\draw[very thick] (.5,0) to (0.2,.3);
	\draw[kirbystyle] (0,0) node[below=-2pt]{\scriptsize $[i]$} to [out=90,in=225] (0.5,0.5) 
		to (0.5,1) node[above=-2pt]{\scriptsize $[i{+}j{+}1]$};
	\draw[kirbystyle] (1,0) node[below=-2pt]{\scriptsize $[j]$} to [out=90,in=315] (0.5,0.5);
\end{tikzpicture}
		=
\begin{tikzpicture}[anchorbase,scale=1]
	\draw[very thick] (.5,0) to (0.8,.3);
	\draw[kirbystyle] (0,0) node[below=-2pt]{\scriptsize $[i]$} to [out=90,in=225] (0.5,0.5) 
		to (0.5,1) node[above=-2pt]{\scriptsize $[i{+}j{+}1]$};
	\draw[kirbystyle] (1,0) node[below=-2pt]{\scriptsize $[j]$} to [out=90,in=315] (0.5,0.5);
\end{tikzpicture}
\, , \quad
\begin{tikzpicture}[anchorbase,scale=1]
	\draw[very thick] (-.5,.3) to  (0.2,.3);
	\draw[kirbystyle] (0,0) node[below=-2pt]{\scriptsize $[i]$} to [out=90,in=225] (0.5,0.5) 
		to (0.5,1) node[above=-2pt]{\scriptsize $[i{+}j{+}1]$};
	\draw[kirbystyle] (1,0) node[below=-2pt]{\scriptsize $[j]$} to [out=90,in=315] (0.5,0.5);
\end{tikzpicture}
=
\begin{tikzpicture}[anchorbase,scale=1]
	\draw[very thick] (-.5,.7) to (.5,.7);
	\draw[kirbystyle] (0,0) node[below=-2pt]{\scriptsize $[i]$} to [out=90,in=225] (0.5,0.5) 
		to (0.5,1) node[above=-2pt]{\scriptsize $[i{+}j{+}1]$};
	\draw[kirbystyle] (1,0) node[below=-2pt]{\scriptsize $[j]$} to [out=90,in=315] (0.5,0.5);
\end{tikzpicture}
\\
\label{eq:more dotsliding}
\begin{tikzpicture}[anchorbase,scale=1]
	\draw[very thick] (.5,.5) to (0,.5);
	\draw[kirbystyle] (0,0) node[below=-2pt]{\scriptsize $[i]$} to (0,1) node[above=-2pt]{\scriptsize $[i{+}1]$};
	\node at (0,.75) {$\bullet$};
\end{tikzpicture}
\  =
\begin{tikzpicture}[anchorbase,scale=1]
	\draw[very thick] (.5,.5) to (0,.5);
	\draw[kirbystyle] (0,0) node[below=-2pt]{\scriptsize $[i]$} to (0,1) node[above=-2pt]{\scriptsize $[i{+}1]$};
	\node at (0,.25) {$\bullet$};
\end{tikzpicture}
+(-1)^i\!\!\!\!
\begin{tikzpicture}[anchorbase,scale=1]
	\draw[very thick] (.5,.5) to (0,.5);
	\draw[kirbystyle] (0,0) node[below=-2pt]{\scriptsize $[i]$} to (0,1) node[above=-2pt]{\scriptsize $[i{+}1]$};
	\node at (0.25,.5) {$\bullet$};
\end{tikzpicture}
\\
\label{eq:belly}
 \begin{tikzpicture}[anchorbase,scale=1]
	\draw[very thick] (.7,.65) to [out=180,in=90] (.4,.5) to [out=270,in=180](.7,.35);
	\draw[kirbystyle] (.7,0) to (.7,1);
\end{tikzpicture}
=0
\,  , \quad
 \begin{tikzpicture}[anchorbase,scale=1]
	\draw[very thick] (.7,.65) to [out=180,in=90] (.4,.5) to [out=270,in=180](.7,.35);
	\draw[kirbystyle] (.7,0) to (.7,1);
	\node at (.4,.48) {$\bullet$};
\end{tikzpicture}
=
\begin{tikzpicture}[anchorbase,scale=1]
	\draw[kirbystyle] (.7,0) to (.7,1);
\end{tikzpicture}
\\
\label{eq:red line and poly rep}
\begin{tikzpicture}[anchorbase,scale=1]
	\draw[very thick] (0,.65) to [out=0,in=90] (.25,.5) to[out=270,in=0] (0,.35);
	\draw[kirbystyle] (.5,0) to (.5,1);
\end{tikzpicture}
\ = \
\begin{tikzpicture}[anchorbase,scale=1]
	\draw[very thick] (0,.65) to(.5,0.65);
	\node at (0.25,.63) {$\bullet$};
	\draw[very thick] (0,.35) to  (.5,0.35);
	\draw[kirbystyle] (.5,0) to (.5,1);
\end{tikzpicture}
+
\begin{tikzpicture}[anchorbase,scale=1]
	\draw[very thick] (0,.65) to  (.5,0.65);
	\draw[very thick] (0,.35) to (.5,0.35);
	\node at (0.25,.33) {$\bullet$};
	\draw[kirbystyle] (.5,0) to (.5,1);
\end{tikzpicture}
\\
\label{eq:dotted rung rels}
\begin{tikzpicture}[anchorbase,scale=1]
	\draw[very thick] (0,.65) to  (.5,0.65);
	\node at (0.25,.63) {$\bullet$};
	\draw[very thick] (0,.35) to (.5,0.35);
	\node at (0.25,.33) {$\bullet$};
	\draw[kirbystyle] (0,0) to (0,1);
	\draw[kirbystyle] (.5,0) to (.5,1);
\end{tikzpicture}
 = 
\begin{tikzpicture}[anchorbase,scale=1]
	\draw[kirbystyle] (0,0) to (0,1);
	\draw[kirbystyle] (.5,0) to (.5,1);
\end{tikzpicture}
\ , \quad
\begin{tikzpicture}[anchorbase,scale=1]
	\draw[very thick] (0,.65) to(.5,0.65);
	\node at (0.25,.63) {$\bullet$};
	\draw[very thick] (0,.35) to  (.5,0.35);
	\draw[kirbystyle] (0,0) to (0,1);
	\draw[kirbystyle] (.5,0) to (.5,1);
\end{tikzpicture}
 + 
\begin{tikzpicture}[anchorbase,scale=1]
	\draw[very thick] (0,.65) to  (.5,0.65);
	\draw[very thick] (0,.35) to (.5,0.35);
	\node at (0.25,.33) {$\bullet$};
	\draw[kirbystyle] (0,0) to (0,1);
	\draw[kirbystyle] (.5,0) to (.5,1);
\end{tikzpicture}
 = 0 \, .
\end{gather}
\end{subequations}
\end{lemma}
Observe that the only relation that picks up a sign under the left-right reflection is \eqref{eq:more dotsliding}, 
whose image yields the relation
\[
\begin{tikzpicture}[anchorbase,xscale=-1]
	\draw[very thick] (.5,.5) to (0,.5);
	\draw[kirbystyle] (0,0) node[below=-2pt]{\scriptsize $[i]$} to (0,1) node[above=-2pt]{\scriptsize $[i{+}1]$};
	\node at (0,.75) {$\bullet$};
\end{tikzpicture}
=
-
\begin{tikzpicture}[anchorbase,xscale=-1]
	\draw[very thick] (.5,.5) to (0,.5);
	\draw[kirbystyle] (0,0) node[below=-2pt]{\scriptsize $[i]$} to (0,1) node[above=-2pt]{\scriptsize $[i{+}1]$};
	\node at (0,.25) {$\bullet$};
\end{tikzpicture}
+ 
\begin{tikzpicture}[anchorbase,xscale=-1]
	\draw[very thick] (.5,.5) to (0,.5);
	\draw[kirbystyle] (0,0) node[below=-2pt]{\scriptsize $[i]$} to (0,1) node[above=-2pt]{\scriptsize $[i{+}1]$};
	\node at (0.25,.5) {$\bullet$};
\end{tikzpicture} \, .
\]

\begin{proof}
Relations \eqref{eq:red-black associativity}, \eqref{eq:more dotsliding}, and \eqref{eq:belly} are easy to check, 
thus we leave this as an exercise. 
Relation \eqref{eq:red line and poly rep} follows from
\[
\begin{tikzpicture}[anchorbase,scale=1]
	\draw[very thick] (0,.65) to [out=0,in=90] (.25,.5) to[out=270,in=0] (0,.35);
	\draw[kirbystyle] (.5,0) to (.5,1);
\end{tikzpicture}
 = 
\begin{tikzpicture}[anchorbase,scale=1]
	\draw[very thick] (-.2,.65) to [out=0,in=90] (.1,.5) to [out=270,in=0] (-.2,.35);
	\draw[very thick] (.7,.65) to [out=180,in=90] (.4,.5) to [out=270,in=180](.7,.35);
	\node at (.4,.48) {$\bullet$};
	\draw[kirbystyle] (.7,0) to (.7,1);
\end{tikzpicture}
 = 
\begin{tikzpicture}[anchorbase,scale=1]
	\draw[very thick] (0,.65) to(.5,0.65);
	\node at (0.25,.63) {$\bullet$};
	\draw[very thick] (0,.35) to  (.5,0.35);
	\draw[kirbystyle] (.5,0) to (.5,1);
\end{tikzpicture}
 + 
\begin{tikzpicture}[anchorbase,scale=1]
	\draw[very thick] (0,.65) to  (.5,0.65);
	\draw[very thick] (0,.35) to (.5,0.35);
	\node at (0.25,.33) {$\bullet$};
	\draw[kirbystyle] (.5,0) to (.5,1);
\end{tikzpicture}
 -  \begin{tikzpicture}[anchorbase,scale=1]
	\draw[very thick] (-.2,.65) to [out=0,in=90] (.1,.5) to [out=270,in=0] (-.2,.35);
	\draw[very thick] (.7,.65) to [out=180,in=90] (.4,.5) to [out=270,in=180](.7,.35);
	\node at (.1,.48) {$\bullet$};
	\draw[kirbystyle] (.7,0) to (.7,1);
\end{tikzpicture}
\]
together with cup annihilation relation in \eqref{eq:belly}.  The first relation in \eqref{eq:dotted rung rels} follows from
\[
\begin{tikzpicture}[anchorbase,scale=1]
	\draw[very thick] (0,.65) to  (.5,0.65);
	\node at (0.25,.63) {$\bullet$};
	\draw[very thick] (0,.35) to (.5,0.35);
	\node at (0.25,.33) {$\bullet$};
	\draw[kirbystyle] (0,0) to (0,1);
	\draw[kirbystyle] (.5,0) to (.5,1);
\end{tikzpicture}
\ \ =\ \ 
\begin{tikzpicture}[anchorbase,scale=1]
	\draw[very thick] (-.2,.65) to [out=0,in=90] (.1,.5) to [out=270,in=0] (-.2,.35);
	\draw[very thick] (.7,.65) to [out=180,in=90] (.4,.5) to [out=270,in=180](.7,.35);
	\node at (.1,.48) {$\bullet$};
	\node at (.4,.48) {$\bullet$};
	\draw[kirbystyle] (-.2,0) to (-.2,1);
	\draw[kirbystyle] (.7,0) to (.7,1);
\end{tikzpicture}
\]
followed by dotted cup absorption \eqref{eq:belly}. 
The second relation in \eqref{eq:dotted rung rels} follows from \eqref{eq:red line and poly rep} 
together with the undotted cup annihilation \eqref{eq:belly}.  
\end{proof}

\begin{definition}\label{def:shorthand}
For $n\in \N$ set:
\[
\begin{tikzpicture}[anchorbase,scale=1]
	\draw[kirbystyle] (0,0)node[below=-2pt]{\scriptsize $[k]$} to (0,1) node[above=-2pt]{\scriptsize $[k{+}i]$};
	\draw[kirbystyle] (1,0) node[below=-2pt]{\scriptsize $[\ell]$} to (1,1) node[above=-2pt]{\scriptsize $[\ell{+}i]$};
	\draw[kirbystyle] (0,.5) to (1,.5);
	\draw[fill=white] (.5,.5) circle (.15);
	\node at (.5,.2) {\small $n$};
	\node[kirb] at (.2,.3) {\scriptsize $[i]$};
\end{tikzpicture}
:= 
\begin{tikzpicture}[anchorbase,xscale=1]
	\draw[kirbystyle] (-.25,-.5)node[below=-2pt]{\scriptsize $[k]$} to (-.25,0) to[out=90,in=225] (0,.25) 
		to (0,.5)node[above=-2pt]{\scriptsize $[k{+}i]$};
	\draw[kirbystyle] (1.25,-.5)node[below=-2pt]{\scriptsize $[\ell]$} to(1.25,0) to [out=90,in=315] (1,.25) 
		to (1,.5)node[above=-2pt]{\scriptsize $[\ell{+}i]$};
	\draw[kirbystyle] (0,.25)[out=315,in=180] to (.5,0) to[out=0,in=225] (1,.25);
	\draw[fill=white] (.5,0) circle (.15);
	\node[kirb] at (.12,-.1) {\scriptsize$[i]$};
	\node at (.5,-.3) {\small $n$};
\end{tikzpicture}
\, ,  \quad
\begin{tikzpicture}[anchorbase,scale=1]
\draw[kirbystyle] (0,1)node[above=-2pt]{\scriptsize $[i]$} \dr (.4,.5) \ru (.8,1) node[above=-2pt]{\scriptsize $[i]$};
\draw[fill=white] (.4,.5) circle (.15);
\node at (.4,.2) {\small $n$};
\end{tikzpicture}
:=
\begin{cases}
\begin{tikzpicture}[anchorbase,scale=1]
\draw[kirbystyle] (0,0) to (0,.5)node[above=-2pt]{\scriptsize $[i]$};
\draw[kirbystyle] (1,0) to (1,.5)node[above=-2pt]{\scriptsize $[i]$};
\draw[very thick] (0,0) to [out=315,in=110] (.25,-.25) to[out=-70,in=180] (.5,-.4) to[out=0,in=250] (.75,-.25) to (1,0);
\draw[very thick] (0,0) to [out=225,in=90] (-.25,-.25) to[out=-90,in=180] (.5,-.8) to[out=0,in=270] (1.25,-.25)
 to[in=-45,out=90] (1,0);
\node[rotate=90] at (.5,-.6) {$\mydots$};
\node at (0,-.25) {$\mydots$};
\node at (1,-.25) {$\mydots$};
\end{tikzpicture}
& \text{ if $n\equiv i$ mod 2} \\
\begin{tikzpicture}[anchorbase,scale=1]
	\draw[kirbystyle] (0,0) to (0,.5)node[above=-2pt]{\scriptsize $[i]$};
	\draw[kirbystyle] (1,0) to (1,.5)node[above=-2pt]{\scriptsize $[i]$};
	\draw[very thick] (0,0) to [out=315,in=110] (.25,-.25) to[out=-70,in=180] (.5,-.4) 
		to[out=0,in=250] (.75,-.25) to (1,0);
	\draw[very thick] (0,0) to [out=225,in=90] (-.25,-.25) to[out=-90,in=180] (.5,-.8) to[out=0,in=270] (1.25,-.25)
		 to[in=-45,out=90] (1,0);
	 \draw[very thick] (0,0) to [out=200,in=90] (-.4,-.25) to[out=-95,in=180] (.5,-1) to[out=0,in=275] (1.4,-.25)
		 to[in=-20,out=90] (1,0);
	\node[rotate=90] at (.5,-.6) {$\mydots$};
	\node at (0,-.25) {$\mydots$};
	\node at (1,-.25) {$\mydots$};
	\node at (.5,-1) {$\bullet$};
\end{tikzpicture}
& \text{ if $n\not\equiv i$ mod 2}
\end{cases}
\]
where in both instances of the second formula there are $n$ undotted black strands running between the \kirbcolword strands. 
\end{definition}

Informally, the $n$-labeled hollow circle can be interpreted as indicating $n$ \emph{dots missing}. We record some relations involving these morphisms.

\begin{lemma}\label{lemma:rungs}
\begin{subequations}
\begin{gather}
\label{eq:holes eat dots}
\begin{tikzpicture}[anchorbase,xscale=-1]
\draw[kirbystyle] (0,1)node[above=-2pt]{\scriptsize $[i]$} \dr (.4,.5) \ru (.8,1) node[above=-2pt]{\scriptsize $[i]$};
\draw[fill=white] (.4,.5) circle (.15);
\node at (.4,.2) {\small $n$};
\node at (.02,.8) {$\bullet$};
\end{tikzpicture}
\ = \
(-1)^{n-1} n\cdot \begin{tikzpicture}[anchorbase,xscale=1]
\draw[kirbystyle] (0,1)node[above=-2pt]{\scriptsize $[i]$} \dr (.4,.5) \ru (.8,1) node[above=-2pt]{\scriptsize $[i]$};
\draw[fill=white] (.4,.5) circle (.15);
\node at (.4,.2) {\small $n{-}1$};
\end{tikzpicture}
\ = \
(-1)^{i-1}\begin{tikzpicture}[anchorbase,xscale=1]
\draw[kirbystyle] (0,1)node[above=-2pt]{\scriptsize $[i]$} \dr (.4,.5) \ru (.8,1) node[above=-2pt]{\scriptsize $[i]$};
\draw[fill=white] (.4,.5) circle (.15);
\node at (.4,.2) {\small $n$};
\node at (.02,.8) {$\bullet$};
\end{tikzpicture}
\\
\label{eq:dotted teardrop}
\begin{tikzpicture}[anchorbase,xscale=-1]
\draw[kirbystyle] (.4,1.5) to[out=225,in=90] (0,1) \dr (.4,.5) \ru (.8,1) to[out=90,in=-45] (.4,1.5) to (.4,1.75)node[above=-2pt]{\scriptsize $[0]$};
\draw[fill=white] (.4,.5) circle (.15);
\node at (.4,.2) {\small $n$};
\node at (.02,.8) {$\bullet$};
\node at (-.3,.8) {\small $m$};
\node[kirb] at (-.15,1.2) {\scriptsize$[i]$};
\node[kirb] at (.95,1.2) {\scriptsize$[i]$};
\end{tikzpicture}
\ = \ 
\d_{n,m}\cdot (-1)^{\binom{n}{2}}n! \cdot 
\begin{tikzpicture}[anchorbase,scale=1]
	\draw[kirbystyle] (0,0)node{$\bullet$} to (0,.5)node[above=-2pt]{\scriptsize $[0]$};
\end{tikzpicture}
\\
\label{eq:z on rungs}
\begin{tikzpicture}[anchorbase,scale=1]
	\draw[kirbystyle] (-.75,-.1)node[below]{\scriptsize $[i]$} to (-.75,1.1)node[above]{\scriptsize $[i{+}n]$};
	\draw[kirbystyle] (.75,-.1)node[below]{\scriptsize $[j]$} to (.75,1.1)node[above]{\scriptsize $[j{+}n]$};
	\draw[very thick] (-.75,.3) to (.75,.3);
	\draw[very thick] (-.75,.7) to (.75,.7);
	\filldraw[white] (-.35,0) rectangle (.35,1); 
	\draw[very thick] (-.35,0) rectangle (.35,1); 
	\node[rotate=90] at (0,.5) {$\xs_n^{n-k}$};
	\node[rotate=90] at (-.55,.5) {$\mydots$};
	\node[rotate=90] at (.55,.5) {$\mydots$};
\end{tikzpicture}
\ = \ 
(-1)^{\binom{n-k}{2}}\frac{n!}{k!}\begin{tikzpicture}[anchorbase,scale=1]
	\draw[kirbystyle] (-.5,.5) to (.5,0.5);
	\draw[kirbystyle] (-.5,0)node[below=-2pt]{\scriptsize $[i]$} to (-.5,1)node[above=-2pt]{\scriptsize $[i{+}n]$};
	\draw[kirbystyle] (.5,0)node[below=-2pt]{\scriptsize $[j]$} to (.5,1)node[above=-2pt]{\scriptsize $[j{+}n]$};
	\draw[fill=white] (0,.5) circle (.15);
	\node at (0,.2) {\small $k$};
\end{tikzpicture}
\ = \ 
(-1)^{k(n-k)}\begin{tikzpicture}[anchorbase,scale=1]
	\draw[kirbystyle] (-.75,-.1)node[below]{\scriptsize $[i]$} to (-.75,1.1)node[above]{\scriptsize $[i{+}n]$};
	\draw[kirbystyle] (.75,-.1)node[below]{\scriptsize $[j]$} to (.75,1.1)node[above]{\scriptsize $[j{+}n]$};
	\draw[very thick] (-.75,.3) to (.75,.3);
	\draw[very thick] (-.75,.7) to (.75,.7);
	\filldraw[white] (-.35,0) rectangle (.35,1); 
	\draw[very thick] (-.35,0) rectangle (.35,1); 
	\node[rotate=-90] at (0,.5) {$\xs_n^{n-k}$};
	\node[rotate=90] at (-.55,.5) {$\mydots$};
	\node[rotate=90] at (.55,.5) {$\mydots$};
\end{tikzpicture}
\end{gather}
\end{subequations}
\end{lemma}
In Equation \eqref{eq:holes eat dots},  it is understood that the expression in the middle is zero for $n=0$.
\begin{proof}
Relation \eqref{eq:holes eat dots} is an exercise 
(note one equation implies the other via the involution from Remark \ref{rem:involution}). 
Relation \eqref{eq:dotted teardrop} is immediate if $n=0$ or $m=0$. 
Otherwise we use \eqref{eq:holes eat dots} to decrease both $m$ and $n$ and the result follows by induction. 
Relation \eqref{eq:z on rungs} can be proved similarly by iterating \eqref{eq:holes eat dots}.
\end{proof}

\subsection{The polynomial representation, revisited}
\label{ss:poly redux}
\label{ss:tensor with kirby}

The diagrammatic category $\KdTL'$ allows for straightforward proofs of 
isomorphisms in $\cdTL$.
For instance, in light of relation \eqref{eq:red line and poly rep}, 
we may view the graphical calculus in $\KdTL'$ as an extension of 
the graphical calculus for the polynomial representation from Remark \ref{rem:diagrammatic poly}.
The following is a key consequence. 
Recall that we denote $\kirby = \kirby_0 \oplus \kirby_1 \in \cdTL$.

\begin{thm}
	\label{thm:tensor with kirby}
There is an isomorphism 
$\kirby \otimes - \cong \Pol(-)\otimes \kirby$ of functors $\cdTL\rightarrow \cdTL$.
\end{thm}
\begin{proof}
Remarks \ref{rem:compact} and \ref{rem:Indmonoidal} show that both
$\kirby\otimes -$ and $\Pol(-)\otimes \kirby$ preserve filtered 
colimits,
so Remark \ref{rem:absolute} implies that it suffices to give an isomorphism of restricted functors 
\[
(\kirby \otimes -)|_{\dTL}  \cong( \Pol(-)\otimes \kirby)|_{\dTL} \colon \dTL\rightarrow \cdTL \, .
\]  
Since functors preserve isomorphisms and using Theorem~\ref{thm:from diag kirby
to kirby}, it further suffices to show that there is an isomorphism $\kirby'
\otimes - \cong \Pol(-)\otimes \kirby'$ of functors $\dTL\rightarrow
\Mat(\Kar(\KdTL'))$, where $\kirby' := \kirby'_0 \oplus \kirby'_1$. Since $\Pol$
is monoidal, we need only establish the isomorphism on generating
objects/morphisms. Equation \eqref{eq:red line and poly rep} gives that 
\begin{equation}\label{eq:V kirby idemp}
\begin{tikzpicture}[anchorbase,xscale=1]
	\draw[very thick] (.5,1) to (.5,0);
	\draw[kirbystyle] (0,0) to (0,1);
\end{tikzpicture}
=
\begin{tikzpicture}[anchorbase,xscale=1]
	\draw[very thick] (.5,1) to (0,.6);
	\draw[very thick] (.5,0) to (0,.4);
	\draw[kirbystyle] (0,0) to (0,1);
	\node at (0.25,.2) {$\bullet$};
\end{tikzpicture}
+ \
\begin{tikzpicture}[anchorbase,xscale=1]
	\draw[very thick] (.5,1) to (0,.6);
	\draw[very thick] (.5,0) to (0,.4);
	\draw[kirbystyle] (0,0) to (0,1);
	\node at (0.25,.8) {$\bullet$};
\end{tikzpicture} \, ,
\end{equation}
and \eqref{eq:belly} implies that the summands on the right-hand side are
orthogonal idempotents. Hence, we have the following direct sum decomposition in
$\Mat(\Kar(\KdTL'))$:
\[
\kirby' \otimes \gen \cong q\inv \kirby' \oplus q \kirby' 
\cong q^{-1} \K[x]/(x^2) \otimes \kirby' = \Pol(\gen) \otimes \kirby' \, .
\]
It is straightforward to check that 
the induced isomorphisms
$\kirby' \otimes \gen^n \cong \Pol(\gen)^{\otimes n} \otimes \kirby' \cong \Pol(\gen^n) \otimes \kirby'$ 
are natural with respect to the generating morphisms 
in $\dTL$, essentially by Remark \ref{rem:diagrammatic poly}.
For example, the components of the isomorphism 
$\Pol(\gen) \otimes \kirby' \xrightarrow{\cong} \gen \otimes \kirby'$ 
corresponding to the subspaces $\K \cdot 1$ and $\K \cdot x$ in $\Pol(\gen)$ are
\[
\begin{tikzpicture}[anchorbase,xscale=1]
	\draw[very thick] (.5,1) to (0,.5);
	\draw[kirbystyle] (0,.25) to (0,1);
\end{tikzpicture}
\quad \text{and} \quad
\begin{tikzpicture}[anchorbase,xscale=1]
	\draw[very thick] (.5,1) to node{$\bullet$} (0,.5);
	\draw[kirbystyle] (0,.25) to (0,1);
\end{tikzpicture}
\]
and it follows from \eqref{eq:V kirby idemp}
that the dot generator in $\dTL$ maps the first component to the second 
and the second to zero, as desired.
\end{proof}

\begin{remark} 
Continuing along the lines of Remark~\ref{rem:dTLrepth},
Theorem \ref{thm:tensor with kirby} shows that 
in the presence of the Kirby color, 
the generating object $c$ of $\dTL$ acts as the multiplicity space $\Pol(c)=V$, 
which we can identify with 
the vector representation of $\slnn{2}$. 
In particular, \emph{any} linear endomorphism $\phi\in\End_{\K}(V)$ 
can be expressed diagrammatically as an element of $\End(\kirby_k \otimes c)$, 
e.g.~\eqref{eq:V kirby idemp} reflects the identity
\[\id_V=F\circ E + E \circ F \in \End_{\K}(V).\]
\end{remark}

The direct sum decomposition from the proof of Theorem~\ref{thm:tensor with
kirby} generalizes as follows.

\begin{proposition}\label{prop:Pn times kirby decomp}
The identity morphism of $\kirby \otimes \JW_n$ 
has the following decomposition into orthogonal idempotents in $KdTL'$:
\begin{equation}
	\label{eq:Pn times kirby decomp}
\begin{tikzpicture}[anchorbase,xscale=-1.2]
	\draw[kirbystyle] (.8,0) to (.8,2);
	\draw[very thick] (-.25,0) to (-.25,2);
	\draw[very thick] (.25,0) to (.25,2);
	\node at (0,.375) {$\mydots$};
	\node at (0,1.625) {$\mydots$};
	\filldraw[white] (-.4,.75) rectangle (.4,1.25); 
	\draw[very thick] (-.4,.75) rectangle (.4,1.25);
	\node at (0,1) {$\JW_n$};
\end{tikzpicture}
\ \ = \ \ 
\frac{(-1)^{\binom{n}{2}}}{n!}\sum_{k+l=n}
\begin{tikzpicture}[anchorbase,xscale=-1.2]
	\draw[kirbystyle] (.8,0) to (.8,2.8) ;
	\draw[very thick] (-.25,0) to (-.25 ,.75) to[out=90,in=180] (.8,1.3);
	\draw[very thick] (.25,0) to (.25 ,.75) to[out=90,in=180] (.8,1);
	\draw[very thick] (-.25,2.8) to (-.25 ,2.05) to[out=270,in=180] (.8,1.5);
	\draw[very thick] (.25,2.8) to (.25 ,2.05) to[out=270,in=180] (.8,1.8);
	\node at (.05,1.9) {$\mydots$};
	\node at (.05,.9) {$\mydots$};
	\filldraw[white] (-.4,2.65) rectangle (.35,2.05);
	\draw[very thick] (-.4,2.65) rectangle (.35,2.05);
	\filldraw[white] (-.4,.15) rectangle (.4,.75);
	\draw[very thick] (-.4,.15) rectangle (.4,.75);
	\node at (0,2.35) {${z_{n}^l}$};
	\node at (0,.45) {${z_{n}^k}$};
\end{tikzpicture}
\end{equation}
\end{proposition}
\begin{proof}
Iterating \eqref{eq:V kirby idemp} gives the relation
\begin{equation}
	\label{eq:idempotent}
\begin{tikzpicture}[anchorbase,xscale=-1.2]
	\draw[kirbystyle] (.75,0) to (.75,1.5);
	\draw[very thick] (0,0) to (0,1.5);
	\node at (.25,.5) {$\mydots$};
	\draw[very thick] (.5,0) to (.5,1.5);
	\draw[thick,decoration={brace},decorate] (0,-.1) -- node [below] {\scriptsize $n$} (.5,-.1);
	\draw[thick,decoration={brace},decorate] (0,-.1) -- node [below] {\scriptsize $n$} (.5,-.1);
\end{tikzpicture}
=
\sum_{f} \
\begin{tikzpicture}[anchorbase,xscale=-1.2]
	\draw[kirbystyle] (.8,0) to (.8,2.8) ;
	\draw[very thick] (-.25,0) to (-.25 ,.75) to[out=90,in=180] (.8,1.3);
	\draw[very thick] (.25,0) to (.25 ,.75) to[out=90,in=180] (.8,1);
	\draw[very thick] (-.25,2.8) to (-.25 ,2.05) to[out=270,in=180] (.8,1.5);
	\draw[very thick] (.25,2.8) to (.25 ,2.05) to[out=270,in=180] (.8,1.8);
	\node at (.05,1.9) {$\mydots$};
	\node at (.05,.9) {$\mydots$};
	\filldraw[white] (-.4,2.65) rectangle (.35,2.05);
	\draw[very thick] (-.4,2.65) rectangle (.35,2.05);
	\filldraw[white] (-.4,.15) rectangle (.4,.75);
	\draw[very thick] (-.4,.15) rectangle (.4,.75);
	\node at (0,2.35) {$f$};
	\node at (0,.45) {$\hat{f}$};
\end{tikzpicture}
\end{equation}
where the sum ranges over all square-free monic monomials $f$ in the variables $\{x_1,\ldots,x_n\}$ and 
$\hat{f}$ denotes the square-free monic monomial complementary to $f$
(uniquely characterized by $f \cdot \hat{f} = x_1 \cdots x_n$). 

Placing a pair of symmetrizers on the top and bottom of \eqref{eq:idempotent} yields
\[
\begin{tikzpicture}[anchorbase,xscale=-1.2]
	\draw[kirbystyle] (.8,0) to (.8,2);
	\draw[very thick] (-.25,0) to (-.25,2);
	\draw[very thick] (.25,0) to (.25,2);
	\node at (0,.375) {$\mydots$};
	\node at (0,1.625) {$\mydots$};
	\filldraw[white] (-.4,.75) rectangle (.4,1.25); 
	\draw[very thick] (-.4,.75) rectangle (.4,1.25);
	\node at (0,1) {\scriptsize$\JW_n$};
\end{tikzpicture}
\ \ = \ \ 
\sum_{k=0}^n \binom{n}{k}
\begin{tikzpicture}[anchorbase,xscale=-1.2]
	\draw[kirbystyle] (.8,-.5) to (.8,3.3) ;
	\draw[very thick] (-.25,3.3) to (-.25 ,2.05) to[out=270,in=180] (.8,1.5);
	\draw[very thick] (.25,3.3) to (.25 ,2.05) to[out=270,in=180] (.8,1.8);
	\node at (.05,1.9) {$\mydots$};
	\filldraw[white] (-.55,2.55) rectangle (.55,2.05);
	\draw[very thick] (-.55,2.55) rectangle (.55,2.05);
	\filldraw[white] (-.55,2.55) rectangle (.55,3.05);
	\draw[very thick] (-.55,2.55) rectangle (.55,3.05);
	\node at (0,2.3) {\scriptsize$x_{k+1}\mydots x_{n}$};
	\node at (0,2.8) {\scriptsize$\JW_n$};
	\draw[very thick] (-.25,-.5) to (-.25 ,.75) to[out=90,in=180] (.8,1.3);
	\draw[very thick] (.25,-.5) to (.25 ,.75) to[out=90,in=180] (.8,1);
	\filldraw[white] (-.5,.25) rectangle (.5,.75);
	\draw[very thick] (-.5,.25) rectangle (.5,.75);
	\filldraw[white] (-.5,.25) rectangle (.5,-.25);
	\draw[very thick] (-.5,.25) rectangle (.5,-.25);
	\node at (.05,.9) {$\mydots$};
	\node at (0,.5) {\scriptsize$x_1\mydots x_k$};
	\node at (0,0) {\scriptsize$\JW_n$};
\end{tikzpicture}
\ \ = \ \ 
\frac{1}{n!}\sum_{k=0}^n
(-1)^{\binom{k}{2}+\binom{n-k}{2}+k(n-k)} 
\begin{tikzpicture}[anchorbase,xscale=-1.2]
	\draw[kirbystyle] (.8,0) to (.8,2.8) ;
	\draw[very thick] (-.25,0) to (-.25 ,.75) to[out=90,in=180] (.8,1.3);
	\draw[very thick] (.25,0) to (.25 ,.75) to[out=90,in=180] (.8,1);
	\draw[very thick] (-.25,2.8) to (-.25 ,2.05) to[out=270,in=180] (.8,1.5);
	\draw[very thick] (.25,2.8) to (.25 ,2.05) to[out=270,in=180] (.8,1.8);
	\node at (.05,1.9) {$\mydots$};
	\node at (.05,.9) {$\mydots$};
	\filldraw[white] (-.4,2.65) rectangle (.35,2.05);
	\draw[very thick] (-.4,2.65) rectangle (.35,2.05);
	\filldraw[white] (-.4,.15) rectangle (.4,.75);
	\draw[very thick] (-.4,.15) rectangle (.4,.75);
	\node at (0,2.35) {\scriptsize$z_{n}^{n-k}$};
	\node at (0,.45) {\scriptsize${z_{n}^k}$};
\end{tikzpicture} \, .
\]
Indeed, the first equality holds since 
we can simultaneously sort all the dotted strands to the left on the bottom and right on the top. 
All of the signs involved in this operation cancel: 
for every dotted strand slid left on the bottom, 
there is a complementary undotted strand 
which is slid to the left in the top half of the diagram. 
For the second equality, 
we first use that $\JW_n x_{k+1}\cdots x_n \JW_n = (-1)^{k(n-k)} \JW_n x_1\cdots x_{n-k} \JW_n$, 
and then apply \eqref{eq:PxxxP}.
The result then follows from the identity $\binom{k}{2}+\binom{n-k}{2} +k(n-k) = \binom{n}{2}$.
\end{proof}

Theorem \ref{thm:tensor with kirby} allows for the computation of 
$\Hom_{\cdTL}(X,\kirby)$ for any compact object $X \in \cdTL$. 
To begin, we have.

\begin{lemma}
	\label{lemma:hom empty to kirby} 
$\Hom_{\cdTL}(\JW_0, \kirby)\cong \K$.
\end{lemma}
\begin{proof}
First, note that $\Hom_{\cdTL}(\JW_0, \kirby_1)=0$ for parity reasons,
so
\[
\Hom_{\cdTL}(\JW_0,\kirby_0\oplus \kirby_1) \cong \Hom_{\cdTL}(\JW_0,\kirby_0)\oplus \Hom_{\cdTL}(\JW_0,\kirby_1)  
= \Hom_{\cdTL}(\JW_0,\kirby_0) \, .
\]
Using \ref{cor:HomPP}, we then compute
\begin{align*}
\Hom_{\cdTL}(\JW_0,\kirby_0) 
&= \colim \big( \Hom_{\Kar(\dTL)}(\JW_0, \JW_0) \xrightarrow{\Hom(\JW_0,U_0)} 
	\Hom_{\Kar(\dTL)}(\JW_0, q^{-2} \JW_2) \xrightarrow{\Hom(\JW_0,U_2)} \cdots \big) \\
&\cong \colim \big( \K \xrightarrow{\id} \K \xrightarrow{\id} \cdots \big) \cong \K \, . \qedhere
\end{align*}
\end{proof}

\begin{proposition}
	\label{prop:kirby reps pol} 
There is an isomorphism of functors $\cdTL^{\op}\rightarrow \Vect^\Z$:
\begin{equation}\label{eq:representing pol}
\Pol^\ast \cong \Hom_{\cdTL}(-,\kirby),
\end{equation}
where $\Pol^\ast$ is $\Pol$ followed by the contravariant graded dual functor 
$\HOM(-,\K)\colon \Vect^\Z\rightarrow \Vect^\Z$.
\end{proposition}

\begin{proof}
We first restrict the domains of the left and right-hand sides of \eqref{eq:representing pol} to $\dTL$. 
If  $X \in \dTL$, then using duality in $\dTL$ and Theorem \ref{thm:tensor with kirby}, we compute
\begin{align*}
\Hom_{\cdTL}(X,\kirby) &\cong  \Hom_{\cdTL}(\JW_0,\kirby \otimes X^\vee ) \cong  \Hom_{\cdTL}(\gen^0,\Pol(X^\vee) \otimes \kirby)\\
& \cong  \Pol(X^\vee)\otimes  \Hom_{\cdTL}(\gen^0,\kirby) \cong  \Pol(X^\vee) \cong\Pol(X)^\ast \, . 
\end{align*}
This last isomorphism holds as $\Pol$ commutes with duals since it is monoidal.
Completing the domain with respect to direct sums, direct summands, and grading shifts then gives an isomorphism of functors
\begin{equation}\label{eq:representing restricted pol}
\Pol^\ast|_{\Mat(\Kar(\dTL))} \cong\Hom_{\cdTL}(-,\kirby) |_{\Mat(\Kar(\dTL))}.
\end{equation}
Both the left and right-hand sides of \eqref{eq:representing pol} send filtered colimits in $\cdTL$ 
to cofiltered limits in $\Vect^\Z$. Since every object of $\cdTL$ is by definition a filtered colimit in $\Mat(\Kar(\dTL))$,
the isomorphism of restrictions \eqref{eq:representing restricted pol} induces the isomorphism \eqref{eq:representing pol}.
\end{proof}

\begin{remark}
	\label{rem:adjoints}
Proposition \ref{prop:kirby reps pol} implies that the Kirby color $\kirby$ is a 
representing object for the functor $\Pol^\ast \colon \cdTL^{\op} \to \Vect^\Z$, 
i.e.
\[
\Pol^\ast \cong \Hom_{\cdTL^{\op}}(\kirby,-) \, .
\]
 Equivalently, 
using the above isomorphism, one can show
\[
	\Hom_{\Vect^\Z}(\Pol(-),V) \cong \Hom_{\cdTL}(-,V\otimes \kirby)
\] for any object $V\in \Vect^\Z$. This shows that $V\mapsto V \otimes \kirby$ is the right-adjoint to $\Pol$, and $\kirby$ is 
its value on the $1$-dimensional vector space $\K$ concentrated in degree zero.
\end{remark}

Using Proposition \ref{prop:kirby reps pol}, we can compute morphisms between Kirby objects.

\begin{cor}
	\label{cor:endkirby}
$\End_{\cdTL}(\kirby_i)\cong \K[\xs]$ 
where $\xs$ is an element of degree 2.
\end{cor}
\begin{proof}
If $i\not\equiv j$, then $\Hom_{\cdTL}(\kirby_i,\kirby_j)= 0$ for parity reasons.
Thus, we have the isomorphisms 
\[
\End_{\cdTL}(\kirby_i)\cong \Hom_{\cdTL}(\kirby_i,\kirby) \cong \Pol(\kirby_i)^\ast 
\cong \bigoplus_{n=0}^\infty q^{2n} \K
\cong \K[\xs]
\]
of graded vectors spaces.
Here, we have used Proposition \ref{prop:kirbyunknot} for the third isomorphism.
For the algebra structure, 
we note that all of the above isomorphisms commute with the action of the center $Z(\dTL)\cong \K[s,\xs]/(s^2=1)$, 
provided we let $s$ acts by the constant $(-1)^i$ on the right-hand side.
\end{proof}

Finally, we compute the morphisms from a Kirby object to any compact object.

\begin{lemma} \label{lem:morphismfromkirby} 
Let $X \in \Mat(\Kar(\dTL))$, then $\Hom_{\cdTL}(\kirby_i,X) = 0$.
\end{lemma}
\begin{proof} 
It suffices to consider $X=\JW_m$ for some $m$. 
We compute $\Hom_{\cdTL}(\kirby_i,\JW_m) $ as the limit of the inverse system
\[
\Hom_{\cdTL}(q^{-i}\JW_i,\JW_m) \leftarrow
\cdots \leftarrow \Hom_{\cdTL}(q^{-i-2n} \JW_{i+2n},\JW_m) 
\leftarrow 
\Hom_{\cdTL}(q^{-i-2n-2} \JW_{i+2n+2},\JW_m) \leftarrow \cdots
\]
of $\Z$-graded vector spaces.
By Corollary~\ref{cor:HomPP}, each term is given by
\[
\Hom(q^{-i-2n}\symobj_{i+2n}, \symobj_m) \cong q^{i+2n} q^{|i+2n-m|} \K[\xs]/ (\xs)^{1+\min(m,i+2n)} \, .
\]
when $i+2n-m$ is even and is zero otherwise.
In the non-zero case, 
the $n^{th}$ term is supported in degrees $>i+2n$, hence the limit vanishes.
\end{proof}

\begin{remark}
	\label{rem:Verma2}
The vanishing in Lemma~\ref{lem:morphismfromkirby} further supports our 
assertion from Remark \ref{rem:Verma} that $q^k \Pol(\kirby_k)$ should be 
interpreted as the dual Verma module $\nabla(k)$, 
rather than the Verma module $\Delta(k)$.
\end{remark}

\subsection{Kirby colored diagrammatics: completion}
	\label{ss:completion}
	
As discussed at the end of \S \ref{ss:ff1}, the category $\KdTL'$ is not
sufficient to describe all the structure we have observed for the Kirby objects
$\kirby_i$ in $\cdTL$. For example, it is an easy consequence of Theorem
\ref{thm:tensor with kirby} (and Remark~\ref{rem:Indmonoidal}) that $\kirby
\otimes \kirby \cong \coprod_{n=0}^{\infty} q^{-2n} \kirby$ in $\cdTL$, and in
Corollary \ref{cor:kirby squared} below we establish the stronger result that
$\kirby \otimes \kirby \cong \bigoplus_{n=0}^{\infty} q^{-2n} \kirby$. However,
this isomorphism cannot be described using $\KdTL'$. Indeed, this would require
specifying a countably-infinite number of inclusion and projection morphisms
giving the isomorphism, which would then express the identity morphism of
$\kirby' \otimes \kirby'$ as an infinite sum of idempotents projecting onto
various shifts of $\kirby'$. This is impossible in $\KdTL'$ for two reasons: 
\begin{itemize}
	\item The morphisms in $\KdTL'$ are \emph{finite} linear combinations of diagrams.
 	\item None of the relations \eqref{eq:red algebra rels}--\eqref{eq:belly} changes the
 	topology of the \kirbcolword part of the diagram. 
\end{itemize}

To accommodate this inadequacy, 
we will adjoin certain infinite sums of morphisms to $\KdTL'$. 
Informally, the allowed infinite sums are exactly those that get truncated to finite sums 
when precomposing with a morphism out of a finite object, 
and two infinite sums are considered equivalent if all of their finite truncations agree.
We now work towards making this precise.

\begin{definition}
	\label{def:canonical inclusions} 
Let $\Incl$ denote the smallest set of morphisms in $\Kar(\KdTL')$ that is
closed under tensor product and contains $\id_{\JW_n} = \JW_n \in
\End_{\Kar(\KdTL')}(\JW_n)$ and $\iota_n\in
\Hom_{\Kar(\KdTL')}(\JW_n,\kirby'_n)$ for all $n\in \N$. Further, given an
object $T \in \KdTL'$, let the set $\Incl(T)$ of \emph{canonical inclusions}
into $T$ be the set of all morphisms $\iota\in \Incl$ with codomain $T$.
\end{definition}

Before we proceed, we establish some useful facts about canonical inclusions.
First, note that by definition the domain of any canonical inclusion is a finite object in $\Kar(\KdTL')$.
There is a transitive, reflexive relation on $\Incl(T)$ given by declaring that
$\iota \leq \iota'$ provided there exists a morphism $\theta$ 
(necessarily in $\Kar(\dTL)$) such that $\iota = \iota' \circ \theta$.

\begin{lemma}
	\label{lemma:canonical maps}
Given canonical inclusions $\iota \colon X \to T$ and $\iota' \colon X' \to T$
with $\iota \leq \iota'$, then there is a unique morphism $\theta \in
\Hom_{\KdTL'}(X,X')$ with $\iota = \iota' \circ \theta$. Consequently, for all
objects $T$ in $\KdTL'$ the relation $\leq$ defines a partial order on
$\Incl(T)$.
\end{lemma}
\begin{proof}
If $T$ is a finite object, then $\iota = \iota' = \id_T$, 
since the only canonical inclusions into a finite object are identity maps, by definition. 
If $T$ is infinite, 
then the morphism $\theta \in \Hom_{\KdTL'}(X , X')$ in the equation $\iota = \iota' \circ \theta$ 
must be a composition of dotted cup maps between the appropriate tensor product of objects $\JW_n$.
Any two such compositions with fixed domain and codomain are equal, 
by \eqref{eq:twodotswitch}.
\end{proof}

\begin{example} 
We have 
\[
\Incl(\kirby'_i)
= \{\iota_m\in \Hom_{\Kar(\KdTL')}(\JW_m,\kirby'_i) \mid m\in \N,m\equiv i \text{ mod } 2\} 
\cong \N
\]
as partially ordered sets.
Similarly, 
\[
\Incl(\kirby'_i \otimes \kirby'_j) = 
\{\iota_m\otimes \iota_n \in \Hom_{\Kar(\KdTL')}(\JW_m\otimes \JW_n,\kirby'_i \otimes \kirby'_j)
	\mid m,n\in \N, (m,n)\equiv(i,j)  \text{ mod } 2 \} 
\cong \N^2
\]
where we use the product partial order on the latter.
\end{example}

Next, we observe that morphisms with domain a finite object 
admit factorizations as follows.

\begin{lemma}
	\label{lem:factor}
Any morphism in $\Kar(\KdTL')$ with domain a finite object can be expressed 
in the form $\iota \circ D$ where $D$ is a morphism in $\Kar(\dTL)$ and $\iota$ is a canonical inclusion.
\end{lemma}
\begin{proof}
Since identity morphisms in $\dTL$ are canonical inclusions, 
it suffices to show that if $D'$ is a morphism in $\Kar(\KdTL')$ and $\iota'$ is a canonical inclusion, 
then $D' \circ \iota'$ 
can be expressed in
the form $\iota \circ D$ where $D$ is a morphism in $\Kar(\dTL)$ and $\iota$ is a canonical inclusion.
However, this can simply be checked on each generating morphism of $\KdTL'$.
\end{proof}

\begin{definition}
	\label{def:completed diag cat}
Let $\KdTL$ denote the category with the same objects as $\KdTL'$ 
and morphisms defined as follows.   
Given objects $S,T \in \KdTL$, 
let $\Hom_{\KdTL}^k(S,T)$ to be the set of equivalence classes of 
formal expressions $\sum_{i\in I} D_i$, 
wherein $I$ is some (potentially infinite) indexing set,
$D_i\in \Hom_{\KdTL'}^k(S,T)$, 
and such that for each $\iota \in \Incl(S)$ 
we have $D_i \circ \iota = 0$ for all but finitely many $i$. 

Two such infinite sums $\sum_{i\in I} D_i$ and $\sum_{j\in J} E_j \in $ are declared
equivalent provided that for all canonical inclusions $\incl \in \Incl(S)$ 
the finite sums $\sum_{i\in I} (D_i\circ \incl)$ and $\sum_{j\in J} (E_j\circ \incl)$ 
are equal as morphisms in $\Kar(\KdTL')$.
\end{definition}

It is straightforward to check that $\KdTL$ inherits 
composition and tensor product from $\KdTL'$, 
and thus naturally carries the structure of a $\Z$-graded $\K$-linear monoidal category.
For example, composition of morphisms in $\KdTL$ is given formally:
\begin{equation}\label{eq:KdTLcomp}
\left(\sum_{j\in J} E_j\right) \circ \left(\sum_{i\in I} D_i \right) := \sum_{(j,i)\in J \times I} (E_j \circ D_i).
\end{equation}
That this is again a morphism in $\KdTL$ follows from Lemma \ref{lem:factor}.
Equation \eqref{eq:KdTLcomp}
further implies that the infinite sums $\sum_{i\in I}D_i$ satisfy the
appropriate notion of distributivity with respect to finite sums in $\KdTL'$.

Our main goal in this section is to establish the following result.

\begin{thm}
	\label{thm:diagpres} 
The functor $\phi\colon \KdTL' \to \cdTL$ 
from Theorem \ref{thm:from diag kirby to kirby} extends to a fully faithful functor 
$\KdTL \to \cdTL$ that we also denote by $\phi$:
\[
\begin{tikzcd}
	\KdTL \arrow[dashed]{dr}
		&  \\
		\KdTL' \arrow[hookrightarrow]{u}\arrow{r}& \cdTL
	\end{tikzcd}
	\]
	\end{thm}
We will first check the existence of the extension. The proof of fullness and
faithfulness requires additional preparation.

\begin{proof}[Proof (existence of extension).]
If $X$ is a finite object, then we have 
$\Hom_{\KdTL}(X,T) = \Hom_{\KdTL'}(X,T)$ 
since the only canonical inclusions into finite objects are identity morphisms.
This determines the extension on such $\Hom$-spaces. 

If $f \in \Hom_{\KdTL}(S,T)$ for $S$ infinite, 
then we obtain a family of morphisms $f\circ \iota \in \Hom_{\KdTL'}(S_\iota, T)$ 
indexed by $\iota \in \mathrm{Incl}(S)$. 
Note that each finite object $S_\iota$ is obtained from $S$ by simply replacing each 
instance of $\kirby_i$ with some $\JW_n$ for $n \equiv i \text{ mod } 2$.
As such, we have a directed system of dotted cup morphisms relating the objects $S_\iota$
(the unique morphisms given in Lemma \ref{lemma:canonical maps}).
Precomposition with these morphisms makes the collection 
$\{\Hom_{\KdTL'}(S_\iota, T) \mid \iota \in \mathrm{Incl}(S)\}$ into an inverse system, 
and the limit of this system is $\Hom_{\cdTL}(S,T)$.
The collection of maps $\{f\circ \iota\}_{\iota \in \Incl(S)}$ is stable by 
\eqref{eq:red relation c}, 
hence it gives an element of the limit.
The extension is thus given by $f \mapsto \{f\circ \iota\}_{\iota \in \Incl(S)} \in \Hom_{\cdTL}(S,T)$.
It is immediate from Definition \ref{def:completed diag cat} 
that this assignment is independent of the representative of equivalence class.
\end{proof}

In order to establish fully faithfulness in Theorem \ref{thm:diagpres}, 
we first prove some relations in the completed category $\KdTL$
which will facilitate the proof.

\begin{subequations}
\begin{lemma}
	\label{lem:rewiring}
The following relations hold in $\KdTL$:
\begin{equation}
\label{eq:red rewire}
\begin{tikzpicture}[anchorbase,scale=1]
\draw[kirbystyle]  (0,0)node[below=-2pt]{\scriptsize $[i]$} to (0,1)node[above=-2pt]{\scriptsize $[i]$};
\draw[kirbystyle]  (1,0)node[below=-2pt]{\scriptsize $[k]$} to (1,1)node[above=-2pt]{\scriptsize $[j{+}k]$};
\draw[kirbystyle]  (.5,0)node[below=-2pt]{\scriptsize $[j]$} to (1,.5);
\end{tikzpicture}
= \
\sum_{n\geq 0}
\frac{(-1)^{\genfrac(){0pt}{2}{n}{2}}}{n!}
\begin{tikzpicture}[anchorbase,scale=1]
\draw[kirbystyle]  (0,0)node[below=-2pt]{\scriptsize $[i]$}  to (0,1)node[above=-2pt]{\scriptsize $[i]$};
\draw[kirbystyle]  (1,0)node[below=-2pt]{\scriptsize $[k]$}  to (1,1)node[above=-2pt]{\scriptsize $[j{+}k]$};
\draw[kirbystyle]  (.5,0)node[below=-2pt]{\scriptsize $[j]$}  to (0,.5);
\draw[kirbystyle] (0,.75) to (1,.75);
\draw[fill=white] (.5,.75) circle (.15);
\node at (.5,1) {\scriptsize $n$};
\node at (.25,.25) {\small $\bullet$};
\node at (.4,.35) {\scriptsize $n$};
\end{tikzpicture}
\, , \quad
\begin{tikzpicture}[anchorbase,xscale=-1]
\draw[kirbystyle]  (0,0)node[below=-2pt]{\scriptsize $[k]$} to (0,1)node[above=-2pt]{\scriptsize $[k]$};
\draw[kirbystyle]  (1,0)node[below=-2pt]{\scriptsize $[i]$} to (1,1)node[above=-2pt]{\scriptsize $[i{+}j]$};
\draw[kirbystyle]  (.5,0)node[below=-2pt]{\scriptsize $[j]$} to (1,.5);
\end{tikzpicture}
= \
\sum_{n\geq 0}
\frac{(-1)^{\genfrac(){0pt}{2}{n}{2}+n(n-j)}}{n!}
\begin{tikzpicture}[anchorbase,xscale=-1]
\draw[kirbystyle]  (0,0)node[below=-2pt]{\scriptsize $[k]$}  to (0,1)node[above=-2pt]{\scriptsize $[k]$};
\draw[kirbystyle]  (1,0)node[below=-2pt]{\scriptsize $[i]$}  to (1,1)node[above=-2pt]{\scriptsize $[i{+}j]$};
\draw[kirbystyle]  (.5,0)node[below=-2pt]{\scriptsize $[j]$}  to (0,.5);
\draw[kirbystyle] (0,.75) to (1,.75);
\draw[fill=white] (.5,.75) circle (.15);
\node at (.5,1) {\scriptsize $n$};
\node at (.25,.25) {\small $\bullet$};
\node at (.4,.35) {\scriptsize $n$};
\end{tikzpicture} \, .
\end{equation}
\end{lemma}
\begin{proof}
Using Proposition \ref{prop:Pn times kirby decomp} 
and equation \eqref{eq:z on rungs}, we compute:
\[
\begin{tikzpicture}[anchorbase]
	\draw[kirbystyle] (-.75,0) to (-.75,2);
	\draw[kirbystyle] (.75,0) to (.75,2);
	\draw[very thick] (-.25,0) to (-.25,1) to [out=90,in=180] (.75,1.5);
	\draw[very thick] (.25,0) to (.25,1) to [out=90,in=180] (.75,1.25);
	\filldraw[white] (-.5,.75) rectangle (.5,.25); 
	\draw[very thick] (-.5,.75) rectangle (.5,.25);
	\node at (0,.5) {\scriptsize$\JW_m$};
	\node at (0,1) {$\mydots$};
	\node at (0,.125) {$\mydots$};
\end{tikzpicture}
\ = \
\frac{(-1)^{\binom{m}{2}}}{m!}
\sum_{n=0}^m \,
\begin{tikzpicture}[anchorbase,xscale=-1]
	\draw[kirbystyle] (-.75,0) to (-.75,2.25);
	\draw[kirbystyle] (.75,0) to (.75,2.25);
	\draw[very thick] (-.75,2) to (.75,2);
	\draw[very thick] (-.75,1.75) to (.75,1.75);	
	\draw[very thick] (-.25,0) to (-.25,1) to [out=90,in=180] (.75,1.5);
	\draw[very thick] (.25,0) to (.25,1) to [out=90,in=180] (.75,1.25);
	\filldraw[white] (-.5,.625) rectangle (.5,.25); 
	\draw[very thick] (-.5,.625) rectangle (.5,.25);
	\node at (0,.43) {\scriptsize$\JW_m$};
	\filldraw[white] (-.5,.625) rectangle (.5,1); 
	\draw[very thick] (-.5,.625) rectangle (.5,1);
	\node at (0,.81) {\scriptsize$z_m^n$};
	\node at (0,1.125) {$\mydots$};
	\node at (0,.125) {$\mydots$};
	\filldraw[white] (-.3,1.625) rectangle (.3,2.125); 
	\draw[very thick] (-.3,1.625) rectangle (.3,2.125); 
		\node[rotate=-90] at (0,1.875) {\scriptsize $z_m^l$};
\end{tikzpicture}
\ = \
\sum_{n=0}^m \frac{(-1)^{\binom{n}{2}}}{n!} \,
\begin{tikzpicture}[anchorbase,xscale=-1]
	\draw[kirbystyle] (-.75,0) to (-.75,2.25);
	\draw[kirbystyle] (.75,0) to (.75,2.25);
	\draw[very thick] (-.25,0) to (-.25,1) to [out=90,in=180] (.75,1.5);
	\draw[very thick] (.25,0) to (.25,1) to [out=90,in=180] (.75,1.25);
	\filldraw[white] (-.5,.75) rectangle (.5,.25); 
	\draw[very thick] (-.5,.75) rectangle (.5,.25);
	\node at (0,.5) {\scriptsize$\xs_m^n$};
	\node at (0,1) {$\mydots$};
	\node at (0,.125) {$\mydots$};
	\draw[kirbystyle] (-.75,1.875) to (.75,1.875);
	\draw[fill=white] (0,1.875) circle (.15);
	\node at (0,2.2) {\scriptsize $n$};
\end{tikzpicture} \ .
\]
Thus, the left- and right-hand sides of the first equation in \eqref{eq:red
rewire} are equal in $\Kar(\KdTL')$ upon precomposing with all 
$\iota \in \Incl(\kirby'_i \otimes \kirby'_j \otimes \kirby'_k)$. 
Hence, they are equal in $\KdTL$.
The second equation in \eqref{eq:red rewire} follows from the first by application of 
the left-right reflection symmetry from Remark \ref{rem:involution}, 
noting that each of the dots on the \kirbcolword strand gives a sign $(-1)^{j-1}$. 
\end{proof}

Precomposing \eqref{eq:red rewire} with the appropriate \kirbcolword unit maps
gives the following relations.

\begin{cor}
The following identities hold in $\KdTL$:
\begin{equation}\label{eq:decomp of identity}
\begin{tikzpicture}[anchorbase,xscale=.75]
\draw[kirbystyle]  (0,0)node[below=-2pt]{\scriptsize $[i]$} to (0,1)node[above=-2pt]{\scriptsize $[i]$};
\draw[kirbystyle]  (1,0)node[below=-2pt]{\scriptsize $[j]$} to (1,1)node[above=-2pt]{\scriptsize $[j]$};
\end{tikzpicture}
= \
\sum_{n\geq 0}
\frac{(-1)^{\genfrac(){0pt}{2}{n}{2}}}{n!}
\begin{tikzpicture}[anchorbase,scale=1]
	\draw[kirbystyle]  (0,0)node[below=-2pt]{\scriptsize $[i]$}  to (0,1.25)node[above=-2pt]{\scriptsize $[i]$};
	\draw[kirbystyle]  (1,0)node[below=-2pt]{\scriptsize $[j]$} to (0,.5);
	\draw[kirbystyle] (0,1) to[out=-90,in=180] (.5,.75) to[out=0,in=-90] (1,1)
		to (1,1.25)node[above=-2pt]{\scriptsize $[j]$};
	\draw[fill=white] (.5,.75) circle (.15);
	\node at (.5,1) {\scriptsize $n$};
	\node at (.5,.25) {\small $\bullet$};
	\node at (.65,.4) {\scriptsize $n$};
\end{tikzpicture}
\ , \qquad
\begin{tikzpicture}[anchorbase,xscale=-.75]
\draw[kirbystyle]  (0,0)node[below=-2pt]{\scriptsize $[i]$} to (0,1)node[above=-2pt]{\scriptsize $[i]$};
\draw[kirbystyle]  (1,0)node[below=-2pt]{\scriptsize $[j]$} to (1,1)node[above=-2pt]{\scriptsize $[j]$};
\end{tikzpicture}
= \
\sum_{n\geq 0}
\frac{(-1)^{\genfrac(){0pt}{2}{n}{2}+n(j-1)}}{n!}
\begin{tikzpicture}[anchorbase,xscale=-1]
	\draw[kirbystyle]  (0,0)node[below=-2pt]{\scriptsize $[i]$}  to (0,1.25)node[above=-2pt]{\scriptsize $[i]$};
	\draw[kirbystyle]  (1,0)node[below=-2pt]{\scriptsize $[j]$} to (0,.5);
	\draw[kirbystyle] (0,1) to[out=-90,in=180] (.5,.75) to[out=0,in=-90] (1,1) 
		to (1,1.25)node[above=-2pt]{\scriptsize $[j]$};
	\draw[fill=white] (.5,.75) circle (.15);
	\node at (.5,1) {\scriptsize $n$};
	\node at (.5,.25) {\small $\bullet$};
	\node at (.65,.4) {\scriptsize $n$};
\end{tikzpicture}
\end{equation}
\begin{equation}
\label{eq:infrel}
\begin{tikzpicture}[anchorbase,xscale=.75]
\draw[kirbystyle]  (0,.5)node{$\bullet$} to (0,1)node[above=-2pt]{\scriptsize $[0]$};
\draw[kirbystyle]  (1,0)node[below=-2pt]{\scriptsize $[j]$} to (1,1)node[above=-2pt]{\scriptsize $[j]$};
\end{tikzpicture}
= \
\sum_{n\geq 0}
\frac{(-1)^{\genfrac(){0pt}{2}{n}{2}}}{n!}
\begin{tikzpicture}[anchorbase,xscale=1]
	\draw[kirbystyle]  (0,.25)node[below=-2pt]{\scriptsize $[j]$}  to (0,1.25)node[above=-2pt]{\scriptsize $[0]$};
	\draw[kirbystyle] (0,1) to[out=-90,in=180] (.5,.75) to[out=0,in=-90] (1,1) 
		to (1,1.25)node[above=-2pt]{\scriptsize $[j]$};
	\draw[fill=white] (.5,.75) circle (.15);
	\node at (.5,1) {\scriptsize $n$};
	\node at (0,.5) {\small $\bullet$};
	\node at (.2,.5) {\scriptsize $n$};
\end{tikzpicture}
\ , \qquad
\begin{tikzpicture}[anchorbase,xscale=-.75]
\draw[kirbystyle]  (0,.5)node{$\bullet$} to (0,1)node[above=-2pt]{\scriptsize $[0]$};
\draw[kirbystyle]  (1,0)node[below=-2pt]{\scriptsize $[j]$} to (1,1)node[above=-2pt]{\scriptsize $[j]$};
\end{tikzpicture}
= \
\sum_{n\geq 0}
\frac{(-1)^{\genfrac(){0pt}{2}{n}{2}+n(j-1)}}{n!}
\begin{tikzpicture}[anchorbase,xscale=-1]
	\draw[kirbystyle]  (0,.25)node[below=-2pt]{\scriptsize $[j]$}  to (0,1.25)node[above=-2pt]{\scriptsize $[0]$};
	\draw[kirbystyle] (0,1) to[out=-90,in=180] (.5,.75) to[out=0,in=-90] (1,1) 
		to (1,1.25)node[above=-2pt]{\scriptsize $[j]$};
	\draw[fill=white] (.5,.75) circle (.15);
	\node at (.5,1) {\scriptsize $n$};
	\node at (0,.5) {\small $\bullet$};
	\node at (.2,.5) {\scriptsize $n$};
\end{tikzpicture}
\end{equation}
\qed
\end{cor}
\end{subequations}

\begin{cor}\label{cor:kirby squared}
We have that 
$\kirby_i\otimes \kirby_j \cong \bigoplus_{k\geq 0}q^{-2k}\kirby_{i+j}$ in $\cdTL$.
\end{cor}
\begin{proof}
Relation \eqref{eq:dotted teardrop} and the associativity relation in \eqref{eq:red algebra rels}
imply that
\begin{equation}
	\label{eq:orthogonality}
\begin{tikzpicture}[anchorbase,scale=.75]
	\draw[kirbystyle]  (0,0)node[below=-2pt]{\scriptsize $[j]$}  to (0,2)node[above=-2pt]{\scriptsize $[j]$};
	\draw[kirbystyle] (0,.75) to[out=-90,in=180] (.5,.5) to[out=0,in=-90] (1,.85) 
		to (1,.85) to[out=90,in=-30] (0,1.5);
	\node at (.65,1.5) {\scriptsize $m$};
	\node at (.5,1.31) {\small $\bullet$};
	\draw[fill=white] (.5,.5) circle (.15);
	\node at (.5,.75) {\scriptsize $n$};
\end{tikzpicture}
\ = \d_{n,m} (-1)^{\genfrac(){0pt}{2}{n}{2}} n! 
\begin{tikzpicture}[anchorbase,scale=.75]
	\draw[kirbystyle]  (0,0)node[below=-2pt]{\scriptsize $[j]$}  to (0,2)node[above=-2pt]{\scriptsize $[j]$};
\end{tikzpicture}
\, .
\end{equation}
Thus, equation \eqref{eq:decomp of identity} expresses 
$\id_{\kirby'_i}\otimes \id_{\kirby'_j}$ as a sum of orthogonal idempotents.
The result then follows by applying the functor $\KdTL\rightarrow \cdTL$ to this decomposition of identity 
and then using the Biproduct Recognition Lemma \ref{lemma:biproduct recognition}.
Note that condition (ii) therein holds by Proposition \ref{prop:kirby reps pol} and 
condition (iii) holds by Definition \ref{def:completed diag cat}.
\end{proof}

The analogous result is true in the purely diagrammatic category $\KdTL$,
but, since we have not yet shown that $\KdTL$ is compactly generated,
we cannot directly use the Biproduct Recognition Lemma \ref{lemma:biproduct recognition}.
Nonetheless, we can repeat the argument in the proof thereof, using the following technical fact.

\begin{lemma}\label{lemma:homs bdd above}
Let $X$ be a finite object in $\KdTL$ and let $T \in \KdTL$ be arbitrary. 
Then, $\Hom_{\KdTL}(X,T)$ is bounded above in $q$-degree.
\end{lemma}
\begin{proof}
The object $X$ has a dual, 
so we have $\Hom_{\KdTL}(X,T) \cong \Hom_{\KdTL}(\gen^0,X^\vee\otimes T)$.
Thus, without loss of generality we may assume $X=\gen^0$. 
Further, using Theorem \ref{thm:tensor with kirby} we may assume that $T$ 
is a tensor product of Kirby objects $T = \kirby'_{i_1}\otimes \cdots \otimes \kirby'_{i_r}$.
By Lemma \ref{lem:factor}, 
$\Hom_{\KdTL'}(\gen^0, \kirby'_{i_1}\otimes \cdots \otimes \kirby'_{i_r})$ 
is spanned by morphisms of form $\iota \circ D$ where 
$D \in \Hom_{\KdTL'}(\gen^0,\gen^{2N})$ 
and $\iota =\iota_{n_1}\otimes\cdots \otimes \iota_{n_r}$ with $n_1+\cdots + n_r = 2N$. 
Since $\iota$ therefore has degree $-2N$ and $D$ has degree at most $2N$, 
$\Hom_{\KdTL}(\gen^0,T)$ is supported in non-positive degrees.
\end{proof}

\begin{prop}
	\label{prop:kirby squared diag}  
We have that $\kirby'_i \otimes \kirby'_j \cong \bigoplus_{k\geq 0}q^{-2k}\kirby'_{i+j}$ in $\KdTL$.
\end{prop}
\begin{proof}
The proof parallels that of the Biproduct Recognition Lemma \ref{lemma:biproduct recognition}. 
Let $Y := \kirby'_i \otimes \kirby'_j$ and $Y_n = q^{-i-j-2n} \kirby'_{i+j}$ and consider the 
projection and inclusion morphisms:
\[
\sigma_n:= 
\begin{tikzpicture}[anchorbase,scale=1]
	\draw[kirbystyle]  (0,.25)node[below=-2pt]{\scriptsize $[i{+}j]$}  to (0,1.25)node[above=-2pt]{\scriptsize $[i]$};
	\draw[kirbystyle] (0,1) to[out=-90,in=180] (.5,.75) to[out=0,in=-90] (1,1)
		to (1,1.25)node[above=-2pt]{\scriptsize $[j]$};
	\draw[fill=white] (.5,.75) circle (.15);
	\node at (.5,1) {\scriptsize $n$};
\end{tikzpicture} \colon Y_n \to Y
\, , \qquad
\pi_n := 
\frac{(-1)^{\genfrac(){0pt}{2}{n}{2}}}{n!}
\begin{tikzpicture}[anchorbase,scale=1]
	\draw[kirbystyle]  (0,0)node[below=-2pt]{\scriptsize $[i]$} to (0,1)node[above=-2pt]{\scriptsize $[i{+}j]$};
	\draw[kirbystyle]  (1,0)node[below=-2pt]{\scriptsize $[j]$} to (0,.5);
	\node at (.5,.25) {\small $\bullet$};
	\node at (.65,.4) {\scriptsize $n$};
\end{tikzpicture} \colon Y\to Y_n \, .
\]
We have 
\begin{equation}
	\label{eq:sigmapi}
\pi_n\circ \sigma_m = \d_{n,m} \id_{\kirby'_{i+j}} \, , \quad
\id_{\kirby'_i}\otimes \id_{\kirby'_j} = \sum_{n=0}^\infty \sigma_n\circ \pi_n
\end{equation}
by \eqref{eq:orthogonality} and \eqref{eq:decomp of identity}, respectively.

We first claim that the collection of maps $\{\pi_n\}_{n \in \N}$ exhibits $Y$ as a product $Y\cong \prod_{n \in \N} Y_n$. 
For this, we must construct a two-sided inverse to the assignment:
\[
\Phi \colon \Hom_{\KdTL}(T,Y) \to \prod_{n \in \N} \Hom_{\KdTL}(T,Y_n) 
\, , \quad 
F \mapsto (\pi_n\circ F)_{n \in \N} \, .
\]
where $T \in \KdTL$ is an arbitrary object. 
Consider the map
\[
\Phi' \colon \prod_{n \in \N} \Hom_{\KdTL}(T,Y_n) \to \Hom_{\KdTL}(T,Y) 
\, , \quad 
(f_n)_{n \in \N} \mapsto   \sum_{n=0}^\infty \sigma_n\circ f_n \, .
\]
To see that this is well-defined, 
it suffices to consider the case when each component $f_n\in \Hom(T,Y_n)$ is homogeneous of degree zero.
Since $Y_n = q^{-i-j-2n} \kirby'_{i+j}$, 
this means that $f_n$ corresponds to a degree $i+j+2n$ element of $\Hom(T,\kirby'_{i+j})$.
Fix a canonical inclusion $\iota \in \Incl(T)$ with domain $T_\iota$, 
then in $\Hom_{\KdTL'}(T_\iota, \kirby'_{i+j})$ we have that
\[
\deg(f_n \circ \iota) = i+j+2n+\deg(\iota) \, .
\]
Lemma \ref{lemma:homs bdd above} then implies that 
$f_n\circ \iota=0$ for all but finitely many $n$.
It is straightforward to check that $\Phi$ and $\Phi'$ are two-sided inverses.

Next, we must show that the collection of maps $\{\sigma_n\}_{n \in \N}$ exhibits $Y$ 
as a coproduct $Y\cong \coprod_{n \in \N} Y_n$, 
i.e.~that, for all objects $T \in \KdTL$, the assignment
\[
\Psi \colon \Hom_{\KdTL}(Y,T) \to \prod_{n \in \N} \Hom_{\KdTL}(Y_n,T)
\, , \quad 
G \mapsto \{G\circ \sigma_n\}_n
\]
has a $2$-sided inverse.
Consider
\[
\Psi' \colon \prod_{n \in \N} \Hom_{\KdTL}(Y_n,T) \to  \Hom_{\KdTL}(Y,T) 
\, , \quad \{g_n\}_n \mapsto  \sum_{n=0}^\infty g_n\circ \pi_n \, .
\]
Given any canonical inclusion $\iota \in \incl(\kirby'_{i})$, 
the first relation in Lemma \ref{lem:KarRels} implies that $\pi_n \circ \iota = 0$
for all but finitely many $n$. 
Hence, this assignment is well-defined,
and we leave it as an
exercise to show that $\Psi$ and $\Psi'$ are two-sided inverses.

Lastly, the condition (i) in Definition \ref{def:biproducts} holds by \eqref{eq:sigmapi}.
\end{proof}

At last, we establish the following.

\begin{proof}[Proof (fully faithfullness in Theorem \ref{thm:diagpres}).]
	Given two objects $S,T$ in $\KdTL$, we need to show that 
	\[
		\phi \colon \Hom_{\KdTL}(S,T) \to \Hom_{\cdTL}(\phi(S),\phi(T))
	\]
	is an isomorphism.   
	If $T$ is finite, then these $\Hom$-spaces are zero unless $S$ is finite as well.
	In that case, both $\Hom$-spaces are simply the corresponding $\Hom$-space 
	in $\Kar(\dTL)$ so there is nothing to show.
	
	Thus, suppose that $T$ is infinite.
	Using (the proof of) Theorem \ref{thm:tensor with kirby}, 
	Corollary \ref{cor:kirby squared} and Proposition \ref{prop:kirby squared diag}, 
	we immediately reduce to the case where $T = \kirby'_i$.
	We now compute all of the relevant $\Hom$-spaces case-by-case.
	
	\begin{enumerate}
	\item If $S=X$ is finite, then without loss of generality we may assume $X=\gen^n$
	and $i\equiv n \text{ mod } 2$, since otherwise 
	$\Hom_{\KdTL}(X,Y)=0$ and $\Hom_{\overline{dTL}}(\phi(X),\phi(Y))=0$. 
Proposition \ref{prop:kirby reps pol} gives that
\[
\Hom_{\cdTL}(\phi(X), \phi(T)) = \Hom_{\cdTL}(\gen^n, \kirby_i) \cong \Pol^\ast(\gen^n)
\]
while Definition \ref{def:completed diag cat} and Lemma \ref{lem:factor}
imply that
\[
\Hom_{\KdTL}(\gen^n,\kirby'_i) = \Hom_{\KdTL'}(\gen^n,\kirby'_i)
= 
\spann \left\{
\begin{tikzpicture}[anchorbase,xscale=-1]\draw[very thick] (.35,-.5) to (.35,0)  to [out=90,in=-45] (0,.5);
	\draw[very thick] (-.35,-.5) to (-.35,0)  to [out=90,in=-135] (0,.5);
	\draw[kirbystyle] (0,0.5) to (0,1) node[above=-2pt]{\scriptsize $[n]$};
	\node at (0,.15) {$\cdots$};
	\filldraw[white] (-.5,-.3) rectangle (.5,0); 
	\draw[very thick] (-.5,-.3) rectangle (.5,0); 
	\node at (0,-.15) {\scriptsize$f$};
\end{tikzpicture}
\right\}
\]
where $f$ ranges over all square free monic monomials in $\K[x_1,\ldots,x_n]$.
The isomorphism from Theorem \ref{thm:tensor with kirby}, 
in the explicit form provided by \eqref{eq:idempotent}, 
identifies these $\Hom$-spaces.

	\item If $S$ is infinite, then as above we immediately reduce to the case of 
	$S=\kirby'_j$. Without loss of generality we may assume $i\equiv j \text{ mod } 2$, 
	since otherwise both $\Hom$-spaces are zero.
	Definition \ref{def:completed diag cat} gives that
\[
\Hom_{\KdTL}(\kirby'_j,\kirby'_i) = 
\Hom_{\KdTL'}(\kirby'_j,\kirby'_i)  = 
\spann\left\{
\begin{tikzpicture}[anchorbase,xscale=1]
	\draw[kirbystyle] (0,0) node[below=-2pt]{\scs$[j]$} to node[black]{$\bullet$} (0,1) node[above=-2pt]{\scs$[i]$};
	\node at (.25,.5) {\scriptsize $n$};
\end{tikzpicture}
\right\}_{n\in \N}
\]
and the result follows from Corollary \ref{cor:endkirby}.
	\end{enumerate}
\end{proof}

\section{Kirby-colored Khovanov homology}
	\label{sec:KirbyKh}

In this section, we review cabling properties of Khovanov homology and use them
to define the Kirby-colored Khovanov homology of a link. Using results of
Manolescu--Neithalath \cite{2020arXiv200908520M}, we show that Kirby-colored
Khovanov homology recovers the skein lasagna modules from
\cite{2019arXiv190712194M} evaluated on $4$-dimensional $2$-handlebodies, and thus gives a
$4$-manifold invariant.

\subsection{Cabling and Khovanov homology}
\label{ss:cabling}

As in the rest of the paper, we work over a field $\K$ of characteristic zero.
The following construction depends on a fully functorial version 
of Khovanov homology. 
While we prefer Blanchet's oriented model \cite{Bla} because of its
compatibility with the $\glN$ version developed in \cite{QR, ETW}, 
the precise fix to functoriality will not play a role here.
See e.g.~the introduction of \cite{2019arXiv190312194B} for a discussion of 
several of the means of fixing the functoriality of Khovanov homology.

\begin{thm}[Cabling in Khovanov homology]
	\label{thm:cable functor}
Let $\mathcal{L} = \mathcal{K}_1\cup \cdots \cup \mathcal{K}_r$ be an $r$-component 
framed oriented link in $S^3$.  There is a functor
\[
\Kh_{\mathcal{L}} \colon \ABN^{\times r} \rightarrow \Vect^{\Z\times \Z}
\]
sending $(\gen^{m_1},\cdots, \gen^{m_r})$ to the Khovanov homology of the cable
$\mathcal{K}_1^{m_1} \cup \cdots \cup \mathcal{K}_r^{m_r}$
of $\mathcal{L}$ obtained by replacing the component $\mathcal{K}_j$ by its
$m_j$-fold parallel cable, 
with parallelism determined by the framing of $\mathcal{K}_j$, 
and with alternating orientations\footnote{If $m_j$ is even then the orientations are balanced; 
if $m_j$ is odd there is one more component oriented parallel to $\mathcal{K}_j$ than antiparallel.}.
\end{thm}

\begin{proof}
Let $\CS_{\mathcal{A}}$ denote the category whose objects are 
oriented, embedded $1$-manifolds in $\mathcal{A} := S^1 \times [0,1]$
and whose morphisms are formal $\K$-linear combinations 
of oriented link cobordisms properly embedded in $\mathcal{A} \times [0,1]$.
Similarly, let $\CS_{S^3}$ denote the category whose objects are oriented links in $S^3$ 
and whose morphisms are formal $\K$-linear combinations 
of link cobordisms in $S^3\times [0,1]$. 
A framed link $\mathcal{L}$ determines an embedding $\mathcal{A}^{\sqcup r}\rightarrow S^3$, 
which gives a functor 
$\Phi_{\mathcal{L}} \colon \CS_{\mathcal{A}}^{\times r} \rightarrow \CS_{S^3}$.

Meanwhile, Khovanov homology 
(with a chosen fix of its functoriality) gives a functor 
$\Kh\colon \CS_{S^3} \rightarrow \Vect^{\Z\times \Z}$.
Since $\Vect^{\Z\times \Z}$ is closed under direct sums and grading shifts, 
the composite $\Kh \circ \Phi_{\mathcal{L}} \colon \CS_{\mathcal{A}}^{\times r} \to \Vect^{\Z\times \Z}$ 
extends to $\Mat(\CS_{\mathcal{A}})^{\times r}$. 
By definition, $\ABN$ is a quotient of $\Mat(\CS_{\mathcal{A}})$ and the functor 
$\Kh \circ \Phi_{\mathcal{L}}$ descends to this quotient, 
since the local relations that define $\ABN$ are satisfied by the Khovanov invariant.  
The resulting functor $\ABN^{\times r}\rightarrow \Vect^{\Z\times \Z}$ is $\Kh_{\mathcal{L}}$.
\end{proof}

\begin{remark}
We work with alternating orientations in the components $\mathcal{K}^n$ so that
the annulus cobordisms $\mathcal{K}^n \leftrightarrow \mathcal{K}^{n+2}$ 
(the images under $\Phi_{\mathcal{L}}$ of cap and cup morphisms in $\dTL \subset \ABN$) are oriented.  
\end{remark}

We will also use $\Kh_{\mathcal{L}}$ to denote the corresponding functors 
$\dTL^{\times r} \to \Vect^{\Z\times \Z}$ and $\cdTL^{\times r} \to \Vect^{\Z\times \Z}$. 
Note that the latter exists since $\Vect^{\Z\times \Z}$ is further closed under 
direct summands and (filtered) colimits.

\begin{example}
	\label{ex:eval is pol}
If $\mathcal{U}$ is the 0-framed unknot, then $\Kh_{\mathcal{U}}$ coincides with the polynomial representation 
(after forgetting about the homological grading, which is trivial on the image of $\Kh_{\mathcal{U}}$).
On the level of objects, $\Kh_{\mathcal{U}}(\gen^n)$ is the Khovanov homology of an $n$-component unlink, 
which is isomorphic to $\Pol(n)$ (with graded dimension $(q+q^{-1})^n$). 
The verification that cups, caps, and dots act in the same way under $\Kh_{\mathcal{U}}$ and $\Pol$ 
is a straightforward exercise.
\end{example}

For the purposes of this paper, we take the following definition of a colored link.
(Recall Theorem \hyperlink{introthm:cabling}{B} and Definition \ref{def:KKh} from the introduction.)

\begin{definition}
	\label{def:colored homology}
A \emph{colored link} is a framed oriented link $\mathcal{L}\subset S^3$ 
with an ordering of its components, $\mathcal{L}=\mathcal{K}_1\cup\cdots \cup \mathcal{K}_r$ 
together with a choice of objects $X_1,\ldots,X_r\in \cdTL$. 
We write the data of a colored link as 
$\mathcal{L}^{\underline{X}} := \mathcal{K}_1^{X_1}\cup \cdots\cup \mathcal{K}_r^{X_r}$, 
where $\underline{X}:=(X_1,\ldots,X_r)$. 
Lastly, the \emph{Khovanov homology of a colored link} 
is denoted $\Kh(\mathcal{L}^{\underline{X}}):=\Kh_{\mathcal{L}}(X_1,\ldots,X_r)$.
\end{definition}

Further, we abbreviate 
$\mathcal{K}_1^{\gen^{m_1}} \cup \cdots\cup \mathcal{K}_r^{\gen^{m_r}}$ 
by $\mathcal{K}_1^{m_1} \cup \cdots \cup \mathcal{K}_r^{m_r}$, 
which is compatible with the usage of the latter in Theorem \ref{thm:cable functor} 
under the functor $\Phi_{\mathcal{L}}$ appearing in its proof.

The $n$-strand braid group $\Br_n$ acts on $\Kh(\mathcal{K}^n)$ 
by cobordisms that braid parallel components of the cable. 
Strictly speaking, the action of a braid also permutes the chosen orientations 
of the components of $\mathcal{K}^n$; 
however, these orientations only affect the 
Khovanov invariant up to an overall shift, so there is no harm in regarding braids as 
acting by automorphisms on $\Kh(\mathcal{K}^n)$. 
Work of Grigsby-Licata-Wehrli \cite{GLW} shows that this
braid group action factors through the symmetric group $\mathfrak{S}_n$, 
and, moreover, that this action coincides with the one induced by the 
$\mathfrak{S}_n$ action on $\gen^n \in \dTL$ from \eqref{eqn:TLbraiding}.

Before discussing Kirby-colored Khovanov homology, 
we take the opportunity to prove a
folk theorem that identifies the invariant subspace $\Kh(\mathcal{K}^n)^{\mathfrak{S}_n}$ 
under this action with both our invariant $\Kh(\mathcal{K}^{\JW_n})$ and another invariant appearing in the 
literature: Khovanov's categorification of the $n$-colored Jones polynomial from \cite{Kho4}.
We denote this latter invariant by $\Kh(\mathcal{K};n)$.

\begin{proposition}
	\label{prop:sym colored} 
If $\mathcal{K}$ is a framed knot and $n \in \N$, 
then $\Kh(\mathcal{K}^n)^{\mathfrak{S}_n} \cong \Kh(\mathcal{K}^{\JW_n}) \cong \Kh(\mathcal{K};n)$.
\end{proposition}

\begin{proof}
Recall that $V$ denotes the vector representation of $\slnn{2}$.
One of the central ideas in \cite{Kho4} is the construction of 
an explicit complex of $\slnn{2}$-representations which gives a 
resolution (in the category of finite-dimensional $\slnn{2}$ representations) 
of $\Sym^n(V)$ by tensor powers of $V$.
Equivalently, Khovanov provides a complex $C_n \in K^b(\TL)$ that gives a resolution of 
$\JW_n \in K^b(\Kar(\TL))$.

In the language of the present paper, 
Khovanov's colored invariant $\Kh(\mathcal{K};n)$ is defined by extending 
$\Kh_{\mathcal{K}} \colon \cdTL \to \Vect^{\Z\times \Z}$ 
to a functor between homotopy categories of complexes
\begin{equation}
	\label{eq:cablecx}
\Kh_{\mathcal{K}} \colon K^b(\cdTL) \to K^b(\Vect^{\Z\times \Z})
\end{equation}
and setting $\Kh(\mathcal{K};n) := H^{\bullet} \big( \Kh_{\mathcal{K}}(C_n) \big)$, 
where here $H^{\bullet} (-)$ denotes taking homology.
Since $\JW_n \simeq C_n$ in $K^b(\cdTL)$ and $\Kh_{\mathcal{K}}$ is additive, 
it is then immediate that
\[
\Kh(\mathcal{K};n) = H^{\bullet} \big( \Kh_{\mathcal{K}}(C_n) \big) 
\cong H^{\bullet} \big( \Kh_{\mathcal{K}}(\JW_n) \big) = \Kh_{\mathcal{K}}(\JW_n) = \Kh(\mathcal{K}^{\JW_n}) \, .
\]

Lastly, note that $\Kh(\mathcal{K}^n)^{\mathfrak{S}_n}$ can be identified with
the object $\big(\Kh(\mathcal{K}^n) , Q_n \big) \in \Kar(\Vect^{\Z \times \Z})
\cong \Vect^{\Z \times \Z}$, where $Q_n = \frac{1}{n!} \sum_{w \in
\mathfrak{S}_n} w$ is the symmetrizing idempotent for the action of
$\mathfrak{S}_n$ on $\Kh(\mathcal{K}^n)$ by braiding cobordisms. The
aforementioned results in \cite{GLW} imply\footnote{Their results are stated in
the setting of annular Khovanov homology, but their proof holds in the
non-annular setting. See \cite[Proposition 3.6]{2020arXiv200908520M} 
and also the discussion in \cite[\S6.1]{2019arXiv190404481G}.} 
that $Q_n = \Kh_{\mathcal{K}}(\JW_n)$, thus
\[
\Kh(\mathcal{K}^n)^{\mathfrak{S}_n} = \big(\Kh(\mathcal{K}^n) , Q_n \big)
=\big(\Kh_{\mathcal{K}}(\gen^n)) , \Kh_{\mathcal{K}}(\JW_n) \big) = \Kh_{\mathcal{K}}\big( (\gen^n , \JW_n) \big)
= \Kh(\mathcal{K}^{\JW_n}) \, . \qedhere
\]
\end{proof}

\begin{remark}\label{rem:dgtrace} Note that the cabling operation from
	\eqref{eq:cablecx} in the proof of Proposition~\ref{prop:sym colored}
	assembles Khovanov \emph{homologies} into a new complex. Although not needed
	here, a more sophistic cabling operation on the level of Khovanov chain
	complexes can be realized in the form of a dg (or $A_\infty$) functor from
	the ``derived horizontal trace'' \cite{2002.06110} of the Bar-Natan
	bicategory to complexes of $\Z$-graded vector spaces.  We expect that all
	the constructions in this paper should have analogues in the setting of
	derived traces, but we do not investigate them here.
\end{remark}

For brevity, 
we will often denote the colored knot $\mathcal{K}^{\symobj_n}$ by $\mathcal{K}^{(n)}$. 

\begin{example}\label{ex:kirby colored}
The invariant $\Kh(\mathcal{K}^{\kirby_i})$ of a knot colored by a Kirby object
is computed as
\[
\Kh(\mathcal{K}^{\kirby_i}) = \colim_{n\in \N} q^{-i-2n} \Kh \big(\mathcal{K}^{(i+2n)} \big)
\]
where the maps in the colimit are induced by dotted annulus cobordisms. 
As in Example \ref{ex:eval is pol}, we can identify $\Kh(\mathcal{U}^{\kirby_i})$
with $\Pol(\kirby_i)$. 
This agrees with the $N=2$ case of \cite[Theorem 1.2]{2020arXiv200908520M}, 
which computes Manolescu--Neithalath's \emph{cabled Khovanov--Rozansky homology} 
of the $0$-framed unknot (in any given homology class). 
A generalization is given below in Theorem~\ref{thm:MN}.

If instead we color by the total Kirby color $\kirby = \kirby_0\oplus \kirby_1$, 
then $\Kh(\mathcal{K}^\kirby)$ is computed as
\[
\Kh(\mathcal{K}^{\kirby}) = \colim_{m\in \N} q^{-m} \Kh \big(\mathcal{K}^{(m)} \big)
\]
where for the purposes of this colimit we must use the 
nonstandard partial order on $\N$
given by $m\unlhd m'$ iff $m\leq m'$ and $m\equiv m' \text{ mod } 2$.
(This colimit isn't filtered, but the result is a direct sum of two filtered colimits.)
\end{example}

\subsection{Kirby colored Khovanov homology and skein modules}
\label{ss:skeins}
To prepare for our main theorem on Kirby colored Khovanov homology, 
let $\mathcal{L}\subset S^3$ be a framed oriented  link 
with a decomposition into sublinks $\mathcal{L} = \mathcal{L}_1\cup \mathcal{L}_2$.  
Denote by $B^4(\mathcal{L}_2)$ the $4$-manifold obtained by attaching $2$-handles 
to the $4$-ball $B^4$ along $\mathcal{L}_2 \subset S^3 = \partial B^4$.  
We regard $\mathcal{L}_1$ as a link in the boundary $3$-manifold
$S^3(\mathcal{L}_2):=\partial B^4(\mathcal{L}_2)$.

Write $\mathcal{L}_2$ in terms of its components as 
$\mathcal{L}_2 = \mathcal{K}_1\cup\cdots\cup \mathcal{K}_r$. 
For $\underline{i}= (i_1,\ldots,i_r) \in \{0,1\}^r$, 
we let $\mathcal{L}_1\cup \mathcal{L}_2^{\kirby_{\underline{i}}}$ denote the colored link 
wherein the components of $\mathcal{L}_1$ are viewed as 
colored by the generating object $\gen$ of $\dTL$ and 
$\mathcal{L}_2^{\kirby_{\underline{i}}} = \mathcal{K}_1^{\kirby_{i_1}} \cup\cdots\cup \mathcal{K}_r^{\kirby_{i_r}}$.

\begin{thm}
	\label{thm:MN}
There is an isomorphism
\[
\Kh(\mathcal{L}_1\cup \mathcal{L}_2^{\kirby_{\underline{i}}})\cong 
\mathcal{S}^2_0(B^4(\mathcal{L}_2);\mathcal{L}_1,\underline{i},\K)
\]
of $\Z\times \Z$-graded vector spaces, 
where the right-hand side denotes the $\glnn{2}$ skein lasagna module 
(i.e.~degree zero blob homology) 
as defined in \cite{2019arXiv190712194M}, 
and $\underline{i}$ is interpreted as a class in the second relative homology group 
$H_2(B^4(\mathcal{L}_2),\mathcal{L}_1;\Z)$.
\end{thm}
\begin{proof}[Proof (sketch)]
The result is a straightforward consequence of the 
Manolescu--Neithalath $2$-handle formula \cite{2020arXiv200908520M} for 
the $\glnn{N}$ skein lasagna module $\mathcal{S}^N_0$ 
in its slightly generalized form from \cite{2206.04616}, together with some
optimizations that are currently only possible in the $N=2$ case.  

Observe that the Kirby-colored Khovanov homology 
$\Kh(\mathcal{L}_1\cup \mathcal{L}_2^{\kirby_{\underline{i}}})$ 
can be computed as in Example \ref{ex:kirby colored}:
\begin{equation}\label{eq:kirby kh}
\Kh(\mathcal{L}_1\cup \mathcal{L}_2^{\kirby_{\underline{i}}}) \cong
\colim_{\underline{n}\in \N^r} \Kh\left(\mathcal{L}_1\cup 
	\mathcal{K}_1^{(i_1+2n_1)}\cup \cdots \cup \mathcal{K}_r^{(i_r+n_r)}\right) \, .
\end{equation}
Next, for the purposes of stating the 2-handle formula, let us temporarily use the notation 
$\mathcal{K}^{m+\overline{n}}$ to denote the $(m+n)$-cable of a framed oriented knot
in which $m$ strands have the orientation inherited from $\mathcal{K}$ and $n$ have the opposite orientation.  
The product of braid groups $\Br_m\times \Br_n$ acts on $\mathcal{K}^{m+\overline{n}}$ 
by cobordisms which braid similarly-oriented components.  

The $2$-handle formula expresses the skein lasagna module 
$\mathcal{S}^2_0(B^4(\mathcal{L}_2);\mathcal{L}_1,\underline{i},\K)$ as
\begin{equation}\label{eq:big direct sum}
\left(\bigoplus_{\underline{n}\in \N^r} q^{-|\underline{i}|-2|\underline{n}|}  
\Kh(\mathcal{L}_1\cup \mathcal{K}_1^{i_1+n_1+\overline{n_1}}\cup 
	\cdots \cup \mathcal{K}_r^{i_r+n_r+\overline{n_r}})\right) \bigg/ \sim
\end{equation}
where $|\underline{i}| = i_1+\cdots+i_r$ and $|\underline{n}| = n_1+\cdots+n_r$.
The relations imposed on this direct sum are as follows:
\begin{enumerate}
\item Partially symmetrize, 
i.e.~take coinvariants with respect to the parabolic subgroup $\mathfrak{S}_{i_j+n_j}\times \mathfrak{S}_{n_j}$ 
acting via cobordisms braiding the similarly-oriented components of 
$\mathcal{K}_j^{i_j+n_j+\overline{n_j}}$.
\item Kill the action of the annulus cobordisms.
\item Identify along the images of dotted annulus maps.
\end{enumerate}
Taken together, (1) and (2) accomplish the full symmetrization on each component 
with respect to $\mathfrak{S}_{i_j+2n_j}$.
The direct sum with these relations is hence precisely the colimit calculated in \eqref{eq:kirby kh}.  
\end{proof}

\begin{remark}
Both of the invariants $\Kh(\mathcal{L}_1\cup \mathcal{L}_2^{\kirby_{\underline{i}}})$ 
and $\mathcal{S}^2_0(B^4(\mathcal{L}_2);\mathcal{L}_1,\underline{i},\K)$ 
are defined for arbitrary $\underline{i}\in \Z^r$.  
The $2$-periodicity of Kirby objects $\kirby_i\cong \kirby_{i+2}$ 
ensures that the former depends only on the parity of the components of $\underline{i}$. 
On the other hand, the same is only true of the skein lasagna invariants
$\mathcal{S}^2_0(B^4(\mathcal{L}_2);\mathcal{L}_1,\underline{i},\K)$ up to overall grading shift, 
because the components of $\mathcal{K}^{i+n+\overline{n}}$ 
can be given an alternating orientation only if $i\in \{0,1\}$. 
This disagreement between Kirby-colored Khovanov homology 
and skein lasagna modules would disappear had we worked with 
$\glnn{2}$ foams instead of the Bar-Natan category.
\end{remark}

\section{Future directions} 
To conclude, 
we list and comment on questions related to
the Kirby color for Khovanov homology that we find interesting, 
but which exceed the scope of this paper. \medskip

\subsection{Kirby color in other constructions}
	\label{ss:other+glN}
Khovanov homology admits many different constructions 
using technology from different areas of mathematics, 
including combinatorics, geometric representation theory, and symplectic geometry. 
We expect that the Kirby color $\kirby$ has an avatar in each of these constructions. 
It is an interesting problem to identify the Kirby color intrinsically 
in each of these settings.

Similarly, we expect that Kirby colors also exist for other link homology theories, 
provided they extend to functorial tangle invariants.
Key examples are the $\glN$ link homology theories 
originally constructed by Khovanov--Rozansky~\cite{KR}.
In this setting, 
the role of the dotted Temperley--Lieb category is played by the dotted $\glN$ web
	category $\dWeb_N$. 
Isomorphism classes of objects in this category are
	indexed by pairs of natural numbers $(k,\ell)\in \Z_{\geq 0}^2$, which are
	represented by $k$ upward-oriented boundary points next to $\ell$
	downward-oriented boundary points. 
Morphisms in $\dWeb_N$ are $\K$-linear combinations of dotted trivalent graphs, 
and all preserve the difference $k-\ell$.
The relation $x^2=0$ in $\dTL$ generalizes to the relation $x^N=0$ in $\dWeb_N$.
	
In the same way as $\dTL$ describes the annular Bar-Natan category $\ABN$, 
the dotted web category $\dWeb_N$ describes 
the category of annular $\glN$ foams (after appropriate completion).
A version of the latter (its positive half) appears in \cite{QR2}, 
and it is possible to adapt the arguments therein 
to the entire category of annular $\glN$ foams, 
defined by feeding the canopolis of $\glN$ foams from \cite{ETW} 
into the machinery from \cite[\S  3.1]{1806.03416}.
This yields the equivalence between $\dWeb_N$ and annular $\glN$ foams.
Under this equivalence, the image of the standard object $(k,\ell) \in \dWeb_N$ is a collection of
concentric essential circles in the annulus of which 
$k$ have the standard orientation and $\ell$ have the opposite orientation. 
Therefore, the total winding number of this object around the core of the annulus is $k-\ell$.

The analogue of Jones-Wenzl projectors in $\dWeb_N$ are idempotent 
morphisms $\JW_{k,\ell}$, one for each object $(k,\ell)$. 
For any winding number $m \in \Z$, 
we expect that one can assemble the morphisms $\JW_{m+n,n}$ for $n\geq \max(0,-m)$ 
into directed systems, with transition maps involving $(N-1)$-fold dotted cup morphisms.
We conjecture that the ind-objects $\kirby_m$ giving the colimits of these directed systems 
play an analogous role to the Kirby colors developed in the present paper for $\glN$ link homology.
In the $\glN$ setting, directed systems of different winding numbers are 
not isomorphic: there are no nonzero morphisms between them.


\subsection{Representing other modules for the annular category}
As discussed in Remark~\ref{rem:adjoints}, 
the Kirby color is a representing
object for the (dual of) $\Pol$, thus in light of Example \ref{ex:eval is pol}
also for \emph{planar evaluation} of $\cABN$. A natural follow-up is to
determine whether other evaluation functors $\rho\colon \cABN\to \Vect^\Gamma$
can be represented by objects in $\cABN$. Interesting examples of such Rhos
include the following: 
\begin{itemize}
	\item Annular Khovanov homology \cite{MR2113902} factors through a functor
	$\cABN \to  \Vect^{\Z\times \Z}$, 
	which involves an additional \emph{annular grading}. 
	In this case, one can ask for an object representing each annular degree.
	
	\item By Theorem~\ref{thm:cable functor}, any framed knot $\mathcal{K}
	\subset S^3$ determines a functor $\Kh_\mathcal{K}\colon \cABN \to
	\Vect^{\Z\times \Z}$. If it exists, a representing object for
	$\kirby_\mathcal{K}$ in $\cABN$ would capture the data of all colored
	Khovanov homologies of $\mathcal{K}$, the Khovanov homologies of all parallel cables
	of $\mathcal{K}$, as well as all linear maps associated by Khovanov homology
	to a certain class of link cobordisms
	which are supported over a tubular neighborhood of $\mathcal{K}$. If
	$\kirby_\mathcal{K}$ could be expressed in terms of
	$\kirby=:\kirby_\mathcal{U}$ and compact objects of $\cABN$, this would
	amount to a quantification of the growth behavior \cite{Wed3} of the colored
	Khovanov homologies of $\mathcal{K}$ relative to those for the unknot
	$\mathcal{U}$. We speculate that this may provide an approach to formalizing
	a notion of \emph{q-holonomicity of colored Khovanov homology}, in
	categorical analogy with the q-holonomicity of the colored Jones polynomial
	\cite{MR2174266}. 
\end{itemize}

\subsection{2-handles in the Asaeda--Frohman TQFT} 

Our primary interest in the
	Kirby color $\kirby$ is due to its relevance for computing $4$-manifold
	invariants associated with Khovanov homology. 
	However, $\kirby$ may also play a role in the $(3+\epsilon)$-dimensional TQFT associated 
	with the Bar-Natan monoidal bicategory $\BN$, 
	whose associated 3-manifold invariants are the Bar-Natan skein modules from
	\cite{MR2370224,Russell,MR2503518,thesis-Fadali}. 
	Specifically, we expect that $\kirby \in \cABN$ models $3$-dimensional 
	$2$-handle attachments in this theory.
	This parallels the fact that the Kirby color for $2$-handle attachments in
	the 4-dimensional Crane--Yetter-type TQFT based on a braided fusion category
	$\CS$ also serves as the Kirby color in the 3-dimensional Turaev--Viro-type
	TQFT associated with the fusion category $\CS$ 
	(with the braiding forgotten). 
	
Our results may therefore be of particular interest in
	the program to categorify the Kauffman bracket skein
	modules of $3$-manifolds presented via Heegaard decompositions.
	This approach first seeks to categorify the Kauffman bracket
	skein algebra of the Heegaard surface in the form of a monoidal category, 
	and then to take a categorified version of a $2$-sided quotient by $2$-handle relations.
This program was initiated in \cite{1806.03416} with the construction of such a
	categorified skein algebra which 
	was partially based on conjectural functoriality properties of Khovanov homology under foams. 
The requisite functoriality properties have recently been addressed in \cite{2209.08794}. 
	It is unclear whether the resulting (candidate)
	categorified skein algebra is compatible with the Kirby color
	as a model for $2$-handles, since the construction in \cite{1806.03416} requires a certain
	truncation by higher degree morphisms, including the dotted annulus maps. 
	The latter, of course, are essential in the definition of the Kirby color, 
	so it is an interesting problem to reconcile this divide.

\subsection{Kirby color via coends:} 
In the context of 3-manifold invariants associated with (not necessarily semisimple) 
	ribbon categories $\CS$, a useful method for 
	encoding cabling procedures is in terms of the coend of the functor 
	$\CS^{\mathrm{op}}\otimes \CS \to \CS$, $(X,Y) \mapsto X^*\otimes Y$. 
If it exists, the coend $A \in \CS$ represents 
	the relative skein module of $\CS$-labelled ribbon graphs in the thickened annulus 
	with boundary points in a specified disk on the boundary, 
	viewed as a $\CS$-module via the action at this boundary disk.
A Kirby element can then be defined \cite{MR2251160} as an element of 
	$\Hom_{\CS}(\oone,A)$ such that the associated cabled link invariants 
	are constant under the second Kirby move. Since morphisms are taken here from 
	the monoidal unit $\oone \in \CS$, this models the non-relative skein of the 
	thickened annulus.
We speculate that an analogous framework can be provided in the categorified
	case by considering a suitable completion of the Bar-Natan monoidal
	bicategory $\BN$.

\subsection{Kirby color in positive characteristic} 

Throughout the present paper, we have worked in characteristic zero. In
characteristic $p>0$, the description of symmetric objects $\Sym^{k}(\gen)$ as
images of symmetrizers on $\gen^k$ as in Definition \ref{def:JW} no longer holds
when $k\geq p$. Thus, it is unclear if a directed system like $\kirby$ can be
constructed in this case. However, the category of tilting modules for
$\mathrm{SL}(2)$, which is modeled by the Temperley--Lieb category, admits
interesting tensor ideals in the modular case, reminiscent of (but richer than)
quantum groups at a root of unity. It is an interesting question whether
Khovanov homology in positive characteristic admits a ``smaller'' Kirby object,
roughly in parallel with the Kirby element for the Jones polynomial at a root of
unity being a finite sum.

\subsection{Connections with Category $\mathcal{O}$} 
	
As discussed in Remarks \ref{rem:Verma} and \ref{rem:Verma2}, 
	for $k \geq 0$ there are strong parallels between the shifted Kirby objects $q^k \kirby_k$
	and the $\slnn{2}$ dual Verma modules $\nabla(k)$ of highest weight $k$.
To recap:
\begin{itemize}
	\item $\Pol(q^k \kirby_k)$ and $\nabla(k)$ have the same graded dimension.
	\item There is a degree zero monic morphism $\JW_k \hookrightarrow q^k \kirby_k$ in $\cdTL$ 
	akin to the inclusion of $\slnn{2}$-modules $\Sym^k(V) \hookrightarrow \nabla(k)$.
	\item $\Hom_{\cdTL}(q^k \kirby_k,X)=0$ for any compact $X \in \cdTL$ and 
	$\Hom_{\Rep(\slnn{2})}(\nabla(k),W)=0$ for any finite-dimensional $W \in \Rep(\slnn{2})$.
\end{itemize}
It would be interesting to understand the precise relation between $\cdTL$ 
and the BGG category $\mathcal{O}(\slnn{2})$, 
and to extend this relation to the $\glN$ setting discussed in \S \ref{ss:other+glN}.

 \subsection{A homotopy colimit} Instead of modelling the Kirby color as a filtered colimit, we can also
 build a similar object as a homotopy colimit of the directed system
 \eqref{eq:dirsys}.  
  For $k\geq 0$ this is the 2-term complex $\Kirby_k$:
 	\begin{equation}
 		\label{eq:Kirbycx}    
 		\begin{tikzcd}
 			q^{-k}\, \JW_{k}  \arrow{r}{\id}\arrow{dr}{U_k} &    q^{-k}\, \JW_{k} \\
 			q^{-k-2}\, \JW_{k+2}\arrow{r}{\id}\arrow{dr}{U_{k+2}}  & q^{-k-2}\, \JW_{k+2} \\
 			\quad\quad\quad \arrow[dotted]{r}&      q^{-k-4} \JW_{k+4}
 		\end{tikzcd} 
 	\end{equation}
 	obtained by taking coproducts in the columns, with the right-hand column in
 	homological degree zero. The complex $\Kirby_k$ can be considered as an
 	object of the dg category $\Ch^b_{\mathrm{dg}}(\Kar(\dTL)^{\coprod})$ of
 	bounded chain complexes over $\Kar(\dTL)^{\coprod}$, where the superscript
 	$\coprod$ refers to completion with respect to countable coproducts.
 	Gaussian elimination immediately implies $\Kirby_k\simeq \Kirby_{k+2}$, the
 	analogue of Lemma~\ref{lem:shift}. We expect that $\Kirby_k$ and $\kirby_k$
 	are quasi-isomorphic when considered as objects of a suitably defined
 	derived category of $\dTL$.

 	One advantage of $\Kirby_k$ is that for $k\in\{0,1\}$ it is manifestly filtered, with
 	subquotients of the form 
 	\[
 	 C_{l}:= \left(q^{-l+2}\, \JW_{l-2} \xrightarrow{U_{l}} q^{-l}\,
 	 \JW_{l}\right)\] for $l\geq 0$, where we declare $\JW_{l}=0$ for $l<0$.
 	Following an observation of Elijah Bodish, these complexes have the property
 	 that all morphisms $C_{l} \to C_{l'}$ are null-homotopic when $l\leq
 	 l'$, and thus may be part of a highest weight structure. We plan to
 	 investigate these structures in future work.

\bibliographystyle{alpha}
\bibliography{pw}

\begin{thebibliography}{BHKW19}

\bibitem[AF07]{MR2370224}
Marta Asaeda and Charles Frohman.
\newblock A note on the {B}ar-{N}atan skein module.
\newblock {\em Internat. J. Math.}, 18(10):1225--1243, 2007.
\newblock \mathscinet{MR2370224} \doi{10.1142/S0129167X07004497}
  \arxiv{math/0602262}.

\bibitem[APS04]{MR2113902}
Marta~M. Asaeda, J\'{o}zef~H. Przytycki, and Adam~S. Sikora.
\newblock Categorification of the {K}auffman bracket skein module of
  {$I$}-bundles over surfaces.
\newblock {\em Algebr. Geom. Topol.}, 4:1177--1210, 2004.
\newblock \mathscinet{MR2113902} \doi{10.2140/agt.2004.4.1177}
  \arxiv{math/0409414}.

\bibitem[BHKW19]{2019arXiv190312194B}
Anna {Beliakova}, Matthew {Hogancamp}, Krzysztof {Karol Putyra}, and
  Stephan~Martin {Wehrli}.
\newblock {On the functoriality of sl(2) tangle homology}, 2019.
\newblock \arxiv{1903.12194}.

\bibitem[BHLW17]{BHLW}
A.~{Beliakova}, K.~{Habiro}, A.~D. {Lauda}, and B.~{Webster}.
\newblock {Current algebras and categorified quantum groups}.
\newblock {\em J. London Math. Soc.}, December 2017.

\bibitem[Bla03]{MR2009998}
Christian Blanchet.
\newblock Introduction to quantum invariants of 3-manifolds, topological
  quantum field theories and modular categories.
\newblock In {\em Geometric and topological methods for quantum field theory
  ({V}illa de {L}eyva, 2001)}, pages 228--264. World Sci. Publ., River Edge,
  NJ, 2003.
\newblock \mathscinet{MR2009998} \doi{10.1142/9789812705068\_0004}.

\bibitem[Bla10]{Bla}
Christian Blanchet.
\newblock An oriented model for {K}hovanov homology.
\newblock {\em J. Knot Theory Ramifications}, 19(2):291--312, 2010.
\newblock \arxiv{1405.7246}.

\bibitem[BN05]{BN2}
Dror Bar-Natan.
\newblock Khovanov's homology for tangles and cobordisms.
\newblock {\em Geom. Topol.}, 9:1443--1499, 2005.
\newblock \arxiv{math.GT/0410495}.

\bibitem[BS18]{BrSt}
Jonathan Brundan and Catharina Stroppel.
\newblock Semi-infinite highest weight categories, 2018.
\newblock \arxiv{1808.08022}.

\bibitem[CS98]{CSbook}
J.S. Carter and M.~Saito.
\newblock {\em Knotted surfaces and their diagrams}, volume~55 of {\em
  Mathematical Surveys and Monographs}.
\newblock American Mathematical Society, Providence, RI, 1998.

\bibitem[Day70]{Day}
Brian Day.
\newblock On closed categories of functors.
\newblock In {\em {Reports of the Midwest Category Seminar IV}}, volume 137 of
  {\em Lecture Notes in Mathematics}. Springer-Verlag, 1970.

\bibitem[EGNO15]{MR3242743}
Pavel Etingof, Shlomo Gelaki, Dmitri Nikshych, and Victor Ostrik.
\newblock {\em Tensor categories}, volume 205 of {\em Mathematical Surveys and
  Monographs}.
\newblock American Mathematical Society, Providence, RI, 2015.
\newblock \mathscinet{MR3242743}
  \url{http://www-math.mit.edu/~etingof/egnobookfinal.pdf}.

\bibitem[ETW18]{ETW}
Michael Ehrig, Daniel Tubbenhauer, and Paul Wedrich.
\newblock Functoriality of colored link homologies.
\newblock {\em Proc. Lond. Math. Soc. (3)}, 117(5):996--1040, 2018.
\newblock \arxiv{1703.06691}.

\bibitem[Fad16]{thesis-Fadali}
Lyla Fadali.
\newblock {\em Bar-Natan Skein Modules in Black and White}.
\newblock PhD thesis, UC San Diego, 2016.

\bibitem[GHW21]{2002.06110}
Eugene Gorsky, Matthew Hogancamp, and Paul Wedrich.
\newblock Derived traces of {S}oergel categories.
\newblock {\em Int. Math. Res. Not. IMRN}, 2021.
\newblock \doi{10.1093/imrn/rnab019}, \arxiv{2002.06110}.

\bibitem[GL05]{MR2174266}
Stavros Garoufalidis and Thang T.~Q. L\^{e}.
\newblock The colored {J}ones function is {$q$}-holonomic.
\newblock {\em Geom. Topol.}, 9:1253--1293, 2005.

\bibitem[GLW18]{GLW}
J.~E. {Grigsby}, A.~{Licata}, and S.~M. {Wehrli}.
\newblock Annular {K}hovanov homology and knotted {S}chur-{W}eyl
  representations.
\newblock {\em Compos. Math.}, 2018.
\newblock \href{http://arxiv.org/abs/1505.04386}{arXiv:1505.04386}.

\bibitem[GW19]{2019arXiv190404481G}
Eugene {Gorsky} and Paul {Wedrich}.
\newblock {Evaluations of annular Khovanov--Rozansky homology}, 2019.
\newblock \arxiv{1904.04481}.

\bibitem[Kai09]{MR2503518}
Uwe Kaiser.
\newblock Frobenius algebras and skein modules of surfaces in 3-manifolds.
\newblock In {\em Algebraic topology---old and new}, volume~85 of {\em Banach
  Center Publ.}, pages 59--81. Polish Acad. Sci. Inst. Math., Warsaw, 2009.

\bibitem[Kho00]{Kho}
M.~Khovanov.
\newblock A categorification of the {J}ones polynomial.
\newblock {\em Duke Math. J.}, 101(3):359--426, 2000.
\newblock \arxiv{math.QA/9908171}.

\bibitem[Kho05]{Kho4}
M.~Khovanov.
\newblock Categorifications of the colored {J}ones polynomial.
\newblock {\em J. Knot Theory Ramifications}, 14(1):111--130, 2005.

\bibitem[Kir78]{MR0467753}
Robion Kirby.
\newblock A calculus for framed links in {$S\sp{3}$}.
\newblock {\em Invent. Math.}, 45(1):35--56, 1978.
\newblock \mathscinet{MR0467753} \doi{10.1007/BF01406222}.

\bibitem[KL94]{MR1280463}
Louis~H. Kauffman and S{\'o}stenes~L. Lins.
\newblock {\em Temperley-{L}ieb recoupling theory and invariants of
  {$3$}-manifolds}, volume 134 of {\em Annals of Mathematics Studies}.
\newblock Princeton University Press, Princeton, NJ, 1994.
\newblock \mathscinet{MR1280463}.

\bibitem[KR08]{KR}
Mikhail Khovanov and Lev Rozansky.
\newblock Matrix factorizations and link homology.
\newblock {\em Fund. Math.}, 199(1):1--91, 2008.
\newblock \arxiv{math.QA/0401268}.

\bibitem[KS06]{KaSch}
M.~Kashiwara and P.~Schapira.
\newblock {\em Categories and Sheaves}.
\newblock Springer, 2006.

\bibitem[Lic92]{MR1185809}
W.~B.~R. Lickorish.
\newblock Calculations with the {T}emperley-{L}ieb algebra.
\newblock {\em Comment. Math. Helv.}, 67(4):571--591, 1992.
\newblock \mathscinet{MR1185809} \doi{10.1007/BF02566519}.

\bibitem[LPRH21]{MR4187261}
Diego Lobos, David Plaza, and Steen Ryom-Hansen.
\newblock The nil-blob algebra: an incarnation of type {$\tilde{A_1}$}
  {S}oergel calculus and of the truncated blob algebra.
\newblock {\em J. Algebra}, 570:297--365, 2021.

\bibitem[ML98]{MacLane}
Saunders Mac~Lane.
\newblock {\em {Categories for the working mathematician}}.
\newblock {Springer-Verlag, New York}, {1998}.

\bibitem[MN22]{2020arXiv200908520M}
Ciprian Manolescu and Ikshu Neithalath.
\newblock Skein lasagna modules for 2-handlebodies.
\newblock {\em J. Reine Angew. Math.}, 788:37--76, 2022.
\newblock \mathscinet{MR4445546} \doi{10.1515/crelle-2022-0021}
  \arxiv{2009.08520}.

\bibitem[MWW19]{2019arXiv190712194M}
Scott {Morrison}, Kevin {Walker}, and Paul {Wedrich}.
\newblock {Invariants of 4-manifolds from Khovanov-Rozansky link homology},
  2019.
\newblock \arxiv{1907.12194}, to appear in \textit{Geom. Topol.}

\bibitem[MWW22]{2206.04616}
Ciprian Manolescu, Kevin Walker, and Paul Wedrich.
\newblock Skein lasagna modules and handle decompositions, 2022.
\newblock \arxiv{2206.04616}.

\bibitem[QR16]{QR}
H.~Queffelec and D.E.V. Rose.
\newblock The $\mathfrak{sl}_n$ foam $2$-category: a combinatorial formulation
  of {K}hovanov--{R}ozansky homology via categorical skew {H}owe duality.
\newblock {\em Adv. Math.}, 302:1251--1339, 2016.
\newblock \arxiv{1405.5920}.

\bibitem[QR18]{QR2}
Hoel Queffelec and David Rose.
\newblock Sutured annular {K}hovanov-{R}ozansky homology.
\newblock {\em Transactions of the American Mathematical Society},
  370(2):1285--1319, 2018.
\newblock \arxiv{1506.08188}.

\bibitem[Que22]{2209.08794}
Hoel Queffelec.
\newblock Gl2 foam functoriality and skein positivity, 2022.
\newblock \arxiv{2209.08794}.

\bibitem[QW21]{1806.03416}
Hoel {Queffelec} and Paul {Wedrich}.
\newblock {Khovanov homology and categorification of skein modules}.
\newblock {\em Quantum Topol.}, 12(1):129--209, 2021.
\newblock \mathscinet{MR4233203} \doi{10.4171/QT/148} \arxiv{1806.03416}.

\bibitem[RT91]{MR1091619}
Nicolai Reshetikhin and Vladimir~G. Turaev.
\newblock Invariants of {$3$}-manifolds via link polynomials and quantum
  groups.
\newblock {\em Invent. Math.}, 103(3):547--597, 1991.
\newblock \mathscinet{MR1091619} \euclid{euclid.cmp/1104180037}.

\bibitem[Rus09]{Russell}
Heather Russell.
\newblock The {B}ar-{N}atan skein module of the solid torus and the homology of
  $(n,n)$ {S}pringer varieties.
\newblock {\em Geom. Dedicata}, 2009.

\bibitem[STWZ21]{sutton2021sl2}
Louise Sutton, Daniel Tubbenhauer, Paul Wedrich, and Jieru Zhu.
\newblock Sl2 tilting modules in the mixed case, 2021.
\newblock \arxiv{2105.07724}.

\bibitem[TW21]{2019arXiv190711560T}
Daniel Tubbenhauer and Paul Wedrich.
\newblock Quivers for {$\rm SL_2$} tilting modules.
\newblock {\em Represent. Theory}, 25:440--480, 2021.
\newblock \mathscinet{MR4273168} \doi{10.1090/ert/569} \arxiv{1907.11560}.

\bibitem[Vir06]{MR2251160}
Alexis Virelizier.
\newblock Kirby elements and quantum invariants.
\newblock {\em Proc. London Math. Soc. (3)}, 93(2):474--514, 2006.

\bibitem[Wed19]{Wed3}
Paul Wedrich.
\newblock Exponential growth of colored {HOMFLY}-{PT} homology.
\newblock {\em Adv. Math.}, 353:471--525, 2019.
\newblock \arxiv{1602.02769}.

\bibitem[Wit89]{MR990772}
Edward Witten.
\newblock Quantum field theory and the {J}ones polynomial.
\newblock {\em Comm. Math. Phys.}, 121(3):351--399, 1989.
\newblock \mathscinet{MR990772} \euclid{euclid.cmp/1104178138}.

\end{thebibliography}

\end{document}